\documentclass[12pt]{amsart}
\usepackage[asymmetric, hmarginratio=2:3]{geometry}
\usepackage[colorlinks=true, citecolor=blue, linkcolor=blue, urlcolor=blue]{hyperref}
\usepackage{amssymb}
\usepackage{comment}

\def\p{\partial}
\DeclareMathOperator{\supp}{supp}

\def\C{\mathbb C}
\def\R{\mathbb R}

\def\N{\mathbb N}

\usepackage{mathtools}

\DeclarePairedDelimiter\norm{\lVert}{\rVert}

\newtheorem{theorem}{Theorem}[section]
\newtheorem{lemma}[theorem]{Lemma}
\newtheorem{proposition}[theorem]{Proposition}
\newtheorem{corollary}[theorem]{Corollary}
\theoremstyle{definition}

\theoremstyle{remark}
\newtheorem{remark}[theorem]{Remark}

\newcommand{\abs}[1]{\lvert #1 \rvert}
\newcommand{\tbl}[1]{\textcolor{blue}{#1}}
\newcommand{\eps}{\epsilon}

\newcommand{\tr}{\text{Tr}\s}
\newcommand{\s}{\hspace{0.5pt}}
\newcommand{\ol}{\overline}
\newcommand{\op}{\overline \p}
\newcommand{\ccdot}{\,\cdot\,}
\newcommand{\z}{\overline{z}}

\newcommand{\OO}{\mathcal{O}}
\renewcommand{\Im}{\text{Im}}
\DeclareMathOperator{\Op}{{\rm Op}}

\newcommand{\pl}{\partial}
\newcommand{\rr}{\mathbb{R}}

\title[The Calder\'on problem on Riemannian surfaces and minimal surfaces]{The Calder\'on problem on Riemannian surfaces and of minimal surfaces}
\author[C. Cârstea]{Cătălin I. Cârstea}
\address{National Yang Ming Chiao Tung University, Hsinchu, Taiwan}
\email{catalin.carstea@gmail.com}
\author[T. Liimatainen]{Tony Liimatainen}
\address{Department of Mathematics and Statistics, University of Helsinki, Helsinki, Finland}
\curraddr{}
\email{tony.liimatainen@helsinki.fi}
\author[L. Tzou]{Leo Tzou}
\address{University of Amsterdam, Amsterdam, Netherlands}
\email{leo.tzou@gmail.com}

\begin{document}
\begin{abstract}
In this paper we prove two results. The first shows that the Dirichlet-Neumann map of the operator $\Delta_g+q$ on a Riemannian surface can determine its topological, differential, and metric structure. Earlier work of this type assumes a priori that the surface is a planar domain \cite{imanuvilov2012partial} or that the geometry is a priori known \cite{guillarmou2011calderon}. We will then apply this result to study a geometric inverse problem for determining minimal surfaces embedded in $3$-dimensional Riemannian manifolds. In particular we will show that knowledge of the volumes of embedded minimal surfaces determine not only their topological and differential structure but also their Riemannian structure as an embedded hypersurface. Such geometric inverse problems are partially inspired by the physical models proposed by the AdS/CFT correspondence.

The crucial ingredient in removing the planar domain assumption is the determination of the boundary trace of holomorphic functions from knowledge of the Dirichlet-Neumann map of $\Delta_g +q$. This requires a new type of argument involving Carleman estimates and construction of CGO whose phase functions are not Morse as in the case of \cite{guillarmou2011calderon}. We anticipate that these techniques could be of use for studying other inverse problems in geometry and PDE.

\end{abstract}
\maketitle
\tableofcontents

\section{Introduction}
Geometric inverse problems for minimal surfaces not only form natural extensions of boundary rigidity problems but provide a path for understanding  the AdS/CFT correspondence. Earlier results on this direction \cite{nurminen1, nurminen2, Tracey, carstea2022inverse} often require geometrical and/or topological assumptions which may not be sutiable for all the physical models under consideration.

In this paper we develop a set of techniqus towards a broader scope of geometrical/topological settings by proving two results. The first of them solves the Calder\'on problem for the Schr\"odinger equation $(\Delta_g+q)v=0$ on general Riemannian surfaces, where all the involved quantities are unknown, including the topological and differential structure of the underlying surface. The second result uses this progress in the Calder\'on problem to solve the inverse problem for the minimal surface equation in dimension $2$ in full generality. The connection between these two problems is that the linearization of the minimal surface equation is the Schr\"odinger equation. 

\subsection{The Calder\'on problem on Riemannian surfaces}
The Calder\'on problem for the Schr\"odinger equation has its origins in the studies by Calder\'on and Gelfand \cite{calderon2006inverse,gel1954some}. They asked if it is possible 
to detect the potential of the Schr\"odinger equation from the Dirichlet-to-Neumann map (DN map) of the equation in a domain of the Euclidean space.  
The DN map is the mathematical formulation of the voltage to current measurements on the boundary of a medium.

We study the Calder\'on problem for the Schr\"odinger equation
\begin{equation}\label{eq:Schrodinger_intro}
 (\Delta_g+q)v=0 \text{ on } \Sigma
\end{equation}
on a general Riemannian surface $\Sigma$ with boundary. In this inverse problem $\Sigma$, its Riemannian metric $g$ and the potential $q$ are all unknown.  For the equation \eqref{eq:Schrodinger_intro} we define the DN map $\Lambda_{g,q}: C^\infty(\p \Sigma)\to C^\infty(\p \Sigma)$ as usual by 
\[
\Lambda_{g,q}(f)=\p_\nu v_f,
\]
where $v_f$ solves \eqref{eq:Schrodinger_intro} with boundary value $f\in C^\infty(\p \Sigma)$. Here $\p_\nu$ is the outer normal vector field on the boundary $\p\Sigma$ of $\Sigma$. 
One can easily see that the coefficients $(g,q)$ and $(c\s F^*g,c^{-1}F^*q)$ give the same DN map, assuming $F$ is a diffeomorphism of $\Sigma$, $c$ is positive function and $F|_{\p M}=\mathrm{Id}$ and $c|_{\p M}=1$. Here $F^*$ is the pullback by $F$. Thus the Calder\'on problem we study has a gauge symmetry.  
We prove the following uniqueness result. 
\begin{theorem}[Uniqueness for the general Calder\'on problem in 2D]\label{thm:unique2d}
 Let $(M_1,g_1)$ and $(M_2,g_2)$ be compact connected $C^\infty$ Riemannian surfaces with mutual boundary $\partial M$, and let also $q_1\in C^\infty(M_1)$ and $q_2\in C^\infty(M_2)$. Assume that $\Lambda_{g_1,q_1}=\Lambda_{g_2,q_2}$. 
 Then there is a conformal diffeomorphism $J:M_1\to M_2$ such that 
 \[
 g_1=\lambda J^*g_2 \ \text{ and } \ q_1=\lambda^{-1}J^*q_2.
 \]
 Here $\lambda$ is a smooth positive function in $M_1$ with $\lambda|_{\p M}=1$, and $J|_{\p M}=\mathrm{Id}$. Note that we do not a priori assume that $M_1$ and $M_2$ have the same topological structure.
\end{theorem}
Due to the above mentioned gauge symmetry, Theorem \ref{thm:unique2d} is the best possible recovery result. Since there are no additional assumptions in the theorem, the result is also general. That there are no restrictions on the topologies of the Riemannian surfaces in the theorem will be crucial when we apply it to study an inverse problem for minimal surfaces. This is because even in $\R^3$ minimal surfaces can have complicated topologies (see e.g. \cite{colding2011course} for pictures of minimal surfaces). 

\subsubsection{Earlier results}
Let us mention some earlier results on the Calder\'on problem for the Schr\"odinger equation $(\Delta_g+q)v=0$ on $\Omega$. Here $\Omega$ is a domain or a manifold of dimension $n\geq 2$, $g$ a Riemannian metric and $q$ a function. The most complete known results are when $\Omega$ is a domain in $\R^n$, already known, and $g$ is the Euclidean metric. In this case, the identification of the potential $q$ has been originally proved in dimension $n>2$ by Sylvester-Uhlmann \cite{sylvester1987global} and Novikov \cite{Nov}, 
and  by Bukgheim \cite{bukhgeim2008recovering} in dimension $2$ for simply connected domains.

A related question is the conductivity problem 
which consists in taking $q=0$ and replacing $\Delta_g$ by $-{\rm div}\sigma\nabla$, where $\sigma$
is a positive definite symmetric matrix field. In this case, recovering a sufficiently smooth isotropic conductivity (that is $\sigma$ is a multiple of the identity matrix) is contained in the above problem of recovering of recovering a potential $q$ on a domain in $\R^n$. For domains of $\rr^2$, Nachman \cite{nachman1988reconstructions} used the $\bar{\pl}$ techniques to show 
that the Cauchy data space determines the conductivity $\sigma$ and gave a reconstruction method.  Astala and P\"aiv\"arinta in \cite{astala2006calderon} improved this result to assuming that $\sigma$ is only a $L^\infty$ scalar function on a simply connected domain. 

The partial data version is of the Calder\'on problem is that the sources are supported and the DN map is restricted to subsets of the boundary. In dimension $2$, the partial data version has been solved in the following cases when the domain is known. The potential $q$ was recovered by Imanuvilov-Uhlmann-Yamamoto in \cite{imanuvilov2010calderon}, when the domain is a Riemannian surface by Guillarmou-Tzou in \cite{guillarmou2011calderon}, and in the work \cite{imanuvilov2012partial} by Imanuvilov-Uhlmann-Yamamoto, both the metric $g$ and the potential $q$ (as well as first order terms) were recovered up to a gauge. In dimensions $n>2$, the partial data problem was solved by Kenig-Sj\"ostrand-Uhlmann in \cite{partialdataCalderon} under conditions on the supports of the data.

There are only a few results  about  recovering the Riemannian manifold $(\Omega,g)$. In the case $q=0$, a Riemannian surface was recover by Lassas-Uhlmann \cite{lassas2001determining} (see also \cite{Belishev03,henkin2008inverse}). The real-analytic case for $n>2$ 
was solved by Lee-Uhlmann \cite{lee1989determining} and Lassas-Taylor-Uhlmann \cite{lassas2003dirichlet}. The case of manifolds admitting limiting Carleman weights and in a same conformal class was considered by Dos Santos Ferreira-Kenig-Salo-Uhlmann \cite{dos2009limiting}, and generalized by Dos Santos Ferreira-Kurylev-Lassas-Salo in \cite{ferreira2013calderon}. 

In Feizmohammadi-Oksanen \cite{FO19} and Lassas-Liimatainen-Lin-Salo \cite{LLLS2019inverse} a version of the Calder\'on problem for nonlinear elliptic equations was studied. Especially, the work \cite{LLLS2019inverse} recovered a Riemannian surface $(\Omega,g)$ and $q$ from the DN map of the equation $\Delta_gu+qu^m=0$, $m\geq 2$, on a Riemannian surface up to a gauge symmetry.

Our Theorem \ref{thm:unique2d} improves earlier results (in the full data case) as follows. It generalizes the result by Guillarmou-Tzou \cite{guillarmou2011calderon} to the case where also the Riemannian surface and its metric are unknown. It also generalized the results by Imanuvilov-Uhlmann-Yamamoto \cite{imanuvilov2010calderon} and \cite{imanuvilov2012partial} to Riemannian surfaces and to unknown domains. The result \cite{lassas2001determining} by Lassas-Uhlmann is generalized by allowing a non-zero potential. It also proves the missing case $m=1$ in \cite{LLLS2019inverse} by Lassas-Liimatainen-Lin-Salo. 

We chose not to consider partial data in this work, or add first order terms to the Schr\"odinger equation, to focus on presenting the new ideas and methods. We however expect the proof of Theorem \ref{thm:unique2d} to extend to more general situations.

\subsubsection{Sketch of the proof of Theorem \ref{thm:unique2d}}
Let us explain how we prove Theorem \ref{thm:unique2d}. The proof is based on recovering the boundary values of holomorphic functions of a Riemannian surface from the DN map of \eqref{eq:Schrodinger_intro}, and then constructing the Riemannian surface and its metric from its holomorphic functions up to conformal mapping. The proof is finalized by referring to the main result of \cite{guillarmou2011calderon} after a conformal transformation.

The main step of the proof Theorem \ref{thm:unique2d} is how to recover the boundary values of holomorphic functions. We explain now the morale how we do it. We extend the Riemannian surface to a slightly larger surface. On the extension, the solutions to the Schr\"odinger equation are then determined by the DN map of $(\Delta_g+q)v=0$ on the original surface. We then derive a new Carleman estimate with boundary terms and without assuming the weight function is Morse in Proposition \ref{prop: carleman estimate}. Using the Carleman estimate, we show that a complex geometrics optics solution for Schr\"odinger equation \eqref{eq:Schrodinger_intro} of the form 
\begin{equation}\label{eq:phase_recovery_CGO}
e^{\Phi/h}(a+r) 
\end{equation}
is determined on the extension up to a term that grows polynomially as $h\to 0$. Here $\Phi$ is holomorphic, not necessarily Morse, and $a$ is a holomorphic amplitude. 

We modify the amplitude $a$ in Proposition \ref{prop: CGO for boundary determination} so that $r$ decays pointwise (in $C^1$) as $h\to 0$. Combining the above, we know $a$ in the extended domain, that $r$ decays there pointwise and that \eqref{eq:phase_recovery_CGO} is known up to a polynomially growing term. We may thus pick out pointwise values of $\Phi$ on the extended domain since $e^{\Phi/h}$ grows there exponentially, by changing the sign of $\Phi$ if necessary. By characterizing the boundary values of holomorphic functions in terms of the DN map of the Laplacian $\Delta_g$, and noticing that the above procedure holds for general holomorphic functions, we recover boundary values of \emph{all} holomorphic functions on the Riemannian surface. This is Proposition \ref{prop: same holomorphic bv}. 

We construct the Riemannian surface and its metric $g$ up to a conformal transformation from the boundary values of its holomorphic functions. The conformal invariance of the Laplacian in dimension $2$ lets us to transform the remaining problem of finding the potential $q$ to the situation of \cite{guillarmou2011calderon}. This completes the proof of Theorem \ref{thm:unique2d}.

\subsection{Inverse problems for minimal surfaces}
The linearization of the minimal surface equation is the Schr\"odinger equation on a Riemannian surface with possibly non-trivial topology. Having solved the inverse problem for the Sch\"rodinger equation in Theorem \ref{thm:unique2d}, we are ready to study inverse problems for general minimal surfaces. 

We consider an inverse problem for $2$-dimensional minimal surfaces embedded in a $3$-dimensional Riemannian manifold $(N,\overline g)$ with boundary. In this paper, we consider minimal surfaces in variational sense and define that
a $C^2$-smooth surface in  $(N,\overline g)$ is a minimal surface if its mean curvature is zero.  Let $\Sigma$ be a minimal surface. We consider minimal surfaces that are perturbations of $\Sigma$ given as graph of functions over $\Sigma$ and write the minimal surface equation in Fermi-coordinates associated to $\Sigma$. In these coordinates the metric of $(N,\overline g)$ has the form 
\begin{equation}\label{prod_metric}
 \overline g=ds^2 +g_{ab}(x,s)dx^adx^b.
\end{equation}
Here $g_s:=g(\ccdot,s)$ is a $1$-parameter family of Riemannian metrics on $\Sigma$. Here also $s$ belongs to an open interval $I\subset \R$ and $(x^a)$ are coordinates on $\Sigma$.
For the parts of this paper that hold in general dimensions we write $\dim(N)=n+1$ and $\dim(\Sigma)=n$, $n \geq 2$. We remark that Fermi-coordinates always exist and thus viewing the minimal surface equation in these coordinates is not a restriction of generality.

Let 
\[
 F(x)=(u(x),x)
\]
be the graph of a function $u:\Sigma\to \R$. 
Then the graph $F$ is a minimal surface if in Fermi-coordinates the function $u$ satisfies
\begin{equation}\label{eq:minimal_surface_general}
 -\frac{1}{\det(g_u)^{1/2}}\nabla\cdot \left( g_u^{-1}\frac{\det(g_u)^{1/2}}{\sqrt{1+\abs{\nabla u}^2_{g_u}}} \right)\nabla u + f(u,\nabla u) =0 \quad  \hbox{ on }\Sigma,
\end{equation}
where 
\[
 f(u,\nabla u)=\frac{1}{2}\frac{1}{(1+\abs{\nabla u}^2_{g_u})^{1/2}}(\p_sg_u^{-1})(\nabla u,\nabla u)+\frac{1}{2}(1+\abs{\nabla u}^2_{g_u})^{1/2}\text{Tr}(g_u^{-1}\p_sg_u).
\]
Here $\nabla$ refers to the $n$-dimensional $\R^n$ Euclidean gradient and $\ccdot$ is the Euclidean inner product on $\R^n$. Throughout this paper we denote 
\[
 g_u(x)=g_{u(x)}=g(x, u(x)).
\]
We refer to the equation \eqref{eq:minimal_surface_general} as the minimal surface equation. We note that if $g(x,s)$ is independent of $s$ we have $f(u,\nabla u)=0$. In this case  the minimal surface equation is the one for the product metric $e\oplus g$ studied recently in \cite{carstea2022inverse}. Here $e$ is the Euclidean metric on $\R$. We will derive the minimal surface equation \eqref{eq:minimal_surface_general} in Section \ref{Section 2}.

We define the associated \emph{Dirichlet-to-Neumann} map (DN map in short) for \eqref{eq:minimal_surface_general} by
\begin{equation}\label{eq:DN_map_for_nonlin}
\Lambda_{\overline g} f = \left. \p_\nu u \right|_{\p \Sigma} \text{ for }  \, f \in U_\delta.
\end{equation}
Here $\nu$  is the unit exterior normal vector of $\Sigma$ with respect to the metric $g$ at $(x,u(x))$,
that is, 
$$
g_{u}(x)(\nu,\nu)=1,\quad g_{u}(x)(\nu,\eta)=0,\quad \hbox{for }\nu\in T_x\partial \Sigma.
$$
In \eqref{eq:DN_map_for_nonlin}, $u$ is the unique small solution to the equation \eqref{eq:minimal_surface_general}  with boundary value $f\in C^{2,\alpha}(\p \Sigma)$ satisfying the conditions given in Proposition \ref{prop:local_well_posedness}, 
and $U_\delta=\{f\in  C^{2,\alpha}(\p \Sigma): \norm{f}_{C^{2,\alpha}(\p \Sigma)}\leq \delta\}$ for some $\delta>0$.  
We refer to Section \ref{Section 2} for details about the local well-posedness of \eqref{eq:minimal_surface_general} and  the DN map.

We present our main results about minimal surfaces next.
\begin{theorem}\label{thm:main}
Let $(\Sigma_1, g_1)$ and $(\Sigma_2,g_2)$ be Riemannian surfaces. Assume that there are $1$-parameter families of Riemannian metrics $g_\beta(\ccdot,s)$, $\beta=1,2$, on $\Sigma_\beta$ and that $\Sigma_\beta$ are minimal surfaces in the sense that $0$ is a solution to \eqref{eq:minimal_surface_general} on both $\Sigma_\beta$. We assume that $\p_s^k|_{s=0}g_1(\ccdot,s)=\p_s^k|_{s=0}g_2(\ccdot,s)$ to infinite order on $\p \Sigma$ for $k=0,1,2$.

Assume that the associated DN maps of \eqref{eq:minimal_surface_general} of  $(\Sigma_1,g_1)$ and $(\Sigma_2,g_2)$ satisfy
for some $\delta>0$ sufficiently small and for all $f\in U_\delta$ 
\[
\Lambda_{g_1}f = \Lambda_{g_2}f.
\]
Then there is an isometry $F:M_1\to M_2$, 
\[
  F^*g_2=g_1,
\]
which satisfies $F|_{\p \Sigma}=\textrm{Id}$. In addition, $F^*\eta_2=\eta_1$, where $\eta_\beta$ are the scalar second fundamental forms of $(\Sigma_\beta,g_\beta)$. {Note that a priori $\Sigma_1$ and $\Sigma_2$ don't need to have the same topology.}
\end{theorem}We note that the scalar second fundamental form $\eta_\beta(X,Y)=\langle \nabla_XN_\beta, Y\rangle_{\overline g_\beta}$ in the theorem depends also on the extrinsic geometry on $\Sigma_\beta$ in $(N_\beta, \overline g_\beta)$. Here $N_\beta$ is the unit normal to $\Sigma_\beta$, $X,Y\in T\Sigma_\beta$ and $\overline g_\beta$ as in \eqref{prod_metric}. The assumption that the unknown quantities in the theorem are known on the boundary, is just to avoid various standard-like boundary determination arguments. 
\begin{remark}
Our meaning that two tensor fields agree to infinite order on the boundary is the following: Let $\Theta_1$ and $\Theta_2$ be tensor fields (for example $\p_s|_{s=0}g_1(\ccdot,s)$ and $\p_s|_{s=0}g_2(\ccdot,s)$) on respective Riemannian surfaces $(\Sigma_1,g_1)$ and $(\Sigma_2,g_2)$, which have a mutual boundary $\p \Sigma$. We say that $\Theta_1=\Theta_2$ to infinite order on $\p\Sigma$ if the coordinate representation of $\Theta_1$ in $g_1$-boundary normal coordinates agrees with that of $\Theta_2$ in $g_2$-boundary normal coordinates to infinite order on $\p \Sigma$. In invariant terms this means that $\mathcal{L}_{\nu_1}^l\Theta_1|_{\p \Sigma}=\mathcal{L}_{\nu_2}^l\Theta_2|_{\p \Sigma}$, $l=0,1,\ldots$, where $\mathcal{L}$ is the Lie derivative and $\nu_\beta$ is the locally defined vector field dual to the gradient of the distance to $\p \Sigma$ in $(\Sigma_\beta,g_\beta)$.
\end{remark}

\subsubsection{The generalized boundary rigidity}
The generalized boundary rigidity problem for minimal surfaces, posed formally in \cite{Tracey}, asks if areas of minimal surfaces embedded in a Riemannian manifold with boundary determine the Riemannian manifold. If the minimal surfaces are $1$-dimensional,
and thus minimizing geodesics, then the problem is the standard boundary rigidity
problem. One of the motivations for the generalized boundary rigidity problem is in
the AdS/CFT correspondence in physics, which we discuss in Section \ref{sec:adscft} below.

Let us relate  our results to the generalized boundary rigidity
problem. We assume that our minimal surfaces are given as solutions to the Dirichlet
problem for the minimal surface equation \eqref{eq:minimal_surface_general} in Fermi-coordinates relative to a minimal surface $\Sigma$. In this situation,
volumes of minimal surfaces determine the DN map of the minimal surface equation,
see Lemma \ref{lem:minimal surfaces and the DN map}. In general, however, the part $\p \Sigma\times I$, $I\subset \R$, of the boundary of Fermi-coordinates might not
belong to the boundary of the manifold $N$, where a minimal surface is embedded. This poses a technical issue on how to pose the Dirichlet problem for \eqref{eq:minimal_surface_general}. For this reason, we study an exterior problem, where $N$ is assumed to belong to a slightly larger manifold $\widetilde N$. 

Let $(\Sigma,g)$ be a minimal surface embedded in $N$. \emph{The exterior problem} asks the following: If $\p \Sigma$ and the volumes of minimal surfaces $\Sigma'$ in the exterior manifold $\widetilde N$, whose boundaries satisfy $\p \Sigma'\subset \widetilde N\setminus N$, are known, do they determine the isometry type of $(\Sigma,g)$?  For the exterior problem we obtain as a consequence Theorem \ref{thm:main} the following.

%
%

\begin{corollary}[Uniqueness result for the exterior problem]\label{cor:exterior_problem}
Let $(\widetilde N,\overline g)$ and $(N, \overline g|_N)$ be $3$-dimensional Riemannian manifolds with boundaries such that $N \subset\subset \widetilde N$.
Let $(\Sigma,g)$ be a $2$-dimensional minimal surface and assume it extends properly to a minimal surface $\widetilde \Sigma \subset \widetilde N$. 
Let $\mathcal S$ be the set of all smooth minimal surface deformations in a small neighourhood of $\widetilde \Sigma$.

Then, the knowledge of $\widetilde N\setminus N$, $\overline g|_{\widetilde N\setminus N}$, $\widetilde \Sigma\setminus \Sigma$, and the relation 
$\left(\partial \widetilde \Sigma',
\hbox{Vol}\, (\widetilde \Sigma')\right)$ for all $\widetilde \Sigma' \in \mathcal S$ 
determines 
  $(\Sigma,g)$ up to a boundary preserving isometry. The isometry also preserves the first fundamental form. Note that a priori we do not assume knowledge of the topological structure of the embedded minimal surface $\Sigma\subset N$.
\end{corollary}

We refer to Section \ref{sec:areas_exterior_problem}
for the technical details how the exterior problem is posed. 
In the corollary, $\mathcal S$ includes the set of deformations given by Proposition \ref{prop:local_well_posedness} obtained by solving the boundary problem for \eqref{eq:minimal_surface_general}.
The corollary says that the knowledge of areas of minimal surfaces determine
them up to isometries. Consequently, Proposition \ref{prop:local_well_posedness} this is a result solving a part of the generalized
boundary rigidity problem. In the generalized
boundary rigidity problem these $2$-dimensional minimal 
surfaces still need to be glued together to construct the whole $3$-dimensional manifold, where the minimal surfaces are embedded.  We study how to do the gluing
in a future work.

\subsubsection{Sketch of the proof of Theorem \ref{thm:main}.}
Let us discuss how we determine a minimal surface $(\Sigma,g)$ from the DN map of \eqref{eq:minimal_surface_general} in the proof of Theorem \ref{thm:main}. One of the main parts of the proof regards analyzing products of numerous correction terms of complex geometrics optics solutions (CGOs). 

The strategy of the proof is based on the higher order linearization method that originates from \cite{kurylev2018inverse}. For this, let us consider $f_j\in C^{2,\alpha}(\p \Sigma)$, $j=1,2,3,4$ and $\alpha>0$. Let us denote by $u=u_{\eps_1f_1+ \cdots + \eps_4f_4}$ the solution of~\eqref{eq:minimal_surface_general} with boundary data $\eps_1f_1+ \cdots +\eps_4f_4$, where $\eps_j>0$ are sufficiently small parameters. We  write $\eps=0$ when referring to $\eps_1 = \cdots= \eps_4=0$. 

By taking the derivative $\p_{\eps_j}|_{\eps=0}$ of the solution $u_{\eps_1f_1+ \cdots + \eps_4f_4}$,  
we see that the function
\[
 v^j:=\frac{\p}{\p \eps_j}\Big|_{\eps=0}\s\s u_{\eps_1f_1+ \cdots + \eps_4f_4}
\]
solves the first linearized equation
\[
 \Delta_gv+qv=0,
\]
where 
\begin{align*}
 q(x)&=\frac{1}{2}\frac{d}{d s}\Big|_{s=0}\text{Tr}(g_s^{-1}\p_sg_s).
\end{align*}
One can check that the first linearized equation is the usual stability equation \cite{colding2011course} for the minimal surfaces.  
Since we know the DN map of \eqref{eq:minimal_surface_general}, we know the DN map of the first linearized equation (see Section \ref{Section 2}). The first linearized equation has a gauge symmetry in the sense that the coefficients $(g,q)$ and $(cg,c^{-1}q)$ give the same DN map. The DN map is also invariant under diffeomorphisms that are identical on the boundary. We determine $(\Sigma, g)$ up to conformal mapping by using the main result of \cite{imanuvilov2012partial}.  The rest of the proof is about determining the conformal factor and the scalar second fundamental form.

We proceed to higher order linearizations. Let us denote by
\[
 \eta(X,Y)=\langle \nabla_XN,Y\rangle_{\overline g}
\]
the scalar second fundamental form. It is a tensor operating on vector fields $X,Y$ tangential to $\Sigma$. The function $w^{jk}:=\frac{\p^2}{\p \eps_j\p \eps_k}\big|_{\eps=0}\s\s u_{\eps_1f_1+ \cdots + \eps_4f_4}$ satisfies the second linearized equation 
\[
 (\Delta_g+q)w^{jk}= \text{terms of the form } \eta(\nabla v^j, \nabla v^k)+ \text{lower order terms}. 
 \]
Lower order terms are terms that contain at most one gradient of a linearized solution $v^j$. The first aim of the second order linearization is to determine $\eta$, which is a $2\times 2$-matrix valued function. Since we know the DN map of second linearization, it follows that the integral 
\begin{multline}\label{eq:known_integral_2}
\int_{\Sigma} v^1 \eta(\nabla v^2,\nabla v^3\big)dV+\int_{\Sigma}v^2 \eta(\nabla v^1,\nabla v^3\big)dV +\int_{\Sigma} v^3 \eta(\nabla v^1,\nabla v^2\big)dV \\
+ \text{lower order and boundary terms}
 \end{multline}
is known. (We remark that in the proof we use slightly different notation where instead of $\eta$ we have $\p_sg_s^{-1}|_{s=0}$.) We disregard lower order and boundary terms from the following argumentation.

To determine the matrix function $\eta$, we use CGOs of the form in \cite{guillarmou2011identification} as the linearized solutions $v^j$. These are solutions to $\Delta_gv+qv=0$ given by the ansatz
\[
 e^{\Phi/h}(a+r_h),
\]
where $\Phi=\phi+i\psi$ is a holomorphic Morse function, $a$ is a holomorphic function and $r_h$ is a correction term given by 
\[
r_h=-  \overline{\p}_\psi^{-1}\sum_{j=0}^\infty T_h^j\op_\psi^{*-1}(qa).
\]
Here $\overline \p_\psi^{\s -1}$ is defined (modulo localization) by $\overline \p_\psi^{\s-1}f =\overline \p^{-1}(e^{-2i\psi/h}f)$, where $\overline \p^{-1}$ is the Cauchy-Riemann operator that solves $\overline \p^{-1}\op=\text{Id}$. 
We note that the dependence of $r_h$ in $h$ is quite complicated and especially is not given as a power series in the small parameter $h$. This leads to complicated error analysis when the correction terms are multiplied in the integrals, such as \eqref{eq:known_integral_2}, that we encounter.

An improvement to earlier literature here is that we notice that it is possible to use CGOs with holomorphic phases, which are not just Morse, but also do not have any critical points. We will use phases $\Phi$, which in local holomorphic coordinates have the expansion
\[
 \Phi(z)=\pm z+cz^2+O(\abs{z}^3), \quad c\neq 0.
\]
Note that in \eqref{eq:known_integral_2}, it is possible to use two CGOs whose phases do not have critical points and one CGO with an antiholomorphic Morse phase to have that the sum of the three phases has the expansion
\[
 z^2-\overline z^2+O(\abs{z}^3).
\]

The main benefit of using a holomorphic phase without critical points in a CGO is that it produces a smaller correction term $r_h$. Indeed, while a general holomorphic Morse function as the phase function produces a correction term $r_h$ of size $O_{L^2}(h^{1/2+\eps})$,  a holomorphic phase without critical points leads to improved estimates 
\[
 \|r_h\|_p + \|d r_h\|_p \leq Ch
\]
for all $p\in (1,\infty)$. We refer to details about $r_h$ to  
Section \ref{sec:CGOs}.
%
%
%
The improved estimates for the correction terms greatly simplify the analysis concerning multiplication of CGOs. 
We expect the use of holomorphic phases without critical points to have other applications in inverse problems for nonlinear equations in dimension $2$. 

After recovering the scalar second fundamental form and other unknowns appearing in the second linearization (up to the conformal mapping), we proceed to consider third linearization of the minimal surface equation. The third linearized equation is quite complicated and we only refer to Lemma \ref{Lem:Integral identity_3rd} for it here. By using stationary phase and improved estimates for the correction terms of CGOs, we are ultimately able to recover the conformal factor from the third linearization. This finishes the proof of Theorem \ref{thm:main}.

\subsubsection{Earlier literature}
The literature on inverse problems for nonlinear partial differential equations is extensive. Without doing a full review of the state of the field, we mention here some early and recent works in order to place our paper in historical context.
There is also a relation between our work and the AdS/CFT duality in physics. References to AdS/CFT duality are given in the next subsection below. 

Common techniques for addressing inverse problems for nonlinear equations have been linearization methods. There one starts by considering data of the form $\epsilon f$, where $\eps$ is small and $f$ is a source or a boundary value. Then one differentiates in $\epsilon$ the solutions corresponding to the data to obtain separate information on the linear and nonlinear coefficients. This approach dates back to a work of Isakov and Sylvester in \cite{isakov1994global}, where the authors considered the equation
\[
-\Delta u +F(x,u)=0 
\]
on a Euclidean domain of dimension greater than or equal to three. Subject to certain constraints, they prove uniqueness for the non-linear functions $F(x,u)$. In dimension two, a similar problem was solved by Isakov and Nachman in \cite{victorN}.
Other important early works are \cite{sun2004inverse, sun2010inverse} for semilinear elliptic equations, and \cite{sun1996,sun1997inverse} for quasilinear elliptic equations. In the present work we will use some results from \cite{lassas2018poisson}, in which the inverse problem for general quasilinear equations on Riemannian manifolds is studied.

For nonlinear hyperbolic equations, nonlinearity was found out to be a helpful feature in solving inverse problems in 
the work of Kurylev, Lassas and Uhlmann \cite{kurylev2018inverse}. The work studied scalar wave equation with a quadratic nonlinearity. By using the nonlinearity, they were able to prove that local measurements determine a globally hyperbolic $4$-dimensional Lorentzian manifold up to a conformal transformation. The inverse problem for the corresponding linear wave equation is open. 

The work \cite{kurylev2018inverse} can also be said to have introduced what is now usually called the \emph{the higher order linearization method} in the study of inverse problems for nonlinear differential equations. As the name implies, this method consists in taking multiple derivatives of the nonlinear equation in question, and the associated DN map, with respect to small parameters in the data, in order to inductively prove uniqueness for various coefficients. The works \cite{FO19,LLLS2019inverse, feizmohammadi2023inverse} introduced the higher order linearization method for the study of inverse problems of semilinear elliptic equations on $\R^n$ and Riemannian manifolds.

Among the many works that employ the higher order linearization method we mention  \cite{MR4052205, LLLS2021b}, where the partial data problem for semilinear elliptic equations is addressed, \cite{CFKKU, kian2020partial, CaNaVa, Carstea2020, CaKa, CaFe1, CaFe2, CaGhNa, CaGhUh}, which deal with inverse problems for quasilinear elliptic equations, and \cite{lai2020partial, liimatainen2022inverse, liimatainen2022uniqueness, harrach2022simultaneous, salo2022inverse} on various other topics in the study of inverse problems for nonlinear elliptic equations. We mention that in \cite{liimatainen2022uniqueness}, the zero function is not necessarily a solution, as is the case for the minimal surface equation \eqref{eq:minimal_surface_general} unless $\Sigma$ itself is a minimal surface.

\subsubsection{Inverse problems for minimal surfaces}
Let us then mention works on inverse problems for minimal surfaces. To the best of our knowledge the first work on the subject is \cite{Tracey}. It studied inverse problem closely related to the one we study in this paper. There it was proven that if a $3$-dimensional Riemannian manifold is topologically a ball and satisfies certain curvature and foliation assumptions, the areas of a sufficiently large class of $2$-dimensional minimal surfaces determine the $3$-dimensional Riemannian manifold. Especially they determined embedded minimal surfaces that are topologically $2$-dimensional disks from the areas. As explained before Corollary \ref{cor:exterior_problem}, the areas of minimal surfaces determine the DN map of the minimal surface equation \eqref{eq:minimal_surface_general}. Consequently, the relation between our work and \cite{Tracey} is that we determine more general embedded minimal surfaces, instead of those satisfying the assumptions in \cite{Tracey}, from the knowledge of the areas of the minimal surfaces.

The papers \cite{nurminen1} and \cite{nurminen2} consider an inverse problem for the minimal surface equation, for hypersurfaces in a manifold of dimension $n=3$ or higher. By denoting $e$ the metric of $\R^n$, the metric in these works is assumed to have the  form 
$c(x)e$ in \cite{nurminen1} and $c(x)(\hat g\oplus e)$ in \cite{nurminen2}, where $\hat g$ is a simple Riemannian metric  (see \cite{nurminen2} for the definition). The quantity being determined in each of these two papers is the conformal factor $c$.

In \cite{carstea2022inverse}, an inverse problem for the minimal surface equation on Riemannian manifolds of the form $\Sigma\times\R$ was studied. Here $\Sigma$ is a smooth compact two dimensional Riemannian manifold with a metric $g$, and the metric $\overline g$ of $\Sigma\times\R$ is of the product form $\overline g = ds^2+g_{ab}(x)dx^adx^b$. 
It was shown in \cite{carstea2022inverse} that the DN map of the minimal surface equation determines $\Sigma$ up to an isometry. 
The current paper thus also generalizes  \cite{carstea2022inverse}.

\subsubsection{Relation to the AdS/CFT correspondence}\label{sec:adscft}
One of the motivations of this work is its quite direct relevance to the AdS/CFT correspondence, or duality, in physics. This duality is a conjectured relationship between two kinds of physical theories proposed by Maldacena \cite{maldacena1999large}. According to the duality, there is a correspondence between physics of a conformal field theory (CFT) and geometric properties of an Anti-de Sitter spacetime (AdS in short). The AdS is referred to as the \emph{bulk} and it has an asymptotic infinity considered to be its boundary. The CFT is assumed to live on the boundary. A remarkable feature of this duality is that it has been successfully used to reduce various complicated quantum mechanical calculations in CFT to easier differential geometrical problems in the bulk.  The mechanism of the duality is not clear in general, see e.g. \cite{van2009comments} for a discussion and examples of the duality.

A recently proposed mechanism by Ryu and Takayanagi \cite{PhysRevLett.96.181602, ryu2006aspects} for the correspondence is the equivalence between the entanglement entropies of a CFT and areas of minimal surfaces in an AdS. Entanglement entropy is roughly speaking the experienced entropy (i.e. state of disorder) of a physical system for an observer who has only access to a subregion of a larger space. 
In their proposal, the accessible subregion $A$ is a subset of the asymptotic infinity of the AdS.  The set $A$ determines a minimal surface $\Sigma$ in the AdS by the assignment that $\Sigma$ is the minimal surface anchored on the asymptotic infinity to the boundary of $A$. By the proposal, the entanglement entropy $S_A$ of a given set $A$ in the CFT is equal to $(4G)^{-1}\text{Vol}(\Sigma)$. Here $G$ is Newton's constant and $\text{Vol}(\Sigma)$ is the volume of $\Sigma$.

Ryu and Takayanagi were able to confirm their duality, $S_A=(4G)^{-1}\text{Vol}(\Sigma)$, in several nontrivial cases. Later on, the duality has been applied to extract the geometry of bulk manifolds from the knowledge of entanglement entropies of families of sets $A$, or equivalently volumes of minimal surfaces $\Sigma$. Recall that volumes of minimal surfaces determine the DN map of the minimal surface equation \eqref{eq:minimal_surface_general} as explained before Corollary \ref{cor:exterior_problem}. Consequently, we can consider this paper to study the extraction or reconstruction of the bulk geometry in the AdS/CFT duality by using the DN map.

In the physics literature, the reconstruction of the metric of the bulk from entanglement entropies has been consider for infinite strips, circular disks and ellipsoids $A$ in \cite{bilson2008extracting, bilson2011extracting,fonda2015shape, hubeny2012extremal, jokela2021towards} and annulus shaped sets $A$ in \cite{jokela2019notes}. These works assume that the bulks have strong symmetries and are asymptotically AdS. 
The work \cite{cao2020building} considers the linearized version of the problem of reconstructing the bulk and its numerics. We also mention here \cite{hubeny2012extremal} that computes second order linearizations of the minimal surface equation for a certain physical model.

In the mathematics literature, the extraction process was studied under various assumptions in the already discussed work \cite{Tracey}. In the recent physics paper \cite{bao2019towards}, the authors argue how the bulk can be reconstructed from volumes of minimal surfaces of codimension $\geq 2$ in general. 
With natural modifications, we expect that the current paper (dealing with codimension $1$) can be used to give rigorous justification to arguments in \cite{bao2019towards}. We have separated the error term analysis from the proof of Theorem \ref{thm:main} to make the proof more accessible to a broader readership. Lastly we mention the work \cite{alexakis2010renormalized} concerning the renormalized area of minimal surfaces embedded in hyperbolic $3$-manifolds.

More generally, the current paper introduces the recent higher order linearization method (discussed above) into the study of the AdS/CFT duality between entanglement entropies and minimal surfaces. 
It seems that the method has not been used in the physics literature in this context. The dual description of the method may provide new techniques for studying the CFT side of the duality.

\subsubsection{Organization of the paper}

In Section \ref{sec:calderon} we prove Theorem \ref{thm:unique2d}. In Sections \ref{sec:Carleman} and \ref{sec:CGOsCalderon} we prove a Carleman estimate and construct complex geometric optics solutions (CGOs) for the Calder\'on problem. In Section \ref{Section 2}, we derive equation the minimal surface equation \eqref{eq:minimal_surface_general} and state a well-posedness result for it. In Section \ref{ss2.4}, we linearize to the minimal surface equation to orders up to $3$ and in Section \ref{ss2.5} we derive the integral identities corresponding to the linearizations. In the first part of Section \ref{sec:CGOs}, based on \cite{guillarmou2011identification} we construct different CGOs, whose phase functions do not have critical points. We show that such solutions have better decay estimates for the remainder terms in the asymptotic parameter $h$.  In Section \ref{Section_4}, we plug in the CGO solutions constructed in the previous section into the integral identities derived in Section \ref{Section 2}. We then estimate the size of the resulting integrals containing remainders as $h\to0$. Finally, in Section \ref{sec:proof_of_main_thm} we use the leading order terms in our integral identities to prove Theorem \ref{thm:main}. In the Appendix we collect some computations and lemmas. 

\subsubsection{Acknowledgments}
The authors are grateful to Matti Lassas who was involved in an earlier iteration of this work, but graciously removed himself as an author from the current, much more general, version. The authors also wish to thank physicists Niko Jokela and Esko Keski-Vakkuri for helpful discussions of the relation of the work to the AdS/CFT correspondence and providing references. 

C.C. was supported by NSTC grant number 112-2115-M-A49-002. T.L. was partially supported by PDE-Inverse project of the European Research Council of the European Union,  and the grant 336786 of the Research Council of Finland. Views and opinions expressed are those of the authors only and do not necessarily reflect those of the European Union or the other funding organizations. Neither the European Union nor the other funding organizations can be held responsible for them.

\section{The general Calder\'on problem on Riemannian surfaces}\label{sec:calderon}
In this section, we prove Theorem \ref{thm:unique2d}, which states that the Dirichlet-to-Neumann map (DN map) of the equation
\begin{equation}\label{eq:2D_Schrodinger}
 (\Delta_g+q)v=0
\end{equation}
on a general Riemannian surface $(M,g)$ determines $(M,g)$ and $q$ up to (conformal) gauge transformation explained in the introduction. Theorem \ref{thm:unique2d} will be a fundamental part of our solution to the inverse problem for minimal surfaces. However, this section is written independently of the rest of the paper to make the results and techniques regarding the Calder\'on problem for \eqref{eq:2D_Schrodinger} to be more easily accessible to a broader readership. 

\subsection{A Carleman estimate with boundary terms}\label{sec:Carleman}
We start by proving a Carleman estimate. We note that unlike in many related Carleman estimates in dimensions $2$, we do not require the phase function to be Morse.

\begin{proposition}
\label{prop: carleman estimate}
Let $(M,g,\partial M)$ be a Riemann surface with boundary, $\varphi\in C^\infty(\overline M)$ be a non-constant harmonic function, and $q\in L^\infty(M)$. Then for all $v\in C^\infty(\overline M)$ we have the following Carleman estimate:{
\begin{eqnarray}\nonumber
\|v\|^2_{L^2(M)} &\leq& \|e^{-\varphi/h} (\Delta_g + q) e^{\varphi/h} v\|_{L^2(M)} \\&&+ h^{-3}\|v\|_{L^2(\partial M)}^2 + h^{-1}\|\partial_\nu v\|_{L^2(\partial M)}^2 +  h^{-1}\|\partial_\tau v\|_{L^2(\partial M)}^2
\end{eqnarray}}
\end{proposition}
We will also prove the following ``shifted'' Carleman estimates for compactly supported smooth functions:

\begin{proposition}
\label{prop: shifted carleman estimate}
Let $(M,g,\partial M)$ be a Riemann surface with boundary compactly contained $(\widetilde M,g,\partial \widetilde M)$, $\varphi\in C^\infty( \widetilde M)$ be a non-constant harmonic function, and $q\in C^\infty(\widetilde M)$. Then for all $k\in\mathbb N$ and $v\in C_c^\infty( \widetilde M)$ we have the following Carleman estimate:
\begin{eqnarray}\nonumber
\|v\|^2_{H^{-k}_{scl}(\widetilde M)} \leq \|e^{-\varphi/h} (\Delta_g + q) e^{\varphi/h} v\|_{H^{-k}_{scl}(\widetilde M)} 
\end{eqnarray}
\end{proposition}

We start by modifying the weight as follows:
if $\varphi_0 := \varphi: M \to \R$ is a non-constant real valued harmonic function with discrete critical points $\{p_1,\dots,p_K\}$ on $\overline M$, we let $\varphi_j:M \to \R$, $j = 1,\dots,K$, be real valued harmonic functions such that $p_j$ is not a critical point of $\varphi_j$, and define the convexified weight $\varphi_{\eps} := \varphi - \frac{h}{2\epsilon}(\sum_{j = 0}^K|\varphi_j|^2)$. Observe that we have for all $x\in M$
\begin{eqnarray}
\label{laplace of phieps}
\Delta_g \varphi_\epsilon = \frac{h}{\epsilon} \sum\limits_{j=0}^K |d\varphi_j|^2 \geq \frac{ch}{\epsilon}
\end{eqnarray}
for some constant $c>0$ and $\eps,h>0$.By the notation $0<h\ll \epsilon$, we mean that we will first fix $\epsilon>0$ small and consider the limiting behavior as $h\to 0$.

We will first prove a variant of Proposition \ref{prop: carleman estimate} for these modified weights and for functions which have particular support properties in local coordinate chart. To this end let $\Omega \subset  M$ be simply connected open subset with $\Gamma := \overline \Omega \cap \partial M \subsetneq \partial\Omega$. We may assume without loss of generality  that $\Omega\subset \mathbb R^2_+$ is a bounded open subset with smooth boundary and that $\Gamma \subset\subset \partial \Omega\cap\partial\mathbb R^2_+$. In these coordinates $\p_x=\p_\tau$ and $\p_y=\p_\nu$ on $\overline \Omega \cap \partial M$ (up to a possible conformal scaling).

\begin{lemma}
\label{lem: dbar estimate}
Let $\omega\in C^\infty(\overline \Omega)$ such that 
\[
\omega|\s _{\Gamma^c} = \partial_\nu \omega|\s _{\Gamma^c} = \partial_\nu^2 \omega|\s _{\Gamma^c} = \dots = 0
\]
then for all $\epsilon>h>0$ we have the following estimate
\[
\|e^{-\varphi_\epsilon/h} \bar\partial e^{\varphi_\epsilon/h}\omega\|_{L^2(\Omega)}^2 + h^{-1} \|\omega\|^2_{L^2(\partial\Omega)} \geq \frac{c}{\epsilon} \|\omega\|^2_{L^2(\Omega)},\]
where the constant $c>0$ is independent of $\epsilon>0$.
\end{lemma}
\proof
We compute 
\begin{eqnarray*}
\|e^{-\varphi_{\eps}/h}\bar{\partial}e^{\varphi_{\eps}/h}\omega\|^2 &=& 
\|(\partial_{x} + \frac{i\partial_y\varphi_\eps}{h})\omega + (i\partial_{y} + \frac{\partial_x\varphi_\eps}{h})\omega\|^2\\
&=&  \|(\partial_{x} + \frac{i\partial_y\varphi_\eps}{h})\omega\|^2 + \|(i\partial_{y} + \frac{\partial_x\varphi_\eps}{h})\omega\|^2\\
&+& \langle  (\partial_{x} + \frac{i\partial_y\varphi_\eps}{h})\omega, (i\partial_{y} + \frac{\partial_x\varphi_\eps}{h})\omega\rangle\\
&+& \langle (i\partial_{y} + \frac{\partial_x\varphi_\eps}{h})\omega, (\partial_{x} + \frac{i\partial_y\varphi_\eps}{h})\omega\rangle.
\end{eqnarray*}
We integrate by parts the last two terms and use the boundary condition of $\omega$ to obtain 
\[
\|e^{-\varphi_{\eps}/h}\bar{\partial}e^{\varphi_{\eps}/h}\omega\|^2 \geq \int_\Omega \frac{\Delta\varphi_\epsilon}{h} |\omega|^2 - \int_{\partial\Omega} \frac{\partial_\nu \varphi_\epsilon}{h}|\omega|^2 \geq \frac{c}{\epsilon} \|\omega\|^2_{L^2(\Omega)} - \int_{\partial\Omega} \frac{\partial_\nu \varphi_\epsilon}{h}|\omega|^2. 
\]
This concludes the proof.
\qed

Lemma \ref{lem: dbar estimate} leads to the following estimate:
\begin{lemma}
\label{lem laplace on domain}
For all $q\in L^\infty (\Omega)$  there is a constant $c>0$ such that for all $\epsilon>h>0$ sufficiently small, the estimate
\begin{multline*}
\| e^{-\varphi_\epsilon/h} (\Delta +q) e^{\varphi_\epsilon /h}u\|^2_{L^2(\Omega)} + h^{-3} \|u\|_{L^2(\partial\Omega)}^2 +h^{-1} \|\partial_\nu u\|^2_{L^2(\partial\Omega)} + h^{-1} \|\partial_\tau u\|^2_{L^2(\tbl{\p \Omega})}\\ 
\geq \frac{c}{\epsilon}\left(\frac{c}{\epsilon} \|u\|^2_{L^2(\Omega)} + \|du\|^2_{L^2(\Omega)} + h^{-2}\|d\varphi_\epsilon u\|_{L^2(\Omega)}^2\right).
\end{multline*}
holds for all $u\in C^\infty(\overline \Omega)$ such that 
$$
u|\s _{\Gamma^c} = \partial_\nu u|\s _{\Gamma^c} = \partial_\nu^2 u|\s _{\Gamma^c} = \dots = 0.
$$
\end{lemma}
\proof
It suffices to prove the claim in the special case when $q = 0$, $u$ real valued, and that the metric is Euclidean.  To do so we apply Lemma \ref{lem: dbar estimate} to $\omega := e^{-\varphi_\epsilon/h} \partial e^{\varphi_\epsilon/h}u$ to obtain
\begin{multline}\label{eq: intermediate estimate}
\|e^{-\varphi_\epsilon/h} \Delta e^{\varphi_\epsilon/h} u\|^2_{L^2(\Omega) }+ h^{-1}\| \frac{\partial_x \varphi_\epsilon}{h} u +\partial_x u\|^2_{L^2(\partial\Omega)}+ h^{-1} \|\frac{\partial_y\varphi_\epsilon}{h} u +\partial_y u\|^2_{L^2(\partial\Omega)}\\
\geq \frac{c}{\epsilon}\left( \|\frac{\partial_x \varphi_\epsilon}{h} u +\partial_x u\|^2_{L^2(\Omega)}  +  \|\frac{\partial_y\varphi_\epsilon}{h} u +\partial_y u\|^2_{L^2(\Omega)}\right).
\end{multline}

Now write
\begin{eqnarray*} \|\frac{\partial_x \varphi_\epsilon}{h} u +\partial_x u\|^2_{L^2(\Omega)} &=&h^{-2} \|\partial_x\varphi_\epsilon u\|^2_{L^2(\Omega)} + \|\partial_x u\|^2_{L^2(\Omega)} + 2h^{-1}\langle \partial_x\varphi_\epsilon u,\partial_x u  \rangle \\
&= &h^{-2} \|\partial_x\varphi_\epsilon u\|^2_{L^2(\Omega)} + \|\partial_x u\|^2_{L^2(\Omega)} - h^{-1}\langle \partial^2_x\varphi_\epsilon,|u|^2  \rangle,
\end{eqnarray*}
where in the last equality we integrated by parts and used the boundary condition we assumed for the function $u$. Similar calculation yields
\begin{multline*}
 \|\frac{\partial_y\varphi_\epsilon}{h} u +\partial_y u\|^2_{L^2(\Omega)} =h^{-2} \|\partial_y\varphi_\epsilon u\|^2_{L^2(\Omega)} + \|\partial_y u\|^2_{L^2(\Omega)} - h^{-1}\langle \partial^2_y\varphi_\epsilon,|u|^2  \rangle  \\
 +h^{-1} \int_{\partial\Omega} |u|^2 \partial_\nu\varphi_\epsilon .
\end{multline*}
By inserting these two identities into \eqref{eq: intermediate estimate}, we get
\begin{multline*}
\|e^{-\varphi_\epsilon/h} \Delta e^{\varphi_\epsilon/h} u\|^2_{L^2(\Omega) }+ h^{-1}\| \frac{\partial_x \varphi_\epsilon}{h} u +\partial_x u\|^2_{L^2(\partial\Omega)}+ h^{-1} \|\frac{\partial_y\varphi_\epsilon}{h} u +\partial_y u\|^2_{L^2(\partial\Omega)}\\
\geq \frac{c}{\epsilon}\left( h^{-2} \|d\varphi_\epsilon u\|^2_{L^2(\Omega)} + h^{-1}\langle \Delta\varphi_\epsilon,|u|^2  \rangle + \|du\|^2_{L^2(\Omega)} +h^{-1} \int_{\partial\Omega} |u|^2 \partial_\nu\varphi_\epsilon\right).
\end{multline*}
Observe that where the boundary condition of $u$ is non-trivial, we have that $\partial_x = \partial_\tau$ and $\partial_y = \partial_\nu$. So this becomes
\begin{multline*}
\|e^{-\varphi_\epsilon/h} \Delta e^{\varphi_\epsilon/h} u\|^2_{L^2(\Omega) }+ h^{-1}\| \frac{\partial_\tau \varphi_\epsilon}{h} u +\partial_\tau u\|^2_{L^2(\partial\Omega)}+ h^{-1} \|\frac{\partial_\nu\varphi_\epsilon}{h} u +\partial_\nu u\|^2_{L^2(\partial\Omega)}\\
\geq\frac{c}{\epsilon}\left( h^{-2} \|d\varphi_\epsilon u\|^2_{L^2(\Omega)}  + \|du\|^2_{L^2(\Omega)}  + h^{-1}\langle \Delta\varphi_\epsilon,|u|^2  \rangle +h^{-1} \int_{\partial\Omega} |u|^2 \partial_\nu\varphi_\epsilon\right)
\end{multline*}
Trivial estimates combined with \eqref{laplace of phieps} and $\epsilon>h>0$ yields the desired result. \qed

We now prove Proposition \ref{prop: carleman estimate} by gluing together the coordinate chart estimates we obtained in Lemma \ref{lem laplace on domain}.
\proof[Proof of Proposition \ref{prop: carleman estimate}] We can identify the boundary $\p M$ with a disjoint union of circles, $\partial M\cong \sqcup \s S^1$. We begin by taking an open cover $\{\Gamma_j\}_{j=1}^N$ for $\partial M$ such that $\Gamma_j \cong (0,1)\subset \mathbb R$ for all $j=1,\dots, N$. Let $\chi'_j\in C_c^\infty(\Gamma_j)$ be such that $\sum\limits_{j =1}^N \chi_j'^2 = 1$  on $\partial M$. Now extend $\Gamma_j$ and $\chi_j'$ to simply connected open sets $\Omega_j\subset M$ and $\chi_j\in C^\infty(\overline M)$ such that  $\Gamma_j \subset \partial\Omega_j$, $\chi_j|_{\partial M} = \chi_j'$, $\supp(\chi_j) \subset \overline \Omega_j$, and $\supp(\chi_j|\s _{\partial\Omega_j}) \subset\subset \Gamma_j$. Now add additional open sets $\Omega_j$ and $\chi_j\in C_c^\infty(\Omega_j)$ for $j = N+1,\dots, \widehat  N$ such that $\{\Omega\}_{j=1}^{\widehat  N}$ covers $M$ and $\sum\limits_{j=1}^{\widehat  N} \chi_j^2 = 1$ on $M$ with $\supp(\chi_j)\cap \partial M = \emptyset$ for $j = N+1,\dots, \widehat  N$. 

For each $j=1,\ldots, \widehat N$ set $u_j := \chi_j u \in C^\infty(\overline\Omega_j)$. Thus $\supp(u_j) \subset \overline \Omega_j$. Then on each $\Omega_j$, the function $u_j$ either vanishes to all order on $\partial \Omega_j$, or it satisfies the boundary conditions stated in Lemma \ref{lem laplace on domain}. We therefore have that 
\begin{multline}\label{eq: sum estimate}
\sum\limits_{j=1}^{\widehat  N}\left(\| e^{-\varphi_\epsilon/h} (\Delta +q) e^{\varphi_\epsilon /h}u_j\|^2_{L^2(\Omega)} + h^{-3} \|u_j\|_{L^2(\partial\Omega)}^2 +h^{-1} \|\partial_\nu u_j\|^2_{L^2(\partial\Omega)}\right. \\
\left.+ h^{-1} \|\partial_\tau u_j\|^2_{L^2(\Omega)}\right)
\geq \sum\limits_{j=1}^{\widehat  N}\frac{c}{\epsilon}\left(\frac{c}{\epsilon} \|u_j\|^2_{L^2(\Omega)} + \|du_j\|^2_{L^2(\Omega)} + h^{-2}\|d\varphi_\epsilon u_j\|_{L^2(\Omega)}^2\right).
\end{multline}
Observe that by construction of $u_j = \chi_j u$, we have  $\sum\limits_{j=1}^{\widehat  N} |u_j|^2 = |u|^2$ yielding
\begin{multline}\label{eq: sum estimate 2}
\sum\limits_{j=1}^{\widehat  N}\frac{c}{\epsilon}\left(\frac{c}{\epsilon} \|u_j\|^2_{L^2(\Omega)} + \|du_j\|^2_{L^2(\Omega)} + h^{-2}\|d\varphi_\epsilon u_j\|_{L^2(\Omega)}^2\right) \\
= \frac{c}{\epsilon}\Big(\frac{c}{\epsilon} \|u\|^2_{L^2(\Omega)} +\sum\limits_{j=1}^{\widehat  N} \|d(\chi_j u)\|^2_{L^2(\Omega)} + h^{-2}\|d\varphi_\epsilon u\|_{L^2(\Omega)}^2\Big).
\end{multline}
Now $ \|d(\chi_j u)\|^2_{L^2(\Omega)}  \geq \|\chi_jdu\|^2_{L^2(\Omega)} -C\|u\|^2_{L^2(\Omega)} $, with $\sum\limits_{j=1}^{\widehat  N}  \|\chi_jdu\|^2_{L^2(\Omega)}=  \|du\|^2_{L^2(\Omega)}$. So by taking $\epsilon>0$ sufficiently small in \eqref{eq: sum estimate 2}, we get from \eqref{eq: sum estimate} that
\begin{multline}
\label{eq: sum estimate 3}
\sum\limits_{j=1}^{\widehat  N}\left(\| e^{-\varphi_\epsilon/h} (\Delta +q) e^{\varphi_\epsilon /h}u_j\|^2_{L^2(\Omega)} + h^{-3} \|u_j\|_{L^2(\partial\Omega)}^2+h^{-1} \|\partial_\nu u_j\|^2_{L^2(\partial\Omega)}\right. \\
+ \left.h^{-1} \|\partial_\tau u_j\|^2_{L^2(\Omega)}\right)\geq \frac{c}{\epsilon}\left(\frac{c}{\epsilon} \|u\|^2_{L^2(\Omega)} +\|du\|^2_{L^2(\Omega)} + h^{-2}\|d\varphi_\epsilon u\|_{L^2(\Omega)}^2\right)
\end{multline}
 for some (new) constant $c>0$ independent of $\epsilon>0$ and $h>0$.
We now observe that on the left-hand side of \eqref{eq: sum estimate 3}
\begin{multline*}
e^{-\varphi_\epsilon/h} (\Delta + q) e^{\varphi_\epsilon /h}u_j    = e^{-\varphi_\epsilon/h} (\Delta +q) e^{\varphi_\epsilon /h}\chi_j u \\
=\chi_j e^{-\varphi_\epsilon/h} (\Delta +q) e^{\varphi_\epsilon /h} u +  u\Delta\chi_{\tbl{j}} + 2e^{-\varphi_\epsilon/h}\langle d\chi_j, d (e^{\varphi_\epsilon/h} u)\rangle \\
=\chi_j e^{-\varphi_\epsilon/h} (\Delta +q) e^{\varphi_\epsilon /h} u +  u\Delta\chi_{\tbl{j}} + \frac{2}{h}u \langle d\chi_j, d\varphi_\epsilon\rangle + 2\langle d\chi_j,du\rangle.
\end{multline*}
 Taking the $L^2$ norm and estimating the above by elementary inequalities we see that for $\epsilon>0$ sufficiently small in \eqref{eq: sum estimate 3} we get
\begin{multline}
\label{eq: sum estimate 4}
\| e^{-\varphi_\epsilon/h} (\Delta +q) e^{\varphi_\epsilon /h}u\|^2_{L^2(\Omega)} + h^{-3} \|u\|_{L^2(\partial\Omega)}^2 +h^{-1} \|\partial_\nu u\|^2_{L^2(\partial\Omega)} + h^{-1} \|\partial_\tau u\|^2_{L^2(\Omega)}\\
\geq \frac{c}{\epsilon}\left(\frac{c}{\epsilon} \|u\|^2_{L^2(\Omega)} +\|du\|^2_{L^2(\Omega)} + h^{-2}\|d\varphi_\epsilon u\|_{L^2(\Omega)}^2\right).
\end{multline}
Now fix $\epsilon>0$ and replace $u$ (as is standard) by $e^{\frac{h}{2\epsilon}(\sum_{j = 0}^K|\varphi_j|^2)} u$ we obtain the desired estimate when $0<h\ll\epsilon$.\qed\\
We now prove Proposition \ref{prop: shifted carleman estimate}:
\begin{proof}[Proof of Proposition \ref{prop: shifted carleman estimate}]
Patching together coordinate chart estimates of Lemma \ref{lem laplace on domain} we get the following $L^2$ estimate on the slightly larger manifold $\widetilde M$ for all functions $u\in C^\infty_c(\widetilde M)$.
\begin{eqnarray}
\label{eq: unshifted estimate}\nonumber
\| e^{-\varphi_\epsilon/h} h^2 e^{\varphi_\epsilon /h}u\|^2_{L^2(\widetilde M)} \geq \frac{c}{\epsilon}\left(\frac{c}{\epsilon}h^4 \|u\|^2_{L^2(\widetilde M)} +h^2 \|hdu\|^2_{L^2(\widetilde M)} + h^2\|d\varphi_\epsilon u\|_{L^2(\widetilde M)}^2\right).
\end{eqnarray}
In particular, if $u\in C^\infty_c(M)$ and $\chi\in C^\infty_c(\widetilde M)$ is identically $1$ in an open neighbourhood of $M$, we can apply the above estimate to \eqref{eq: unshifted estimate} to the function $\chi \langle hD\rangle^{-k}u$ to get
\begin{multline}
\label{eq: unshifted estimate II}
\| e^{-\varphi_\epsilon/h} h^2\Delta e^{\varphi_\epsilon /h}\chi \langle hD\rangle^{-k}u\|^2_{L^2(\widetilde M)} \\
\geq  \frac{c}{\epsilon}\left(\frac{c}{\epsilon} h^4 \|\chi \langle hD\rangle^{-k}u\|^2_{L^2(\widetilde M)}+ h^2\|hd\left(\chi \langle hD\rangle^{-k}u\right)\|^2_{L^2(\widetilde M)}\right. \\
+  h^{2}\|d\varphi_\epsilon \left(\chi \langle hD\rangle^{-k}u\right)\|_{L^2(\widetilde M)}^2\Big).
\end{multline}
We recall here that $\langle hD\rangle$ is the semiclassical operator whose symbol is given by $(1+|\xi|^2)^{1/2}$.
Commuting $e^{-\varphi_\epsilon/h} \Delta e^{\varphi_\epsilon /h}$ with $\chi \langle hD\rangle^{-k}$ we get 
\begin{eqnarray}
\label{eq: unshifted estimate III}
\| e^{-\varphi_\epsilon/h} h^2\Delta  e^{\varphi_\epsilon /h}u\|^2_{H_{scl}^{-k}(\widetilde M)} +\|e^{-\varphi_\epsilon/h}[h^2\Delta , \chi ] e^{\varphi_\epsilon /h}\langle hD\rangle^{-k}u\|^2_{L^2(\widetilde M)}  \\\nonumber
+\|\chi[e^{-\varphi_\epsilon/h} h^2\Delta  e^{\varphi_\epsilon /h}, \langle hD\rangle^{-k}]u\|^2_{L^2(\widetilde M)}\\\nonumber
\geq \frac{c}{\epsilon}\left(\frac{c}{\epsilon}h^4 \|\chi \langle hD\rangle^{-k}u\|^2_{L^2(\widetilde M)}+h^2\|hd\left(\chi \langle hD\rangle^{-k}u\right)\|^2_{L^2(\widetilde M)}\right. \\\nonumber
\qquad \qquad\qquad\qquad\qquad\qquad\qquad\qquad\qquad \left.+  h^{2}\|d\varphi_\epsilon \left(\chi \langle hD\rangle^{-k}u\right)\|_{L^2(\widetilde M)}^2\right).
\end{eqnarray}
The commutator term with $[\Delta,\chi]$ has a coefficient supported uniformly away from $M$ so the term $\|e^{-\varphi_\epsilon/h}[ h^2\Delta , \chi ] e^{\varphi_\epsilon /h}\langle hD\rangle^{-k}u\|^2_{L^2(\widetilde M)}$ in \eqref{eq: unshifted estimate III} can be absorbed on the right side:

\begin{eqnarray}\label{eq: unshifted estimate IV}
\| e^{-\varphi_\epsilon/h} h^2\Delta  e^{\varphi_\epsilon /h}u\|^2_{H_{scl}^{-k}(\widetilde M)} +\|\chi[e^{-\varphi_\epsilon/h} h^2\Delta  e^{\varphi_\epsilon /h}, \langle hD\rangle^{-k}]u\|^2_{L^2(\widetilde M)} \\\nonumber
\geq \frac{c}{\epsilon}\left(\frac{c}{\epsilon}h^4 \|\chi \langle hD\rangle^{-k}u\|^2_{L^2(\widetilde M)}+h^2\|hd\left(\chi \langle hD\rangle^{-k}u\right)\|^2_{L^2(\widetilde M)}\right. 
\\\nonumber
\left.+  h^{2}\|d\varphi_\epsilon \left(\chi \langle hD\rangle^{-k}u\right)\|_{L^2(\widetilde M)}^2\right).
\end{eqnarray}

Standard semiclassical calculus allows us to bound the right side of \eqref{eq: unshifted estimate IV} from below  to obtain
\begin{eqnarray}
\label{eq: unshifted estimate V}
\| e^{-\varphi_\epsilon/h} h^2\Delta  e^{\varphi_\epsilon /h}u\|^2_{H_{scl}^{-k}(\widetilde M)} +\|\chi[e^{-\varphi_\epsilon/h} h^2\Delta  e^{\varphi_\epsilon /h}, \langle hD\rangle^{-k}]u\|^2_{L^2(\widetilde M)} \geq \\\nonumber
 \frac{c}{\epsilon}\left(\frac{c}{\epsilon}h^4 \|u\|^2_{H^{-k}(\widetilde M)}+h^2\|h\s du\|^2_{H^{-k}(\widetilde M)} +  h^{2}\|d\varphi_\epsilon \left(\chi \langle hD\rangle^{-k}u\right)\|_{L^2(\widetilde M)}^2\right),
\end{eqnarray}
which holds for all $0<h\ll \epsilon\ll 1$ sufficiently small and $u\in C^\infty_c(M)$,

We now compute the commutator $[e^{-\varphi_\epsilon/h} h^2\Delta  e^{\varphi_\epsilon /h}, \langle hD\rangle^{-k}]$ by first recalling \eqref{laplace of phieps} then apply standard semiclassical calculus
\begin{eqnarray}
\label{eq: commutator}
[e^{-\varphi_\epsilon/h} h^2\Delta  e^{\varphi_\epsilon /h}, \langle hD\rangle^{-k}] =\left [h^2\Delta +  \frac{h^2}{\epsilon} \sum\limits_{j=0}^K |d\varphi_j|^2, \langle hD \rangle^{-k}\right]  \\\nonumber
+\left[ d\varphi_\epsilon \cdot hd ,   \langle hD\rangle^{-k}\right] + \left[|d\varphi_\epsilon|^2,\langle hD\rangle^{-k} \right]\\\nonumber
= h \Op_h\left(\sum\limits_{j=1}^2\xi_j a^j_\epsilon(x,\xi;h)\right)+h^2 \Op_h\left(\sum\limits_{j=1}^2 b^j_\epsilon(x,\xi;h)\right)
\end{eqnarray}
where $a_\epsilon^j(x,\xi;h)\in S^{-k}_{scl}$ and $b_\epsilon^j(x, \xi;h)\in S^{-k-1}_{scl}$ are $\epsilon$-dependent symbols which are uniformly bounded for $0<h\ll\epsilon \ll 1$.  Standard mapping properties of semiclassical $\Psi$DO allow us to absorb the resulting two terms of \eqref{eq: commutator}  into the right side of \eqref{eq: unshifted estimate V}, provided we choose $0<h\ll\epsilon\ll1$ sufficiently small
\begin{equation}
\label{eq: unshifted estimate VI}
\| e^{-\varphi_\epsilon/h} h^2\Delta  e^{\varphi_\epsilon /h}u\|^2_{H_{scl}^{-k}(\widetilde M)} \geq \frac{c}{\epsilon}\left(\frac{c}{\epsilon}h^4 \|u\|^2_{H^{-k}(\widetilde M)}+h^2\|hdu\|^2_{H^{-k}(\widetilde M)} \right).
\end{equation}
For $0<h\ll\epsilon\ll1$ sufficiently small, we may replace $\Delta$ in \eqref{eq: unshifted estimate VI} by $\Delta +q$ and absorb the extra term on the right side to get
\begin{equation}
\label{eq: unshifted estimate VII}
\| e^{-\varphi_\epsilon/h} h^2(\Delta +q) e^{\varphi_\epsilon /h}u\|^2_{H_{scl}^{-k}(\widetilde M)} \geq \frac{c}{\epsilon}\left(\frac{c}{\epsilon}h^4 \|u\|^2_{H^{-k}(\widetilde M)}+h^2\|hdu\|^2_{H^{-k}(\widetilde M)} \right).
\end{equation}
Recalling the definition of $\varphi_{\eps} := \varphi - \frac{h}{2\epsilon}(\sum_{j = 0}^K|\varphi_j|^2)$, we have for fixed $\epsilon>0$ sufficiently small,
\begin{equation}
\label{eq: unshifted estimate VIII}
\| e^{-\varphi/h} h^2(\Delta+q)  e^{\varphi /h}u\|^2_{H_{scl}^{-k}(\widetilde M)} \geq C\left(\frac{c}{\epsilon}h^4 \|u\|^2_{H^{-k}(\widetilde M)}+h^2\|hdu\|^2_{H^{-k}(\widetilde M)} \right)
\end{equation}
for some constant $C>0$ independent of $h$.
\end{proof}

An immediate Corollary to Proposition \ref{prop: shifted carleman estimate} is the following solvability result (compare for example \cite[Lemma 4.3.2]{guillarmou2011calderon}).
\begin{corollary}
Let $(M,g,\partial M)$ be a Riemann surface with boundary compactly contained $(\widetilde M,g,\partial \widetilde M)$, $\varphi\in C^\infty( \widetilde M)$ be a non-constant real valued harmonic function on $\widetilde M$, and $q\in C^\infty(\widetilde M)$. Let $k\in \mathbb N$. For all $f\in H^k(\widetilde M)$ and all $h>0$, there exists a family of solutions $r_h \in H^k(M)$ solving 
$$ e^{-\varphi/h}(\Delta + q)  e^{\varphi /h}r_h = f \quad {\rm on}\ M$$
satisfying the estimate $\|r_h\|_{H^k_{scl}} \leq C \|f\|_{H^k_{scl}}$.
\end{corollary}

\subsection{Complex geometric optics solutions for the Calder\'on problem}\label{sec:CGOsCalderon}
The purpose of this section is to construct a special type of CGOs for the Calder\'on problem of the equation $\Delta_g+q$. These solutions will be used to determine the boundary values of holomorphic functions.
\begin{proposition}
\label{prop: CGO for boundary determination}
Let $(M, g, \partial M)\subset\subset (\widetilde M, \widetilde g, \partial \widetilde M)$ where both are Riemann surfaces with smooth boundary. Let $\Phi$ be a holomorphic function on $(\widetilde M, \widetilde g,\partial \widetilde M)$. Let $a\in C^\infty(\widetilde M)$ be a holomorphic function which vanishes to sufficiently high order at all critical points of $\Phi$. Then there exists solutions to $(\Delta_g + q) u = 0$ in $M$ of the form
$$u= e^{\Phi/h} (a + r)$$
where $\|r\|_{C^1(\overline M)} = o(1)$ as $h\to 0$.
\end{proposition}
To prove this proposition, we first need the following preliminary ansatz. We obtain
\begin{lemma}\label{lem: CGO ansatz}
Let $\Phi$ and $a$ be a holomorphic functions on $(M',g, \partial M')$. For all $k, N_1\in \mathbb N$, there exists an $N_0\in \mathbb N$ such that if $a$ has vanishing of order greater than $N_0$ at all critical points $\{p_1,\dots, p_l\}$ of $\Phi$, then we can construct an ansatz for $u$ satisfying $e^{-\Phi/h}(\Delta_g + q) u = O_{C^k(M')}(h^{N_1})$ of the form
\begin{equation}
\label{eq: CGO ansatz}
u = e^{\Phi/h}\left (a + \sum\limits_{j=1}^{N_1} h^j r_j\right)
\end{equation}
with $r_j\in C^{k+2}(M')$ independent of $h>0$.
\end{lemma}
\proof
We modify the construction in \cite{guillarmou2011calderon} to take advantage of the fact that $a$ vanishes to high order at critical points of $\Phi$. To this end let us apply the operator $e^{-\Phi/h} (\Delta_g + q)$ to the ansatz \eqref{eq: CGO ansatz}. By recalling that $\Delta_g=-2i\star \overline \p \p$, that $a$ is holomorphic and the identity
\[
 (-2i)^{-1}\Delta_g (e^{\Phi/h}f) =\star\s  \overline \p \left[ e^{\overline \Phi/h} \p (e^{2i\psi/h} f)\right]=e^{ \Phi/h} \star \overline \p \left[ e^{-2i\psi/h}   \p (e^{2i\psi/h} f)\right],
\]  
holding for any $C^2$ smooth $f$, we obtain
\begin{multline}\label{eq: ansatz 1}
e^{-\Phi/h} (\Delta_g + q) e^{\Phi/h}\left (a + \sum\limits_{j=1}^{N_1} h^j r_j\right) \\
=  \left( -2i\star\bar\partial e^{-2i\psi/h} \partial e^{2i\psi/h} + q\right) \sum\limits_{j=1}^{N_1} h^j r_j + qa \\
 = \star\bar\partial\left(-4\sum\limits_{j=1}^{N_1} h^{j-1}\partial\psi r_j-2i\sum\limits_{j=1}^{N_1} h^j \partial r_j \right) +\sum\limits_{j=1}^{N_1} h^j q r_j +qa.
\end{multline}
So if we want the expression on the right hand side of \eqref{eq: ansatz 1} to be $O_{C^k(M')}(h^{N_1})$,  we should look to solve 
\begin{eqnarray}\label{eq: ansatz of CGO}\nonumber
&&-4   \sum\limits_{j=1}^{N_1} h^{j-1}\partial\psi r_j+\sum\limits_{j=1}^{N_1} h^j( \tbl{-2i}\partial r_j +    \partial G(q r_j)  -  \omega_j)\\&&= -\partial G(aq) + \omega_0+ O_{C^{k+1}(M)}(h^{N_1})
\end{eqnarray}
for some $\omega_0,\dots,\omega_{N_1}\in C^\infty(M; T^*_{1,0}M)$ satisfying $\star\overline \p\omega_j=\partial^* \omega_j = 0$. 

Using Lemma 2.2.4 of \cite{guillarmou2011calderon}, we will first choose the power series expansion of $\omega_0$ at the points $\{p_1,\dots, p_l\}$ so that the Taylor series expansion of 
\[
 b(z,\bar z) dz:=-\partial G(aq) + \omega_0
\]
 near each critical point is of the form
\begin{eqnarray}\label{eq: prelim expansion of b} 
b(z,\bar z) = \sum\limits_{j\geq 1,\s m\geq 0,\s j+m \leq N_0 }b^0_{j,m} \bar z^j z^m + O(|z|^{N_0+1}),
\end{eqnarray}
where $N_0$ is such that $a$ vanishes to order of at least $N_0$ at all critical points of $\Phi$. 
Note that the sum over $j$ starts at $1$ instead of $0$ as we have chosen $\omega_0$ to killed off the $j=0$ terms in the double sum.
Now, by construction, $\partial^* b = aq$ which vanishes to order $N_0$ at all critical points $\{p_1,\dots, p_l\}$ of $\Phi$. This combined with \eqref{eq: prelim expansion of b} shows that
\begin{eqnarray}
 \sum\limits_{j\geq 1,\s m\geq 0,\s j+m \leq N_0 }j\s b^0_{j,m} \bar z^{j-1} z^m  = O(|z|^{N_0})
\end{eqnarray}
which means
\begin{eqnarray}\label{eq: b0jm vanishes}
b^0_{j,m} = 0,\quad {\rm for\ all}\quad j+m \leq N_0, j\geq 1, m\geq 0.
\end{eqnarray}
Inserting this into \eqref{eq: prelim expansion of b} yields that the coordinate expression $b(z,\bar z) dz$ for $\partial G(aq) - \omega_0$ around each critical point is given by
\begin{eqnarray}
\label{eq: b vanishes to high order}
b(z, \bar z) = O(|z|^{N_0+1}).
\end{eqnarray}
Motivated by \eqref{eq: ansatz of CGO} we choose 
\begin{eqnarray}
\label{eq: choice of r1}
r_1 :=- \frac{\partial G(aq) - \omega_0}{4\partial\psi}.
\end{eqnarray}
Note that for each fixed $k\in \mathbb N$, the condition \eqref{eq: b vanishes to high order} ensures that $r_1\in C^{k+2}(\widetilde M)$ if $N_0\in \mathbb N$ is chosen sufficiently large. If we had chosen to stop here, we would have constructed an ansatz satisfying \eqref{eq: ansatz of CGO} with an error of $O_{C^{k+1}}(h)$.

We wish to improve the decay of the error term. To this end, note that at each critical point $p_j$ of $\Phi$, the coordinate expression of $r_1$ satisfies
\begin{eqnarray}\label{eq: r_1 vanishing}
r_1 = O(|z|^{N_0 + 1- {\rm deg}(\partial\psi)})
\end{eqnarray}
where  ${\rm deg}(\partial\psi)$ is the degree of vanishing at $p_j$ and we have suppressed the dependence on the point in our notation. We now choose $\omega_1$ in \eqref{eq: ansatz of CGO} so that in coordinates at each critical point,
\begin{eqnarray}\label{eq: omega1 kills terms}\nonumber\partial r_1 + \partial G(qr_1) - \omega_1 = \sum\limits_{j\geq 1, m\geq 0, j+m \leq N_0 -{\rm deg}(\partial\psi)}b^1_{j,m} \bar z^j z^m +  O(|z|^{N_0 - {\rm deg}(\partial\psi)}).\\
\end{eqnarray}
From this, we can repeat the argument we made to deduce \eqref{eq: b0jm vanishes} to show that 
\begin{eqnarray}\label{eq: b1jm vanishes}
b^1_{j,m} = 0,\quad {\rm for\ all}\quad j+m \leq N_0 - {\rm deg}(\partial\psi), j\geq 1, m\geq 0.
\end{eqnarray}
Inserting this into \eqref{eq: omega1 kills terms} we get that, at each critical point of $\Phi$,
\begin{eqnarray}\label{eq: vanishing of b1}
\partial r_1 + \partial G(qr_1) - \omega_1 =   O(|z|^{N_0 - {\rm deg}(\partial\psi)}).
\end{eqnarray}
The ansatz \eqref{eq: ansatz of CGO} motivates us to choose 
\begin{eqnarray}
\label{eq: def of r2}
r_2 := \frac{\partial r_1 + \partial G(qr_1) - \omega_1}{4\partial\psi}.
\end{eqnarray}
Note that by \eqref{eq: vanishing of b1}, if $N_0\in \mathbb N$ is sufficiently large, $r_2\in C^{k+2}(\widetilde M)$ and 
\begin{eqnarray}
\label{eq: vanishing of r2}
r_2 = O(|z|^{N_0-2{\rm deg}(\partial\psi)})
\end{eqnarray}
in conformal coordinates around each critical point of $\Phi$. If we had stopped here, the error term in \eqref{eq: ansatz of CGO} would decay like $O_{C^{k+1}}(h^2)$.

Therefore, for each fixed $N_1,k \in \mathbb N$, if we assume $N_0\in \mathbb N$ is sufficiently large, we can proceed inductively to construct $r_3,\dots, r_{N_1}\in C^{k+1}(\widetilde M)$ so that the remainder in \eqref{eq: ansatz of CGO} will decay like $O_{C^{k+1}}(h^{N_1})$. This concludes the proof.
\qed
\subsection{Boundary values of holomorphic functions}
For a Riemann surface with boundary, $(M,g,\partial M)$, we denote the set of holomorphic functions by $\OO(M)$. The goal of this section is to prove  
\begin{proposition}
\label{prop: same holomorphic bv}
Let $(M_1, g_1, \partial M_1)$ and $(M_2, g_2,\partial M_2)$ be two Riemann surfaces with identical boundary. If $\Lambda_{g_1,q_1} =\Lambda_{g_2, q_2}$ on $\partial M_1 = \partial M_2$ then there exists enlarged surfaces $( M_j', g_j', \partial M_j')$, $j=1,2$ with $\Omega := M_1'\setminus M_1 = M_2' \setminus M_2$ and $ g_1' = g_2'$ which is Euclidean near $\partial M_1' = \partial M_2'$  on $\Omega$ such that
\begin{eqnarray}
\label{eq: same holomorphic bv}
\{ \Phi |\s _{\partial M_1'} |\s  \Phi\in \OO( M_1')\} =\{\Phi|\s _{\partial  M_2'} |\s  \Phi\in \OO( M_2')\}.
\end{eqnarray}
\end{proposition}
To prove this proposition we will work with real parts of holomorphic functions, which we denote by $\mathcal H (M) := {\rm Re}(\OO(M))$ and $\mathcal H (M)\s|_{\partial M}$ to be the trace of these functions. Let us note that we have the well-known characterization 
\begin{equation}\label{eq characterization of useful harmonic functions}
\mathcal H (M) = \{u\in C^\infty(\overline M)\s |\s  \Delta_g u = 0,\  \int_{\gamma_j} \star du  =0,\ \text{ for all } j=1,\ldots N\},
\end{equation}
where $\gamma_1,\dots ,\gamma_N$ form the homological basis of $(M,\partial M)$. We also note that there is a characterization of boundary values of holomorphic functions:
\begin{lemma}[Lemma 4.1 of \cite{guillarmou2011identification}]
\label{lem: characterization of holom}
Let $(M, g,\partial M)$ be a Riemann surface with boundary. A complex valued function $f\in C^\infty (\partial M)$ extends holomorphically to $M$ if and only if $\partial_\tau f - i\Lambda_{g,0}f  =0$.
\end{lemma}
\proof
By Lemma 4.1 of \cite{guillarmou2011identification}, the harmonic extension $u_f\in C^\infty(M)$ of $f\in C^\infty(\partial M)$ is holomorphic if and only if 
$$\int_{\partial M} f\iota_{\partial M}^* \eta =0$$
for all $\eta \in C^\infty(M; T^*_{0,1}M)$ with $\bar\partial \eta  =0$. By Proposition 2.1 of \cite{guillarmou2011identification}, we have that the set of $\eta\in C^\infty(M; T^*_{0,1}M)$ with $\bar\partial\eta = 0$ are precisely sections of the form $\eta = \partial \phi$ for some $\phi\in C^\infty(M)$ harmonic. So $f$ extends holomorphically if and only if
$$\int_{\partial M} f\iota_{\partial M}^* \partial \phi=0$$
for all harmonic functions $\phi$. A coordinate calculation yields that 
\[
\iota_{\partial M}^* \partial \phi = \left(\partial_\tau( \phi|\s _{\partial M} )+ i\Lambda_{g,0}(\phi|\s _{\partial M})\right) d\tau,
\]
where $\tau$ is the coordinate on the boundary. By integrating by parts and using the self-adjointness of $\Lambda_{g,0}$ we get that
\[
\int_{\partial M} (-\partial_\tau f + i\Lambda_{g,0} f) (\phi|\s _{\partial M}) d\tau= 0
\] 
for all harmonic functions $\phi$. But the boundary values of harmonic functions can be chosen to be arbitrary. We therefore get that
\[
(-\partial_\tau f + i\Lambda_{g,0} f)  = 0
\]
is equivalent to $f$ having a holomorphic extension.
\qed

The harmonic extension of an arbitrary function $f\in C^\infty(\partial M)$ does not, in general, admit a harmonic conjugate. To this end, we prove the following lemma which gives a natural way to modify $f\in C^\infty(\partial M)$ to a function whose harmonic extension admits a harmonic conjugate:
\begin{lemma}\label{lem: projection to harmonic conjugate}
Let $(M, g, \partial M)$ be a Riemann surface with boundary. There exists a linear operator ${\rm proj} : C^\infty(\partial M;\mathbb R)\to \mathcal H(M)|\s _{\partial M}$, continuous with respect to the $H^{3/2}(\partial \tbl{M})$ topology, which maps $C^\omega(\partial M)$ into itself, and $C^\infty(\partial M)$ into itself, and is the identity map on $\mathcal H(M)\s |_{\partial M}$.
\end{lemma}
\begin{proof}
Let $\{\gamma_1,\dots, \gamma_N\}$ be a homological basis for $( M, \partial  M)$.  We define on $C^\infty(\partial  M; \mathbb R)$ the linear map
$${\rm Coho}: f\mapsto \left(\int_{\gamma_1} *du_f,\dots, \int_{\gamma_N} *d u_f\right) \in \mathbb R^N$$
where $u_{ f}\in C^\infty( M)$ is the harmonic extension of $ f\in C^\infty(\partial  M;\mathbb R)$.
The range of ${\rm Coho}$ is isomorphic to $\R^{N'}$ for some $N'\leq N$, which is spanned by the basis $e_j := {\rm Coho}(f_j)$
for some $ f_1,\dots,  f_{N'} \in C^\infty(\partial  M)$. Note that using standard Fourier series representation of functions on $\partial M$, every $f\in C^\infty(\partial M)$ can be approximated by real analytic functions in any regularity norm, such as $H^{3/2}(\p M)$. Since linear independence is an open condition, we may assume without loss of generality that the functions $f_j$ defining the basis vectors $e_j = {\rm Coho}(f_j)$ are all in $C^\omega(\partial M;\mathbb R)$.

For all $f\in C^\infty(\partial M)$ we have that ${\rm Coho}(f) = \sum_{j=1}^{N'} a_j(f) {\rm Coho}(f_j)$ where for each $j = 1,\dots, N'$, $f\mapsto a_j(f)\in \mathbb R$ depends linearly on $f$ and is a bounded linear functional with respect to the $H^{3/2}(\partial M)$ topology due to elliptic regularity. We now define 
$${\rm proj}: f \mapsto f - \sum\limits_{j=1}^{N'}a_j(f)f_j.$$
Then clearly the range of $\rm proj$ is in the null-space of the map ${\rm Coho}$. So by \eqref{eq characterization of useful harmonic functions}, ${\rm proj}(f)\in \mathcal H (M)|\s _{\partial M}$ for all $f\in C^\infty(\partial M)$. And by \eqref{eq characterization of useful harmonic functions} again, $\rm proj$ acts as the identity map on $\mathcal H (M)|\s _{\partial M}$. Finally, ${\rm proj}(C^\omega(\partial M)) \subset C^\omega(\partial M)$ and ${\rm proj}(C^\infty(\partial M)) \subset C^\infty(\partial M)$  since each of the $f_j$ are in $C^\omega(\partial M)$.

\end{proof}

We will show that if $\Lambda_{g_1,q_1} =\Lambda_{g_2, q_2}$ on $\partial M_1 = \partial M_2$, then  $\Lambda_{g_1,0}(f) = \Lambda_{g_2,0}(f)$ on the boundary of an extended surface $\partial \widetilde M_1 = \partial \widetilde M_2$ for $f\in \big(\mathcal H(\widetilde M_1)\big)|\s _{\partial \widetilde M_1}$. To this end, we first prove a density Lemma:
\begin{lemma}\label{lem: density of real analytic} On any Riemann surface $(M,g,\partial M)$ with boundary, the function space ${\mathcal H}(M)|\s _{\partial M} \cap C^\omega(\partial M; \mathbb R) $ is $L^2$-dense in ${\mathcal H}(M)|\s _{\partial M}$.
\end{lemma}
\begin{proof}
Let $f\in \mathcal H(M)|\s _{\partial M}$ and choose $f^{(\ell)}\in C^\omega(\partial M ;\mathbb R)$ such that $f^{(\ell)}\to f$ in $H^{3/2}(\partial M)$. By Lemma \ref{lem: projection to harmonic conjugate}, the sequence of functions $\{{\rm proj}(f^{(\ell)}) \}_{\ell}$
is contained in ${\mathcal H}(M)|\s _{\partial M} \cap C^\omega(\partial M; \mathbb R)$ and converges in $L^2$ to $f$ since $\rm proj$ is a continuous operator with respect to the $H^{3/2}(\partial M)$ topology and is the identity when acting on $f\in \mathcal H( M)|\s _{\partial M}$.
\end{proof}
\begin{lemma}
\label{lem: DN map for harmonic}
Let $(M_1, g_1, \partial M_1)$ and $(M_2, g_2,\partial M_2)$ be two Riemann surfaces with identical boundary. If $\Lambda_{g_1,q_1} =\Lambda_{g_2, q_2}$ on $\partial M_1 = \partial M_2$ then there exists enlarged surfaces $( M_j',  g_j', \partial M_j')$, $j=1,2$ with open set $\Omega := M_1'\setminus M_1 =  M_2' \setminus M_2$ and $ g_1' =  g_2'$ on $\Omega$ which is Euclidean near $\partial M_1' = \partial M_2'$ such that 
\begin{eqnarray}
\label{eq: DN for harmonics are equal}
\Lambda_{ g_1',0} f= \Lambda_{ g_2',0}f
\end{eqnarray}
for all $f\in \mathcal H( M_1')|\s _{\partial M_1'}$.
\end{lemma}
\proof
Since $\partial M_1 = \partial M_2$ and $\Lambda_{g_1, q_1} = \Lambda_{g_2, q_2}$ we can glue collar neighbourhoods $\Omega$ on each component of $\partial M_1 = \partial M_2$ and extend the metrics $g_1$ and $g_2$ smoothly onto $\Omega$ so that they are identical in this neighourhood and Euclidean near $\partial  M_2' = \partial M_1'$. Set 
\[
M_j':= M_j \cup \Omega. 
\]
We also extend $q_1$ and $q_2$ smoothly onto $\Omega$ so that they agree on $\Omega$. We denote by $g'_j$ and $q_j'$ the metric and function defined on the extended manifolds $M_j'$.

We need to verify \eqref{eq: DN for harmonics are equal}. By using Lemma \ref{lem: density of real analytic}, we only need to verify \eqref{eq: DN for harmonics are equal} for real-analytic boundary functions. To this end, let $\phi_1 \in \mathcal H( M_1')$ be a non-constant harmonic function with real analytic boundary value on $\partial M_1'$, which is the real part of the holomorphic function $\Phi_1 := \phi_1 +i\psi_1$. The lemma is verified  once we show that if $\phi_2\in C^\infty( M_2')$ is the harmonic extension of $\phi_1|\s _{\partial  M_1'}$ into $ M_2'$ then
\begin{eqnarray}    
\label{eq: phi1=phi2 on Omega}
\phi_1 =\phi_2\ {\rm on}\ \Omega.
\end{eqnarray}

Suppose by contradiction that we have that at some point $\hat p\in \Omega^{\rm int}$, $\phi_1(\hat p) \neq \phi_2(\hat p)$ and we may assume without loss of generality that 
\begin{eqnarray}\label{eq: phi1 is bigger}
\phi_1(\hat p) >\phi_2(\hat p)
\end{eqnarray}
and that $d\Phi_1(\hat p)\neq 0$ by density. By the fact that $\phi_1$ is real analytic on $\partial M'_1$, we may use Cauchy-Kovalevskaya to show that there exists a slightly larger surface $\widetilde M_1$ compactly containing $M'_1$, $M_1'\subset \subset \widetilde M_1$, such that $\phi_1$ extends to be a harmonic function on $\widetilde {M_1}$, which we still denote by $\phi_1$. Since $\phi_1$ has a harmonic conjugate on $ M_1'$, we can conclude that $\star d \phi_1$ is exact on $ M_1'$. And since $\star d\phi_1$ is closed on the extended manifold $\widetilde M_1$, which has the same homological basis as $ M_1'$, we can conclude that it is also exact on $\widetilde M_1$. Therefore $\phi_1$ has a harmonic conjugate on $\widetilde M_1$ which we will still denote by $\psi_1$. So $\Phi$ extends to be a holomorphic function on $\widetilde M_1$ which we still denote by $\Phi_1 = \phi_1 + i\psi_1$. 

Let $a\in \mathcal O (\widetilde M_1)$ have zeros of sufficiently high order at all critical points of $\Phi_1$ so that Lemma \ref{lem: CGO ansatz} can be applied and that $a(\hat p) = 1$. Such holomorphic functions exist due to Lemma 2.2.4 of \cite{guillarmou2011calderon}. Here we used that $d\Phi_1(\hat p)\neq 0$.

Now let $u_1$ be the CGO which solves $(\Delta_{\widetilde g_1} + q_1) u_1 = 0$ on $( M_1',  g_1', \partial M_1')$ as constructed in Proposition \ref{prop: CGO for boundary determination}:
\begin{eqnarray}
\label{eq: CGO u1}
u_1 = e^{\Phi_1/h}(a + r),\ \|r\|_{C^1(\overline M_1')} = o(1)
\end{eqnarray}
with $\Phi_1 = \phi_1 + i\psi_1$.

Now let $u_2$ be a solution on $(M_2, g_2, \partial M_2)$ to
$$
(\Delta_{g_2}+ q_2) u_2 = 0\ {\rm on}\ M_2,\ u_2|\s _{\partial M_2} = u_1|\s _{\partial M_1}.
$$
Since $\Lambda_{g_1,q_1} = \Lambda_{g_2,q_2}$ on $\partial M_2 = \partial M_1$, we can extend $u_2$ to be a solution to $(\Delta_{\widetilde g_2}+ q_2) u_2 = 0$ on $ M_2' = M_2 \cup \Omega$ by setting
\begin{eqnarray}
\label{eq: u2 on Omega}
u_2 = u_1\ {\rm on}\ \Omega,
\end{eqnarray} 
where in $\Omega$, $ g_1'|\s _{\Omega} =g_2'|\s _{\Omega}$ and $q_1'|\s _{\Omega} = q_2'|\s _{\Omega}$.

We will apply the Carleman estimate of Proposition \ref{prop: carleman estimate} to $e^{-\phi_2/h} u_2$ on the surface $ M_2'$ with harmonic weight $\varphi = \phi_2$ and differential operator $\Delta_{\widetilde g_2} + q_2$ to obtain
{ \begin{eqnarray*}
\| e^{-\phi_2/h}u_2\|^2_{L^2(M_2')} &\leq&  h^{-3}\|e^{-\phi_2/h}u_2\|_{L^2(\partial  M_2')}^2 + h^{-1}\|\partial_\nu (e^{-\phi_2/h}u_2)\|_{L^2(\partial  M_2')}^2\\&& +  h^{-1}\|\partial_\tau (e^{-\phi_2/h}u_2)\|_{L^2(\partial  M_2')}^2 .
\end{eqnarray*}
Now since $\phi_1 = \phi_2$ on $\partial M_2'$ by construction and $u_1 = u_2$ on $\Omega$ by \eqref{eq: u2 on Omega} we can change the rightside of the inequality:

\begin{eqnarray*}
\| e^{-\phi_2/h}u_2\|^2_{L^2(M_2')} &\leq&  h^{-3}\|e^{-\phi_1/h}u_1\|_{L^2(\partial  M_2')}^2 + h^{-1}\|\partial_\nu (e^{-\phi_1/h}u_1)\|_{L^2(\partial  M_2')}^2\\&& +  h^{-1}\|\partial_\tau (e^{-\phi_1/h}u_1)\|_{L^2(\partial  M_2')}^2 .
\end{eqnarray*}
Now use the explicit form of $u_1$ given by \eqref{eq: CGO u1} we get

$$
\| e^{-\phi_2/h}u_2\|^2_{L^2(\Omega)} \leq Ch^{-3}.
$$
Now use again $u_1|\s _\Omega = u_2|\s _\Omega$ by \eqref{eq: u2 on Omega} and the explicit form of $u_1$ given by \eqref{eq: CGO u1}. This shows
$$
\| e^{(\phi_1-\phi_2)/h}(a+r)\|^2_{L^2(\Omega)} \leq Ch^{-3}
$$
with $\|r\|_{C^1(\widetilde M_1)} = o(1)$. With our assumption \eqref{eq: phi1 is bigger} and that $a(\hat p)  =1$, this leads to a contradiction as $h\to 0$. We therefore have that \eqref{eq: phi1=phi2 on Omega} must hold and the lemma is complete.
\qed

\begin{proof}[Proof of Proposition \ref{prop: same holomorphic bv}]
Let $( M_j', g_j', \partial M_j')$, $j = 1,2$, be the extended surfaces constructed in Lemma \ref{lem: DN map for harmonic} so that $\partial M_1' = \partial M_2'$. It suffices to prove inclusion in one direction in \eqref{eq: same holomorphic bv}. To this end, let $\Phi_1 = \phi_1 + i\psi_1\in \mathcal O( M_1')$ so that $\phi_1|\s _{\partial  M_1'}, \psi_1|\s _{\partial M_1'} \in \mathcal H( M_1')|\s _{\partial  M_1'}$. By Lemma \ref{lem: DN map for harmonic} we have that 
\begin{eqnarray}
\label{eq: DN are the same}
\Lambda_{ g_1'}(\Phi_1|\s _{\partial M_1'}) = \Lambda_{ g_2'}(\Phi_1|\s _{\partial M_2'}).
\end{eqnarray}

Since $\Phi_1|\s _{\partial M_1'}$ has a holomorphic extension by construction, Lemma \ref{lem: characterization of holom} asserts that it must satisfy $\partial_\tau (\Phi_1|\s _{\partial M_1'}) - i\Lambda_{ g_1'}  (\Phi_1|\s _{\partial M_1'}) = 0$. By \eqref{eq: DN are the same} and $\partial M_1' = \partial M_2'$, we have that $\partial_\tau (\Phi_1|\s _{\partial M_2'}) - i\Lambda_{ g_2'}  (\Phi_1|\s _{\partial M_2'}) = 0$. So by  Lemma \ref{lem: characterization of holom} we have that $\Phi_1|\s _{\partial  M_2'}$ extends holomorphically into $( M_2',  g_2', \partial M_2')$. So we have 
$$\{ \Phi |\s _{\partial  M_1'} |\s  \Phi\in \OO( M_1')\} \subset\{\Phi|\s _{\partial  M_2'} |\s  \Phi\in \OO( M_2')\}$$
as required.\qedhere
\end{proof}
\subsection{Poisson embedding of surfaces}
To complete the proof of Theorem \ref{thm:unique2d}, we apply the standard argument of Poisson embedding (see \cite{lassas2020poisson}) to our situation where the data is restricted to boundary values of holomorphic functions.  This will show that $(M_1,g_1)$ and $(M_2,g_2)$ in the theorem are the same up to a conformal mapping. After this, referring to  \cite{guillarmou2010calderon} finishes the proof.  We mention that the existence of a conformal mapping proved in this section also follows from \cite{lassas2001determining, Belishev03}, but we include a new proof for completeness and future applications.

Let $(M,g, \partial M)$ be a Riemann surface with boundary and we set for $s\geq 1/2$, 
$$
H_{\mathcal H(M)}^s(\partial M) :=\overline{\left( \mathcal H(M)|\s _{\partial M}\right)}^{H^s(\partial M)}\subset H^s(\partial M).
$$
We define the Poisson embedding as the map $P : \overline M \to \big( H_{\mathcal H(M)}^s(\partial M)\big)'$ by 
\begin{eqnarray}
\label{eq: poisson embedding}
P(p) f := u_{f}(p),
\end{eqnarray}
where $u_f$ denotes the harmonic extension of $f\in  H_{\mathcal H (M)}^s(\partial M)$ and the notation $\big( H_{\mathcal H(M)}^s(\partial M) \big)'$ refers to the dual of $H_{\mathcal H(M)}^s(\partial M)$. 

\begin{lemma}\label{lem: local embedding}
For $j = 1,2$, let $( M_j', g_j')$ and $\Omega$  be as in Proposition \ref{prop: same holomorphic bv} so that \eqref{eq: same holomorphic bv}  holds. Let $P_j : M_j'\to  \big( H_{\mathcal H(M'_j)}^s(\partial M_j') \big)'$ be the Poisson embedding with $\partial M' = \partial M_1' = \partial M_2'$. Then $P_1(p) = P_2(p)$ for all $p\in \overline\Omega$.
\end{lemma}
\begin{proof}
Due to \eqref{prop: same holomorphic bv}, we have that $\mathcal H(M_1') |\s _{\partial M_1'} = \mathcal H(M_2')|\s _{\partial M_2'}$ with $\partial M_1' = \partial M_2'$. This means that 
\begin{eqnarray}\label{eq: same boundary target set}
H_{\mathcal H(M'_1)}^s(\partial M_1') =  H_{\mathcal H(M'_2)}^s(\partial M_2').
\end{eqnarray}
Therefore, the ranges of $P_1$ and $P_2$ are both contained in 
$$
\big(H_{\mathcal H(M'_1)}^s(\partial M_1')\big)' =\big(  H_{\mathcal H(M'_2)}^s(\partial M_2')\big)'.
$$

We need to show that 
\begin{eqnarray}\label{eq: Poisson agrees}
P_1(p)f = P_2(p)f
\end{eqnarray}
for all $p\in\overline\Omega$ and $f\in H_{\mathcal H(M_1')}^s(\partial M_1')$. It suffices to show that this holds for all $f\in \mathcal H (M_1')|\s _{\partial M_1'} = \mathcal H(M_2')|\s _{\partial M_2'}$. For such $f$, denote by $u_f^1$ and $u_f^2$ its harmonic extensions on $(M_1', g_1', \partial M_1')$ and $(M_2', g_2', \partial M_2')$ respectively. Since by assumption $g_1' = g_2' =g'$ on $\Omega$, both $u_f^1$ and $u_f^2$ satisfy $\Delta_{g'} u_f^1 = \Delta_{g'} u_f^2 =0$ on $\Omega$. Since $f\in \mathcal H (M_1')|\s _{\partial M_1'} = \mathcal H(M_2')|\s _{\partial M_2'}$, we can use \eqref{prop: same holomorphic bv} to conclude that $\partial_\nu u_f^1 = \partial_\nu u_f^2$ on $\partial M'\supset \partial \Omega$. Unique continuation then allows us to conclude that $u_f^1 = u_f^2$ in $\Omega$. This is precisely \eqref{eq: Poisson agrees}.
\end{proof}

\begin{lemma}\label{lem: same range}
Under the assumption of Lemma \ref{lem: local embedding} we can conclude that $P_1(\overline{M_1'}) = P_2(\overline{M_2'})$.
\end{lemma}
\begin{proof} It suffices to show that $P_1(\overline{M_1'}) \subset P_2(\overline{M_2'})$ as the claim then follows by interchanging the roles of $M_1'$ and $M_2'$.  We already have by Lemma \ref{lem: local embedding} that, $P_1(\overline\Omega) \subset P_2(\overline{M_2'})$.  We set $B\subset \overline{M_1'}$ to be the subset of $\overline {M_1'}$ defined by:
\begin{multline}\label{eq: def of B}
B := \big\{ x\in \overline{M_1'} \mid {\rm there\ exists\ open\ set\ }U\subset \overline{M_1'}\  \\
{\rm such\ that } \ x\in U\ {\rm and}\ P_1(U) \subset P_2(\overline{M_2'})\big\}.
\end{multline}
The set $B$ is open by definition and contains $\Omega$ so it is not empty. Therefore, we only need to check that it is a closed subset of $\overline {M_1'}$.


To this end, let $J:= P_2^{-1} \circ P_1 : B \to J(B) \subset \overline {M_2'}$. 
Let $x_1\in \overline {M_1'}$ be a limiting point of a sequence $\{p_k\}\subset B$. The sequence $\{J(p_k)\}\subset J(B)$ then converges to $x_2\in \overline{M_2'}$, after passing to a subsequence and using compactness of $\overline{M_2'}$  if necessary. We need to show that $x_1\in B$.

The mapping $J$  by definition satisfies 
\begin{eqnarray}\label{eq: J is harmonic map}
u_f^2(J(p)) = u_f^1(p)
\end{eqnarray}
for all $p\in B$ and $f\in \mathcal H(M_1')|\s _{\partial M_1'} = \mathcal H(M_2')|\s _{\partial M_2'}$.


Lemma 2.2.4 of \cite{guillarmou2011calderon} allows us to choose $\hat f$ such that $du^2_{\hat f}(x_2) \neq 0$ so that $u^2_{\hat f} : U_2\to\mathbb R$ and its harmonic conjugate forms a holomorphic coordinate system $\Phi_{\hat f}^2: U_2\to \mathbb C$ on some open neighbourhood $U_2\subset \overline{M_2'}$ containing $x_2$. If necessarily, we may use Lemma 2.2.4 of \cite{guillarmou2011calderon} again to add a small perturbation to $\hat f$ so that it simultaneously satisfies $du^1_{\hat f}(x_1) \neq 0$. This allows us to choose an open neighbourhood $U_1\subset \overline{M_1'}$ containing $x_1$ on which $u_{\hat f}^1: U_1 \to \mathbb R$ and its harmonic conjugate forms a holomorphic coordinate system $\Phi_{\hat f}^1: U_1\to \mathbb C$. Redefine $U_1$ smaller, if necessary, so that $\Phi^1_{\hat f}(U_1)\subset \Phi^2_{\hat f}(U_2)$.

 Let $f\in \mathcal H(M_1')|\s _{\partial M_1'} = \mathcal H(M_2')|\s _{\partial M_2'}$. To conclude the proof, we show that on $\Phi^1_{\hat f}(U_1)\subset \Phi^2_{\hat f}(U_2)\subset \C$
 \begin{eqnarray}\label{eq: isothermal intertwines}
 u_f^1\circ (\Phi^1_{\hat f})^{-1}=u_f^2\circ (\Phi^2_{\hat f})^{-1}.
 \end{eqnarray}
 Since $u_f^1=u_f^2\circ J$ on $B$ for all $f\in \mathcal H(M_1')|\s _{\partial M_1'} = \mathcal H(M_2')|\s _{\partial M_2'}$, we can conclude that $J=(\Phi_{\hat f}^2)^{-1}\circ \Phi_{\hat f}^1$ on $B\cap U_1$. So \eqref{eq: isothermal intertwines} holds on the open set $\Phi^1_{\hat f}(U_1\cap B)$. 
 Since the coordinates $\Phi_{\hat f}^1$ and $\Phi_{\hat f}^2$ are isothermal, we have 
 \[
 c_1^{-1}\Delta_{\R^2}(u_f^1\circ (\Phi^1_{\hat f})^{-1})=0=c_2^{-1}\Delta_{\R^2}(u_f^2\circ (\Phi^2_{\hat f})^{-1}),
 \]
 on $\Phi^1_{\hat f}(U_1)$ for some positive smooth functions $c_1$ and $c_2$. 
 Thus $u_f^1\circ (\Phi^1_{\hat f})^{-1}$ and $u_f^2\circ (\Phi^2_{\hat f})^{-1}$ both satisfy the same Laplace equation in $\Phi^1_{\hat f}(U_1) \subset \mathbb C$. These functions agree on the open set $\Phi^1_{\hat f}(U_1\cap B)$. Thus they agree everywhere on $\Phi^1_{\hat f}(U_1)$ by unique continuation. Consequently, we have that $P_1(U_1) f\subset  P_2(U_2) f$ for all $f\in \mathcal H(M_1')|\s _{\partial M_1'} = \mathcal H(M_2')|\s _{\partial M_2'}$. By the fact that $ \mathcal H(M_j')|\s _{\partial M_j'}$ is dense in $H^s_{\mathcal H(M_j')}(\partial M_j')$, for $j = 1, 2$, we have that $P_1(U_1) \subset P_2(\overline{M_2'})$.
 This shows that $x_1$ is contained in an open neighbourhood $U_1\subset \overline{M_1'}$ for which $P_1(U_1) \subset P_2(\overline{M'_2})$. This means that $x_1\in B$ and thus $B$ is closed. This concludes the proof.
\end{proof}
We now set $J: \overline{M_1'} \to \overline {M_2'}$ by $J:= P_2^{-1} \circ P_1$ and show the following. 
\begin{lemma}
\label{lem: J is conformal diffeo} Under the assumption of Lemma \ref{lem: local embedding}, the map $J = P_2^{-1} \circ P_1 : \overline{M_1'} \to \overline{M_2'}$ is a conformal diffeomorphism.
\end{lemma}
\begin{proof}
We first show that $J:  \overline{M_1'} \to \overline {M_2'}$ is a bijection. Suppose that $ J(p_1) = J(p_2)$. This means $P_1(p_1) = P_1(p_2)$, which implies that $u_f(p_1) = u_f(p_2)$ for all $f\in \mathcal H(M_1')|\s _{\partial M_1'} = \mathcal H(M_2')|\s _{\partial M_2'}$. By \cite[Lemma 2.2.4]{guillarmou2011calderon} we have that elements of $\mathcal H(M_1')$ separate points on $M_1'$. This means that $p_1 = p_2$, and we have injectivity. Surjectivity is provided by Lemma \ref{lem: same range}.

We now show that 
\begin{equation}\label{eq: J is cont}
J\in C(\overline{M_1'}; \overline {M_2'}).
\end{equation}
To this end, let $p_k\to \hat p \in \overline{M_1'}$. By compactness, we may assume that the sequence $\{J(p_k)\}$ has a limit point $\hat q\in \overline{M_2'}$. Assume that 
\begin{equation}\label{eq: neq}
\hat q \neq J(\hat p).
\end{equation}
We have that for each point of the sequence, $u_f^2(J(p_k)) = u_f^1(p_k)$ for all $f\in  \mathcal H(M_1')|\s _{\partial M_1'} = \mathcal H(M_2')|\s _{\partial M_2'}$. Taking the limit $k\to \infty$ implies $u_f^2(\hat q) = u_f^1(\hat p)$ for all $f\in  \mathcal H(M_1')|\s _{\partial M_1'} = \mathcal H(M_2')|\s _{\partial M_2'}$. In other words, $P_2(\hat q) f = P_1(\hat p)f$ for all $f\in  H_{\mathcal H(M'_1)}^s(\partial M_1') =  H_{\mathcal H(M'_2)}^s(\partial M_2')$ and $s\geq 1/2$ by density. This means that $P_2(\hat q) = P_1(\hat p)$ as elements of $\big( H_{\mathcal H(M'_1)}^s(\partial M_1') \big)'= \big( H_{\mathcal H(M'_2)}^s(\partial M_2')\big)'$ or  $\hat q = (P_2^{-1} \circ P_1)(\hat p) = J(\hat p)$ which contradicts \eqref{eq: neq}. Therefore \eqref{eq: J is cont} holds. By swapping the index of the argument, $J^{-1} = P_1^{-1} \circ P_2$ is also continuous.

We now show that locally $J$ is a holomorphic map which is also an immersion. Since $J = \textrm{Id}$ on $\Omega \subset M_1' \cap M_2'$ by Lemma \ref{lem: local embedding}, we only need to check this for $\hat p$ in the interior of $M_1'$. By bijectivity, $\hat q := J(\hat p)$ is in the interior of $M_2'$. Let $f\in C^\infty(\partial M'_1; \mathbb C) = C^\infty(\partial M_2';\mathbb C)$ satisfy $\Re(f), \Im(f) \in \mathcal H(M_1')|\s _{\partial M_1'} = \mathcal H (M_2')|\s _{\partial M_2'}$ with its harmonic extension $u^2_f = u^2_{\Re(f)} + iu^2_{\Im(f)}$ holomorphic on $M_2'$. Note that by \eqref{eq: same holomorphic bv}, the harmonic extension of $f$ on $M_1'$ given by $u_f^1 = u_{\Re(f)}^1 + i u_{\Im(f)}^1$ is also holomorphic.

As in the proof of Lemma \ref{lem: same range}, by using Lemma 2.2.4 of \cite{guillarmou2011calderon} we may choose $\hat f$ 
so that $u^1_{\hat f} : U_1\to\mathbb C$ and $u^2_{\hat f} : U_2\to\mathbb C$ are holomorphic coordinate systems on some open neighbourhoods $U_1\subset M_1'$ and $U_2\subset M_2'$ with $\hat p \in U_1$ and $\hat q\in U_2$.
By definition, $u_{\hat f}^2(J(p)) = u_{\hat f}^1(p)$ for all $p\in M_1'$ which in particular holds for $p\in U_1$. Since $\hat q = J(\hat p) \in U_2$, we can use continuity statement \eqref{eq: J is cont} to choose open set $U_1$ sufficiently small so that $J(p) \subset U_2$ for all $p\in U_1$. As $u_{\hat f}^2$ is a coordinate system on $U_2$, we can write $J(p) = (u_{\hat f}^2)^{-1}\circ u_{\hat f}^1(p)$ for all $p\in U_1$. This shows that for all $\hat p$ in the interior of $M_1'$, there exist an open neighourhood $U_1\subset M_1'$ such that $J : U_1 \to M_2'$ is a holomorphic immersion. 

We have shown that $J: \overline {M_1'} \to \overline{M_2'}$ is a bijection that is also locally a holomorphic immersion. Standard coordinate calculation yields that $J$ is a conformal diffeomorphism. 
\end{proof}

\begin{proof}[End of proof of Theorem \ref{thm:unique2d}] Since $J$ is the identity on $\Omega$ by Lemma \ref{lem: local embedding}, we have so far shown that there is a smooth conformal mapping $J:M_1\to M_2$,
\[
 g_1=\lambda J^*g_2,
\]
where $\lambda$ is positive and smooth and $\lambda|_{\p M}=1$. 
Let us make a change of coordinates to pass from the equation $(\Delta_{g_2}+q_2)v_2$ on $M_2$ onto the manifold $(M_1,g_1)$. We denote
\[
\widetilde{q}_2= \lambda^{-1}J^*q_2.
\]
Let $f\in C^\infty(\p M)$ and let $v_2$ be the solution to
\[
\Delta_{g_2}v_2+q_2v_2=0 \text{ in }M_2,\text{ with } 
\]
with $v_2=f$  on  $\p M$. Let us denote
\[
\widetilde{v}_2:=J^*v_2=v_2\circ J. 
\]
By using the conformal invariance of the Laplacian in dimension $2$, we have 
\begin{multline*}
\Delta_{g_1}\widetilde{v}_2+\widetilde{q}_2\widetilde{v}_2=\Delta_{\lambda J^*g_2}\widetilde{v}_2 +\widetilde{q}_2\widetilde{v}_2 =\lambda^{-1}\Delta_{J^*g_2}\widetilde{v}_2 +\lambda^{-1}(J^*q_2)\widetilde{v}_2\\
= \lambda^{-1}J^*(\Delta_{g_2}v_2)+ \lambda^{-1}(J^*q_2)J^*v_2 =\lambda^{-1}J^*\left(\Delta_{g_2} v_2  + q_2v_2\right).
\end{multline*}
Since $v_2$ solves $\Delta_{g_2}v_2+ q_2v_2=0$ in $M_2$, we have that 
\begin{align}\label{conformal change}
	\begin{cases}
	\Delta_{g_1}\widetilde{v}_2+\widetilde{q}_2 \widetilde{v}_2=0 & \text{ in }M_1,\\
	\widetilde{v}_2=f & \text{ on }\p M,
	\end{cases}
\end{align}
since $\lambda|_{\p M}=1$ and $J|_{\p M}=\text{Id}$.

Let then $v_1$ solve 
\[
\Delta_{g_1}v_1+q_1v_1=0 \text{ in }M_1,\text{ with } 
\]
with $v_1=f$. We next show that 
\begin{equation}\label{DN_maps_on_single_manifold}
\p_{\nu_1} v_1=\p_{\nu_1}\widetilde{v}_2 \text{ on }\p M,
\end{equation}
where $\p_{\nu_1}$ is the normal vector field on $(M_1,g_1)$. 
Since $\Lambda_{g_1,q_1}=\Lambda_{g_2,q_2}$, we have 
\[
\p_{\nu_1}v_1=\p_{\nu_2}v_2=\nu_2\cdot dv_2=\nu_2\cdot d(v_2\circ J \circ J^{-1})=(J^{-1}_*\nu_2)\cdot d\s\widetilde{v}_2=\nu_1\cdot d\s\widetilde{v}_2=\p_{\nu_1} \widetilde{v}_2.
%
\]
Here 
we used the facts that $J:M_1 \to M_2$ is diffeomorphic conformal mapping with $J|_{\p M}=\text{Id}$ and $\lambda|_{\p M} =1$. Since $f\in C^\infty(\p M)$ was arbitrary, we have 
\begin{equation*}\label{transfd_DN_maps_agree}
\Lambda_{g_1,q_1} = \Lambda_{g_1,\tilde{q}_2}.
\end{equation*}
Since the metrics are now the same in both of the above DN maps, we have 
\[
 q_1=\tilde q_2=\lambda^{-1}J^*q_2
\]
by the main theorem of \cite{guillarmou2011calderon}. This finishes the proof.
\end{proof}

\section{Minimal surfaces}\label{Section 2}
In this section we collect facts and identities related to minimal surfaces that we will need in the proof of our second main theorem, Theorem \ref{thm:main}.
\subsection{Derivation of the minimal surface equation}
We start by deriving the minimal surface equation \eqref{eq:minimal_surface_general}. This is the equation a minimal surface satisfies in Fermi-coordinates. For possible future references, we consider in this section $n$-dimensional minimal surfaces embedded in $n+1$-dimensional Riemannian manifolds, $n\geq 2$.

We assume that the minimal surfaces are given as graphs over a submanifold $\Sigma$. Let us choose Fermi coordinates for a neighborhood $N=I\times \Sigma$ of $\Sigma$. Here $I$ is an interval in $\R$ containing $0$.  In Fermi coordinates the metric is of the form
\begin{equation*}
 \overline g=ds^2 +g_{ab}(x,s)dx^adx^b,
\end{equation*}
where $g(x,s)$ is a Riemannian metric on $\Sigma$ for all $s$.  $N=I\times \Sigma$ for all $s$. Here and below we also use Einstein summation over repeated indices.   We consider $g(\ccdot,s)=g_s$ as $1$-parameter family of Riemannian metrics on $\Sigma$. We have $\dim(N)=n+1$ and $\dim(\Sigma)=n$.

Let 
\[
 F(x)=(u(x),x)
\]
be the graph of a function $u:\Sigma\to I\subset \R$. The volume form of the graph $Y$ of $F$ is given by the determinant of the induced metric on the graph $Y$. Coordinates for $Y$ are given by coordinates on $\Sigma$ by
\[
 x\mapsto (u(x),x).
\]
Let $(x^k)_{k=1}^n$ be the above coordinates on $Y$ and let $\p_k$ be the corresponding coordinate vectors. Y .
For simplicity, let us assume that $(x^k)_{k=1}^n$ are global coordinates on $\Sigma$. The general case when
$\Sigma$ is covered with finitely many coordinate charts can be considered using a suitable partition
of unity. Let us also denote by $h_{jk}$ the induced metric on $Y$:    
\[
 h_{jk}(x)=\overline g_{F(x)}(F_*\p_{x_j},F_*\p_{x_k}),
\]
where $j,k=1,\ldots, n$. Here we have that 
\[
 F_*\p_{x_j}=DF_j^a\p_a,
\]
where $a=0,1,\ldots, n$. Note that if $a\neq 0$, then $DF_j^a=\delta_j^a$. We also have $DF_j^{\s 0}=\p_ju$. It follows that the induced metric on $Y$ reads
\begin{align*}
 h_{jk}&=DF_j^aDF_k^b g_{F(x)}(\p_a,\p_b)=DF_j^{\s 0}DF_k^{\s 0} \s\overline g_{00}|_{F(x)}+\sum_{\alpha,\beta=1}^nDF_j^\alpha DF_k^\beta \overline g_{F(x)}(\p_\alpha,\p_\beta) \\
 &=\p_ju(x)\p_ku(x)+g_{jk}|_{F(x)}=\p_ju(x)\p_ku(x)+g(x,u(x))_{jk}.
\end{align*}

The volume of $Y$ is 
\[
 \text{Vol}(Y)=\int_\Sigma \sqrt{\det(h)}dx^1\wedge\cdots \wedge dx^n.
\]
Using the formula for $h_{jk}$ we have that
\[
 \det(h)=\det\Big(\nabla u\otimes \nabla u +g(x,u(x))\Big)=\det(g(x,u(x)))\det\Big(I+(g(x,u)^{-1}\nabla u)\otimes \nabla u\Big).
\]
By \cite[Lemma 1.1]{Ding}, we have 
\[
 \det\Big(I+(g(x,u)^{-1}\nabla u)\otimes \nabla u\Big)=1+(g(x,u)^{-1}\nabla u) \cdot \nabla u=1+\abs{\nabla u}_{g(x,u(x))}^2.
\]
Finally, the volume of $Y$ equals
\[
 \text{Vol}(Y)=\int_\Sigma \sqrt{1+\abs{\nabla u}^2_{g(x,u)}}\det(g(x,u))^{1/2}dx^1\wedge\cdots \wedge dx^n.
\]

Let us then compute the minimal surface equation. For that, we consider a variation
\[
 Y(u+tv):=\{(u(x)+tv(x),x): x\in \Sigma\} \subset N
\]
of the surface $Y$, where $v:\Sigma \to \R$ is a smooth function. We denote the volume of $Y(u+tv)$ by $\text{Vol}(u+tv)$. Then, 
\begin{align*}
 &\frac{d}{dt}\Big|_{t=0}\text{Vol}(u+tv)=\frac 12 \int_\Sigma \frac{\det(g(x,u))^{1/2}}{\sqrt{1+\abs{\nabla u}^2_{g(x,u)}}}\frac{d}{dt}\Big|_{t=0}\left(\abs{\nabla (u+tv)}^2_{g(x,u+tv)}\right) \\
 &+\int_\Sigma \sqrt{1+\abs{\nabla u}^2_{g(x,u)}} \frac{d}{dt}\Big|_{t=0}\det(g(x,u+tv))^{1/2}.
\end{align*}
We have
\begin{align*}
 &\frac{d}{dt}\Big|_{t=0}\abs{\nabla (u+tv)}^2_{g(x,u+tv)}=2g(x,u)^{-1}(\nabla u, \nabla v)+v\p_s(g^{-1})(x,u)(\nabla u,\nabla u) \\
\end{align*}
and
\begin{align*}
 &\frac{d}{dt}\Big|_{t=0}\det(g(x,u+tv))^{1/2}=\frac{1}{2}\det(g(x,u))^{1/2}\text{Tr}(g(x,u)^{-1}v\p_sg(x,u)).
\end{align*}
Thus
\begin{multline}\label{eq:first_var}
 \frac{d}{dt}\Big|_{t=0}\text{Vol}(u+tv)=\frac 12 \int_\Sigma \frac{\det(g(x,u))^{1/2}}{\sqrt{1+\abs{\nabla u}^2_{g(x,u)}}}\Big(2g(x,u)^{-1}(\nabla u, \nabla v)+\p_s(g^{-1})(x,u)(\nabla u,\nabla u)v\Big) \\
 +\int_\Sigma \sqrt{1+\abs{\nabla u}^2_{g(x,u)}}\frac{1}{2}\det(g(x,u))^{1/2}\text{Tr}(g(x,u)^{-1}\p_sg(x,u))v.
\end{multline}

We recall that if $Y$ is a minimal surface (in the variational sense), then $t = 0$ is a critical point of the map $t \mapsto \text{Vol}(u + tv)$ for all functions $v$ that vanish on the boundary. We also denote 
\[
 f(u,\nabla u)=\frac{1}{2}\frac{1}{(1+\abs{\nabla u}^2_{g_u})^{1/2}}(\p_sg_u^{-1})(\nabla u,\nabla u)+\frac{1}{2}(1+\abs{\nabla u}^2_{g_u})^{1/2}\text{Tr}(g_u^{-1}\p_sg_u),
\]
where we shorthanded $g_u(x)=g(x,u(x))$ and $\p_sg_u=\p_sg_s|_{s=u}$ etc.
It follows by integrating by parts that the minimal surface equation is
\begin{equation*}
 -\frac{1}{\det(g_u)^{1/2}}\nabla\cdot \left( g_u^{-1}\frac{\det(g_u)^{1/2}}{\sqrt{1+\abs{\nabla u}^2_{g_u}}} \right)\nabla u + f(u,\nabla u) =0. 
\end{equation*}
This is \eqref{eq:minimal_surface_general}. Here we used that the boundary term arising from integrating by parts in \eqref{eq:first_var} is
\begin{equation}\label{eq:boundary_term}
 \int_{\p \Sigma} \frac{v}{\sqrt{1+\abs{\nabla u}^2_{g(x,u)}}}(\nabla u, \nu)_{g(x,u)}dS_{g(x,u)},
\end{equation}
which is zero since $v|_{\p \Sigma}=0$. Here $(\ccdot,\ccdot)_{g(x,u)}$ and $dS_{g(x,u)}$ denote the  inner product and the volume form on the $\p\Sigma$ induced by the Riemannian metric $g(x,u)$ respectively.

\subsubsection{Local well posedness}
We will compute in Section \ref{ss2.4} that the linearization of the minimal surface equation at the solution $u=0$ is 
\[
 \big(\Delta_g+\frac{1}{2}\p_s|_{s=0} \text{Tr}\s (g_s^{-1}\p_sg_s)\big)v=0.
\]
Using this, we formulate a well posedness result for the minimal surface equation.
\begin{proposition}[Local well posedness]\label{prop:local_well_posedness}
 Assume that $\text{Tr}\s (g_0^{-1}\p_sg_0)=0$ on $\Sigma$, and that $0$ is not a Dirichlet eigenvalue of the linear operator $\Delta_g+\frac{1}{2}\p_s|_{s=0} \text{Tr}\s (g_s^{-1}\p_sg_s)$ on $\Sigma$, the minimal surface equation \eqref{eq:minimal_surface_general} is well posed in the following sense. 
There exist positive constants $\delta$ and $C$ such that for any Dirichlet data $f\in U_\delta=\{f\in  C^{2,\alpha}(\partial\Sigma): \norm{f}_{C^{2,\alpha}(\p \Sigma)}\leq \delta\}$ there exists a solution $u\in C^{2,\alpha}(\Sigma)$ such that $||u||_{C^{2,\alpha}(\Sigma)}\leq C||f||_{C^{2,\alpha}(\p\Sigma)}$. The solution is unique among those  that satisfy $||u||_{C^{2,\alpha}(\Sigma)}\leq C\delta$. The correspondences $f\to u$ and $f\to\partial_\nu u|_{\p\Sigma}$ are $C^\infty$ as maps from $U_\delta$ into $C^{2,\alpha}(\Sigma)$ and $C^{1,\alpha}(\p\Sigma)$, respectively.
\end{proposition}

We omit the proof that follows along the arguments in \cite[Section 2]{LLLS2019inverse}, \cite[Appendix B]{CFKKU}. Note that the assumption $\text{Tr}(g_0^{-1}\p_sg_0)=0$ on $\Sigma$ is equivalent to the assumption that $u\equiv 0$ is a solution to the minimal surface equation \eqref{eq:minimal_surface_general}. Either of these assumptions are in turn equivalent to $\Sigma$ having zero mean curvature (see e.g. \cite[Section 2]{lassas2016calder} for formulas about mean curvature in Fermi-coordinates). The operator $\Delta_g+\frac{1}{2}\p_s|_{s=0} \text{Tr}(g_s^{-1}\p_sg_s)$ is the linearized operator of the minimal surface equation. In the context of minimal surfaces it is known as the stability operator (see \cite[Section 8]{colding2011course}) and the condition that $0$ is not its Dirichlet eigenvalue related to the stability of the minimal surface $\Sigma$. To see that $\Delta_g+\frac{1}{2}\p_s|_{s=0} \text{Tr}(g_s^{-1}\p_sg_s)$ is the stability operator, one can check that $\frac{1}{2}\p_s|_{s=0} \text{Tr}(g_s^{-1}\p_sg_s)=-\abs{A}^2-\text{\text{Ric}}(N,N)$, where $\abs{A}$ is the norm of the second fundamental form, $\text{Ric}$ is the Ricci curvature and $N$ normal vector field to $\Sigma$.

\subsection{Variation of area and the exterior problem}\label{sec:areas_exterior_problem}
Let us then explain how the DN map is related to areas of minimal surfaces. Here a technical issue arises: If a surface $\Sigma$ is embedded in a Riemannian manifold $N$, it might be that there is no positive $\eps >0$ such that the domain $\Sigma \times (-\eps,\eps)$ of Fermi coordinates is contained in $N$. It can also be that the boundary $(\partial \Sigma) \times ((-\eps,\eps)\setminus \{0\})$ intersects the interior of the unknown manifold $N$. Therefore, considering minimal surfaces as solutions to the
minimal surface equation \eqref{eq:minimal_surface_general} in Fermi-coordinates is technically complicated. To
address this technical issues, we consider the situation in an extended manifold $\widetilde N$ that contains $N$. We always assume $\widetilde N$ to be geodesically complete. 

We explain the relation of areas of minimal surfaces and the DN map using $\widetilde N$. This section is mainly for motivational purposes and we choose to keep the exposition short. We consider the following problem:
%
%
%
%

\noindent\textbf{The exterior problem:} Let $(\Sigma,g)$ be a minimal surface. Do $\p \Sigma$ and the volumes of minimal surfaces $\Sigma'$ in the exterior manifold $\widetilde N$, whose boundaries satisfy
$\partial \Sigma'\subset \widetilde N\setminus N$, determine the isometry type of $(\Sigma,g)$?


We say that a minimal surface $\Sigma$ extends properly to a minimal surface $\widetilde \Sigma$ if $\widetilde \Sigma$ is a minimal surface in $\widetilde N$, $\Sigma\subset \subset \widetilde \Sigma$ and $0$ is not an eigenvalue of the first linearization of the minimal surface equation \eqref{eq:minimal_surface_general} on $\widetilde \Sigma$. 
Let  $\delta_0>0$.
We say that $\mathcal S$   is  a family of smooth  minimal surface  deformations of the surface $\widetilde \Sigma$ when $\mathcal S$  consists of functions  $\mathcal F:(-\delta_0,\delta_0)\times  \widetilde \Sigma\to \widetilde N$ such that $\mathcal F(s,x)$, $(s,x)\in (-\delta_0,\delta_0)\times \widetilde \Sigma$
is in $C^\infty((-\delta_0,\delta_0);C^3( \widetilde \Sigma))$ and all surfaces $\widetilde \Sigma(s)=\{\mathcal F(s,x):\ x\in \widetilde \Sigma\}$ are minimal surfaces in $\overline N$.

For the exterior problem, we record the following corollary of Theorem \ref{thm:main}. 
In the proof of the corollary we are going to refer to Lemma \ref{lem:F_id_infty}, which can be found from Section \ref{sec:proof_of_main_thm}. 

To prove Corollary \ref{cor:exterior_problem}, we record in the next lemma how volumes of minimal surfaces are related to DN map in a domain of Fermi coordinates.

\begin{lemma} \label{lem:minimal surfaces and the DN map}
Let $(\Sigma, g)$ be a Riemannian manifold with a boundary $\p \Sigma$ and $ N= \Sigma\times (-\delta,\delta)$, $\delta >0$.  Assume that $N$ is equipped with a Riemannian metric of the form $\overline g = ds^2 +g_{ab}(x,s)dx^adx^b$ and that $\epsilon >0$ is so small that for all $h$ in the set
$\mathcal W_\epsilon=\{h\in C^\infty(\p \Sigma):\|h\|_{C^{2,\alpha}(\p  \Sigma)}<\eps\}$
the minimal surface equation \eqref{eq:minimal_surface_general} has a unique solution $u_h: \Sigma\to \R$ such that $\norm{u_h}_{C^{2,\alpha}(\Sigma)}\leq \eps$, and $\epsilon<\delta$. 

Let $Y(h)=\{(u_h(x),x):\ x\in \Sigma\}$ be the minimal surface with the boundary value $h$.
Then the boundary $\p \Sigma$, the metric
$\overline g|_{\p  \Sigma\times (-\delta,\delta)}$ and 
the volumes $ \text{Vol}\s (Y(h))$ of the minimal surfaces, given  for all
$h\in \mathcal W_\epsilon$, determine the values of Dirichlet-to-Neumann map for the equation \eqref{eq:minimal_surface_general}, that is, $\Lambda_{\overline g}(h)$
 for 
$h\in \mathcal W_\epsilon$. Moreover, Fr\'echet derivatives of the map $h\to  \text{Vol}\s (Y(h))$ 
to the order $k+1$ at $h=0$  determine the  Fr\'echet derivatives of the map $h\to \Lambda_{\overline g}(h)$
at $h=0$ to the order $k$.
%
%
\end{lemma}
\begin{proof}

Let us define a non-linear boundary map $\mathcal{N}_g$ by
\begin{eqnarray*}
\mathcal{N}_{\overline g} (u|_{\p \Sigma})= \frac 1{\sqrt{1+\abs{\nabla u}^2_{g(x,u)}}}
 (\nu, \nabla u)_{g(x,u)}\bigg|_{\p \Sigma},
\end{eqnarray*}
where $u$ is the solution of the minimal surface equation \eqref{eq:minimal_surface_general}. Let $u_h$ be the solution of the minimal surface equation \eqref{eq:minimal_surface_general} with
boundary value  $u_h|_{\p\Sigma}=h$. By the calculations leading to \eqref{eq:boundary_term}, we see that 
%
the  volumes $ \text{Vol}(Y(h))$ of the minimal surfaces $Y(h)$
and their minimal surface variations $ \text{Vol}(Y(h+tw))$ determine
\begin{eqnarray*}
 \frac{d}{dt}\Big|_{t=0}\text{Vol}(Y(h+tw))
 &=&  \int_{\p \Sigma} \frac{w}{\sqrt{1+\abs{\nabla u}^2_{g(x,u)}}}(\nabla u, \nu)_{g(x,u)}dS_{g(x,u)}.
\end{eqnarray*}
Since the boundary $\p \Sigma$ and the metric $ \overline g|_{\p \Sigma\times (-\delta,\delta)}$ are known, and as $w\in C^\infty(\p \Sigma)$ is arbitrary, we see that $ \frac{d}{dt}\big|_{t=0}\text{Vol}(Y(h+tw))
$ with varying values of $w$ determine $\mathcal{N}_{\overline g} (h)$ and consequently also $\Lambda_{\overline g}$. Moreover, the above formula yields that the Fr\'echet derivatives of the map $h\to  \text{Vol}\s (Y(h))$ 
to the order $k+1$ at $h=0$  determine the  Fr\'echet derivatives of the map $h\to \Lambda_{\overline g}(h)$
at $h=0$ to the order $k$.
\end{proof}

\begin{proof}[Proof of Corollary \ref{cor:exterior_problem}]
Since $\widetilde N$ is complete, there exists a neighbourhood $\widetilde U$ of $\widetilde\Sigma$ in $\widetilde N$,
such that in  $\widetilde U$ there are  Fermi coordinates associated to $\widetilde\Sigma$. That is, there exists the Fermi coordinate map $\tilde \psi:\widetilde U\to \mathcal{D}:=\widetilde\Sigma \times (-\eps,\eps)$, for some $\eps>0$. 

Let $s\mapsto h_s \in C^{2,\alpha}(\partial\widetilde\Sigma)$ be a smooth 1-parameter family of functions with $h_0 = 0$.
Then there is $0<\delta_1<\epsilon$  such that 
for $|s|<\delta_1$  the functions
$h_s$  satisfy 
$\|h_s\|_{C^{2,\alpha}(\p \widetilde  \Sigma)}<\eps$, where $\eps$  is such that the minimal surface equation has
 a unique small solution with the Dirichlet boundary value $h_s$, see Proposition \ref{prop:local_well_posedness}. Now set
 $$\widetilde \Sigma(s)=\tilde\psi^{-1}\left(\{(x, u_s(x)):\ x\in \widetilde \Sigma,\ u_s\ {\rm solves}\ \eqref{eq:minimal_surface_general}\ {\rm with\ }u_s\mid_{\partial\widetilde \Sigma} = h_s\}\right)$$

 Moreover, by  Lemma \ref{lem:minimal surfaces and the DN map},
the functions  $s\to h_s$  and $s\to 
\hbox{Vol}(\widetilde \Sigma(s))$, $s\in (-\delta_1,\delta_1)$ determine the
 derivatives of the map $s\to \Lambda_{\overline g}(h_s)$ 
at $h=0$ to any order $k$.
 By the proof of Theorem \ref{thm:main}, these data determines $(\widetilde \Sigma, g)$ up to an isometry that preserves $\p \widetilde \Sigma$. The isometry also preserves the second scalar fundamental form.

 What is left is to show that the isometry restricts to an isometry of $\Sigma$. For this we use a uniqueness result for isometries. It follows from the proof of Theorem \ref{thm:main} and Lemma \ref{lem:F_id_infty} that the isometry agrees with the identity map of $\widetilde \Sigma \setminus \Sigma$ to infinite order on $\p\widetilde\Sigma$. Since by assumption $\widetilde \Sigma\setminus \Sigma$ and $\overline g|_{\widetilde N\setminus N}$ are known, also the induced metric on $\widetilde \Sigma\setminus \Sigma$ is known. By the proof of \cite[Theorem 3.3]{lassas2019conformal}, isometries agreeing to high order order on an open subset of boundary are unique. (For this, apply the proof of \cite[Theorem 3.3]{lassas2019conformal} with harmonic coordinates instead of conformal harmonic coordinates.) By the above facts, it follows that the isometry restricted to $\widetilde \Sigma \setminus \Sigma$ is unique and consequently the identity.  
 By injectivity it then follows that the isometry maps $\Sigma$ onto itself. By continuity the isometry is the identity on $\p \Sigma$. This proves the claim.  
\end{proof}

\subsection{Higher order linearization}\label{ss2.4}
In this section we discuss the higher order linearization method for the minimal surface equation on $(\Sigma,g)$. While we assume in this paper that $\Sigma$ is $2$-dimensional, the computations in this section hold in higher dimensions as well. We will derive the corresponding integral identities for the first, second and third order linearizations. 

Later we will see that the first order linearization can be used to determine a $2$-dimensional minimal surface $\Sigma$ up to a conformal transformation.  Second order linearization will used to recover $\p_s|_{s=0}g(x,s)$ in Fermi coordinates of $\Sigma$. The third order linearization, together what can be recovered by considering the first and second linearizations, will be used to recover the related conformal factor. 

For $j=1,\ldots,4$ let $\eps_j\in \R$ and $f_j\in C^{2,\alpha}(\p M)$ for some $0<\alpha<1$. Let us denote $\eps=(\eps_1,\eps_2,\eps_3,\eps_4)$. We consider boundary values $f=f_\eps$  of the form 
\begin{align}\label{f_epsilon}
	f_\eps:=\displaystyle \sum_{j=1}^4\eps_j f_j
\end{align}
for the minimal surface equation
\begin{equation}\label{eq:minimal_surf_sec2}
\begin{aligned}
\begin{cases}
 -\frac{1}{\det(g_u)^{1/2}}\nabla\cdot \left( g_u^{-1}\frac{\det(g_u)^{1/2}}{\sqrt{1+\abs{\nabla u}^2_{g_u}}} \right)\nabla u + f(u,\nabla u) =0, 
&\text{ in } \Sigma,
\\
u= f_\eps 
&\text{ on }\p \Sigma,
\end{cases}
    \end{aligned}
\end{equation}
in Fermi coordinates associated to $\Sigma\subset N$. Recall that in Fermi-coordinates $\overline g=ds^2 +g_{ab}(x,s)dx^adx^b$.  We denote the one parameter family of metrics $g(\ccdot,s)$  on $\Sigma$ also simply by $g_s$ and
\[
 g_u(\ccdot)=g(\ccdot ,u(\ccdot)).
\]
Observe that $f_\eps \in U_\delta$ for sufficiently small $\epsilon$, so that by Proposition \ref{prop:local_well_posedness} the problem \eqref{eq:minimal_surf_sec2} is well-posed.

In this paper, we use the positive sign convention for the Laplacian. In local coordinates of $(\Sigma,g)$
\[
 \Delta_g u=-\nabla\cdot\nabla u= -\abs{g}^{1/2}\p_a\big(\abs{g}^{1/2}g^{ab}\p_b u\big).
\]
We denote by $\nabla$, $\Delta$, $\ccdot$ and $\abs{\ccdot}$ the corresponding covariant derivative, Laplacian, inner product and norm given be the metric $g$ if there is no change of confusion. We record the higher order linearizations of \eqref{eq:minimal_surf_sec2} at $\eps=0$, which corresponds to zero solution. We denote the first, second and third linearizations by
\begin{equation}\label{eq:not_for_lins}
 v^{j}:= \left.\frac{\p}{\p\epsilon_j}\right|_{\eps=0} u_f, \quad w^{jk}:= \left.\frac{\p^2}{\p\epsilon_j\p\epsilon_k}\right|_{\eps=0} u_f, \quad w^{jkl}:= \left.\frac{\p^3}{\p\epsilon_j\p\epsilon_k\p\epsilon_l}\right|_{\eps=0} u_f
\end{equation}
respectively. 

Next we compute the linearized equations that $v^{j}$, $w^{jk}$ and $w^{jkl}$ solve. For this, we set up some notation. 
Let us calculate the higher order linearization of the minimal surface equation. For this, for $l=1,2,\ldots$, we set up some notation. We denote
\begin{align*}
 d_u=\abs{g_u}^{1/2}, \quad h_u=\text{Tr}(g_u^{-1}\p_sg_s)|_{s=u} \ \text{ and } \ k_u&=g_u^{-1}.
 \end{align*}
and $d=\abs{g_u}^{1/2}|_{u=0}$ and $d^{(1)}:=\p_s\abs{g_s}^{1/2}|_{s=0}$. 
We also denote
\begin{align*}
h&=h_u|_{u=0}, &  h^{(l)}_u&=\left(\frac{\p}{\p s}\right)^l\Big|_{s=u(x)}\text{Tr}(g_s^{-1}\p_sg_s), &  h^{(l)}&=\left(\frac{\p}{\p s}\right)^l\Big|_{s=0}\text{Tr}(g_s^{-1}\p_sg_s), \\
 k(x)&=g^{-1}, &   k^{(l)}_u(x)&=\left(\frac{\p}{\p s}\right)^l\Big|_{s=u(x)}g_s^{-1}, &  k^{(l)}&=\left(\frac{\p}{\p s}\right)^l\Big|_{s=0}g_s^{-1} 
\end{align*}
for $l=1,2,3$. 
Note that $k_u,k,k^{(l)}_u, k^{(l)}$ are symmetric $2$-tensor fields on $\Sigma$. 
With these notations, the minimal surface equation \eqref{eq:minimal_surface_general} can be written in local coordinates on $\Sigma$ as 
\begin{multline*}
 0=-\frac{1}{\det(g_u)^{1/2}}\nabla\cdot \left( g_u^{-1}\frac{\det(g_u)^{1/2}}{\sqrt{1+\abs{\nabla u}^2_{g_u}}} \right)\nabla u + f(u,\nabla u)  \\
 =-d_u^{-1}\s \nabla\cdot\big(k_ud_u(1+\abs{\nabla u}^2_{g_u})^{-1/2}\nabla u\big)+f(u,\nabla u), 
\end{multline*}
where
\[
 f(u,\nabla u)=\frac{1}{2}\frac{1}{(1+\abs{\nabla u}^2_{g_u})^{1/2}}k_u^{(1)}(\nabla u,\nabla u)+\frac{1}{2}(1+\abs{\nabla u}^2_{g_u})^{1/2}h_u. 
\]
Here $\nabla$ is defined with respect to $\R^n$ metric. 
Since $u\equiv 0$ is a solution, we have $f(0,0)=0$, which implies 
\begin{equation}\label{eq_h_equiv_0}
h=\text{Tr}(g^{-1}\p_sg)|_{s=0}=0,
\end{equation}
which leads to 
\begin{eqnarray}\label{eq: k1 is trace free}
\text{Tr}_g(k^{(1)}):= \text{Tr}(gk^{(1)}) =0.
\end{eqnarray}
Note also that 
\begin{equation}\label{eq:first_der_of_d_zero}
 d^{(1)}=\Big(\frac{1}{2}\abs{g_u}^{1/2}\text{Tr}(g_u^{-1}\p_sg_s)\Big)\Big|_{u=0}=\frac{1}{2}dh=0.
\end{equation}

Next we write the minimal surface equation as 
\begin{equation*}
 0=
 P^uu+f(u,\nabla u ), 
\end{equation*}
where $P^u$ is the partial differential operator given by
\[
 P^uF=-\frac{1}{\det(g_u)^{1/2}}\nabla\cdot \left( g_u^{-1}\frac{\det(g_u)^{1/2}}{\sqrt{1+\abs{\nabla u}^2_{g_u}}} \right)\nabla F. 
\]
We then have
\[
 P:=P^{u}|_{u=0}=\Delta_g.
\]
We will see that factoring the minimal surface equation in this way is beneficial for calculations.  We also denote
\begin{align*}
 &P_{j}^u=\left(\frac{\p}{\p \eps_j}P^u\right), && P_{jk}^u=\left(\frac{\p^2}{\p \eps_k \p \eps_j}P^u\right),  &&& P_{jkl}^u=\left(\frac{\p^3}{\p \eps_k \p \eps_j\p\eps_l}P^u\right) \\ 
 &P^j=\left(\frac{\p}{\p \eps_j}P^u\right)\Big|_{\eps=0}, && P^{jk}=\left(\frac{\p^2}{\p \eps_k \p \eps_j}P^u\right)\Big|_{\eps=0},  &&& P^{jkl}=\left(\frac{\p^3}{\p \eps_k \p \eps_j\p\eps_l}P^u\right)\Big|_{\eps=0} \\
 &u_j=\left(\frac{\p}{\p \eps_j}u\right), && u_{jk}=\left(\frac{\p^2}{\p \eps_k \p \eps_j}u\right), &&& u_{jkl}=\left(\frac{\p^3}{\p \eps_k \p \eps_j\p \eps_l}u\right).
\end{align*}
With this notation we have
\[
 v^{j}=u_j|_{u=0}, \quad w^{jl}=u_{jl}|_{u=0}, \quad w^{jkl}=u_{jkl}|_{u=0}
\]
We will also use for convenience the physicists' short hand notation where indices without a specified value in brackets are symmetrised over. For example
\[
 P^u_{(j}u_{k)}=P^u_{j}u_{k}+P^u_{k}u_{j}
\]
and
\[
 P^u_{(jl}u_{k)}=P^u_{jl}u_{k}+P^u_{jk}u_{l}+P^u_{kl}u_{j}.
\]

With the above notations, the equation for the first linearization is
\[
 0=P_j^uu+P^uu_j+\p_{\eps_j}f.
\]
The equation for the second linearization is
\[
 0=P_{jk}^uu+P_{j}^uu_{k} +P_k^uu_j +P^u u_{jk}+\p_{\eps_{jk}}f =P^u_{jk}u+P^u_{(j}u_{k)}+P^u u_{jk}+\p_{\eps_{jk}}f,
 \]
and for the third linearization it is 
\begin{multline*}
 0=P_{jkl}^uu+P_{jk}^uu_l+P_{jl}^uu_{k}+P_{j}^uu_{kl}+P_{kl}^uu_{j}+P_k^uu_{jl}+P^u_{l}u_{jk}+P^u u_{jkl}+\p_{\eps_{jkl}}f\\
 =P^u_{jkl}u+P^u_{(jl}u_{k)}+P^u_{(j}u_{kl)}+P^u u_{jkl}+\p_{\eps_{jkl}}f.
\end{multline*}

Note that 
\[
P_j^uu|_{u=0}=P_{jk}^uu|_{u=0}=P_{jkl}^uu|_{u=0}=0.
\]
Thus evaluating the linearizations at $u=0$, equivalently at $\eps=(\eps_1,\eps_2,\eps_3)=0$, we obtain the equation for the first
\[
  0=Pv_j+\p_{\eps_j}|_{\eps=0}f=(\Delta_g+h^{(1)}/2)v_j 
\]
second
\[
  0=P^{(j}v^{k)}+P w^{jk}+\p_{\eps_{jl}}|_{\eps=0}f=(\Delta_g+h^{(1)}/2)w^{jk}+P^{(j}v^{k)}+k^{(1)}(\nabla v^j,\nabla v^k)+\frac{1}{2}h^{(2)}v^jv^k
\]
and third
\begin{multline*}
  0=P^{(jl}v^{k)}+P^{(j}w^{kl)}+P w^{jkl}+\p_{\eps_{jkl}}|_{\eps=0}f \\
  =(\Delta_g+h^{(1)}/2)w^{jkl}+P^{(jl}v^{k)}+P^{(j}w^{kl)}+k^{(2)}(\nabla v^{(j},\nabla v^k)v^{l)}+k^{(1)}(\nabla v^{(j},\nabla w^{kl)}) \\
  +\frac{1}{2}g(\nabla v^{(j},\nabla v^k)v^{l)}h^{(1)} +\frac{1}{2}w^{(jk}v^{l)}h^{(2)}+\frac{1}{2}v^{j}v^kv^{l}h^{(3)} 
\end{multline*}
linearizations at $\eps=0$. 

Here we used that
\begin{equation*}
 \p_{\eps_j}|_{\eps=0}f=(1+\abs{\nabla u}^2_{u})^{1/2}h_u=\frac{1}{2}\p_{\eps_j}|_{\eps=0}h_u=\frac{1}{2}h^{(1)}v^j,
\end{equation*}
since $h=0$ by \eqref{eq_h_equiv_0}. Similarly, we used
\begin{equation*}
 \p_{\eps_{jl}}|_{\eps=0}f=\frac{1}{2}\p_{\eps_{jl}}|_{\eps=0}h_u=\frac{1}{2}h^{(2)}v^jv^l.
\end{equation*}
We also used 
\begin{multline}\label{eq:third_deriv_of_f}
 \p_{\eps_{jkl}}|_{\eps=0}f
=k^{(2)}(\nabla v^{(j},\nabla v^k)v^{l)}+k^{(1)}(\nabla w^{(jk},\nabla v^{l)})+\frac{1}{2}g(\nabla v^{(j},\nabla v^k)v^{l)}h^{(1)} \\
 +\frac{1}{2}w^{(jk}v^{l)}h^{(2)}+\frac{1}{2}v^{j}v^{k}v^{l}h^{(3)}+\frac{1}{2}w^{jkl}h^{(1)}.
\end{multline}
We have placed the calculation behind \eqref{eq:third_deriv_of_f} in Appendix \ref{appx:calculations}.

By collecting the results of the above calculations, we obtain:
\begin{lemma}[Higher order linearizations]\label{lem:high_ord_lin} Let $f$ be as in \eqref{f_epsilon}, and for $j,k,l\in \{1,\ldots,4\}$ let $v^{j}$, $w^{jk}$ and $w^{jkl}$ be as in \eqref{eq:not_for_lins}.

 \noindent\textbf{(1)} The first linearization $v^{j}$ 
 satisfies the equation 
 \begin{equation}\label{linear_eq}
	\begin{aligned}
		\begin{cases}
			(\Delta_g+h^{(1)}/2)v^j=0 
			& \text{ in } \Sigma,
			\\
			v^{j}=f_j
			&\text{ on }\p \Sigma.
		\end{cases}
	\end{aligned}
\end{equation}
 \noindent\textbf{(2)} The second linearization $w^{jk}$ 
 satisfies 
 \begin{equation}\label{2nd_lin_eq}
 \begin{aligned}
		\begin{cases}
  (\Delta_g+h^{(1)}/2)w^{jk}+P^{(j}v^{k)}+k^{(1)}(\nabla v^{j},\nabla v^{k})+\frac{1}{2}h^{(2)}v^{j}v^{k}= 0& \text{ in } \Sigma \\
  w^{jk}=0
			&\text{ on }\p \Sigma.
  		\end{cases}
	\end{aligned}
 \end{equation}

 \noindent\textbf{(3)} The third linearization $w^{jkl}$ 
 satisfies the equation
  \begin{equation}\label{3rd_lin_eq}
	\begin{aligned}
		\begin{cases}
			(\Delta_g+h^{(1)}/2)w^{jkl}+P^{(jl}v^{k)}+P^{(j}w^{kl)} \\
	\qquad	+k^{(2)}(\nabla v^{(j},\nabla v^{k})v^{l)}+k^{(1)}(\nabla v^{(j},\nabla w^{kl)}) \\
  \qquad\quad+\frac{1}{2}g(\nabla v^{(j},\nabla v^{k})v^{l)}h^{(1)} +\frac{1}{2}w^{(jk}v^{l)}h^{(2)}+\frac{1}{2}v^{j}v^{k}v^{l}h^{(3)}=0 
			& \text{ in } \Sigma. \\
			w^{jkl}=0
			&\text{ on }\p \Sigma,
		\end{cases}
	\end{aligned}
\end{equation}
\end{lemma}

\subsection{Integral identities} \label{ss2.5}
Next we derive the corresponding integral identities for the second and third linearized equations.
		Let $v^m$ be solution to the first linearization
\[
 (\Delta_g+h^{(1)}/2)v^m=0
\]
with boundary value $f_m$.  Let us denote
 \[
  I_2=P^{(j}v^{k)}+k^{(1)}(\nabla v^{j},\nabla v^{k})+\frac{1}{2}h^{(2)}v^{j}v^{k}.
 \]
 Then, by \eqref{2nd_lin_eq} the second linearization $w^{jk}$ solves
 \[
  (\Delta_g+h^{(1)}/2)w^{jk}=-I_2.
 \]
Integration by parts yields
\begin{multline}\label{eq:2nd_lin_calc}
 \int_{\p \Sigma}f_m\p_\nu w^{jk}dS=\int_{\Sigma} v^m\Delta_{g} w^{jk}dV + \int_{\Sigma} \nabla v^m\cdot \nabla w^{jk}dV \\
 =-\frac{1}{2}\int_{\Sigma} v^mh^{(1)} w^{jk}dV-\int_\Sigma v^mI_2dV 
 + \int_{\p \Sigma}w^{jk}\p_\nu v^mdS- \int_{\Sigma}w^{jk} \Delta_{g} v^mdV \\
 =-\int_\Sigma v^m I_2dV. 
\end{multline}
In the last equality we used $v^m$ solves the first linearized equation, which canceled the integrals involving $w^{jk}$ over $\Sigma$. We also used that $w^{jk}|_{\p \Sigma}=0$.

Next we calculate  $-\int_\Sigma v^mI_2dV $. We have
\begin{equation*}
 P^j=-\frac{d}{d\eps_j}\Big|_{\eps=0}d_u^{-1}\nabla\cdot k_ud_u(1+\abs{\nabla u}^2_u)^{-1/2}\nabla =-d^{-1}\nabla\cdot k^{(1)}v^jd\nabla.
\end{equation*}
Here we used $d^{(1)}=0$. Using the $d^{-1}\nabla\cdot k d$ is the Riemannian divergence $\nabla^g\ccdot$ on $(\Sigma,g)$, we may write the operator $P^{j}$ as
\begin{equation}\label{eq:formula_for_Pj}
 P^j=-\nabla^g\cdot (v^jgk^{(1)}\nabla).
\end{equation}
%
Thus 
\begin{equation*}
 \int_\Sigma v^m P^jv^kdV
 =\int_\Sigma v^jk^{(1)}(\nabla v^m,\nabla v^k)dV - \int_{\p \Sigma} v^mv^j k^{(1)}(\nu,\nabla v^k)dS
\end{equation*}
and it follows that 
\begin{equation*}
 \int_\Sigma v^m P^{(j}v^{k)}dV
 =\int_\Sigma k^{(1)}(\nabla v^m,\nabla v^{(j})v^{k)}dV \\
 - \int_{\p \Sigma} v^m k^{(1)}(\nu,\nabla v^{(j})v^{k)}dS. 
\end{equation*}
By collecting the results of above calculations, we have proven:
\begin{lemma}[Integral identity for the second linearization]\label{Lem:Integral identity_2nd} Let $f_\eps$ be as in \eqref{f_epsilon}, and for $j,k,m\in \{1,2,3\}$ let $v^{j}$ and $w^{jk}$ be as in \eqref{eq:not_for_lins}.
	The integral identity for the second linearization is
	\begin{multline}\label{eq:second_integral_id}
	 \int_{\p \Sigma} f_m \s \p^2_{\eps_j \eps_k}\big|_{\epsilon=0} \Lambda (f_\epsilon) \, dS_g =\int_{\Sigma} v^m k^{(1)}(\nabla v^k,\nabla v^j\big)dV+\int_{\Sigma}v^k k^{(1)}(\nabla v^j,\nabla v^m\big)dV\\
 +\int_{\Sigma} v^j k^{(1)}(\nabla v^k,\nabla v^m\big)dV -\frac{1}{2}\int_{\Sigma} h^{(2)}v^jv^kv^mdV  \\
 - \int_{\p \Sigma} v^m k^{(1)}(\nu,\nabla v^{(j})v^{k)}dS,
	\end{multline}
	which holds for any $j,k,m\in \{1,2,3\}$.
\end{lemma}

Next we turn to deriving the integral identity for the third linearization, which in full generality will be pretty complicated. We first give the identity and then derive it. For this, we denote
\begin{equation}\label{eq:HRB}
 \begin{split}
 H&=-\int_{\Sigma}v^{(j}v^{k} k^{(2)}(\nabla v^{l)},\nabla v^m) dV +\int_{\Sigma}v^m g(\nabla (d^{-1}d^{(2)}v^{(j}v^k),\nabla v^{l)})dV \\
 &\qquad-\int_\Sigma v^m k^{(2)}(\nabla v^{(j},\nabla v^{k})v^{l)}dV-\frac{1}{2}\int_\Sigma v^mv^{j}v^{k}v^{l}h^{(3)}dV,\\
 R&=-\int_{\Sigma}w^{(jk}k^{(1)}(\nabla v^{l)},\nabla v^m)dV-\int_\Sigma k^{(1)}(\nabla v^m,\nabla v^{(j})w^{kl)}dV \\
 &\qquad-\int_{\Sigma}v^mk^{(1)}(\nabla v^{(j},\nabla w^{kl)})dV-\frac{1}{2}\int_\Sigma v^mg(\nabla v^{(j},\nabla v^{k})v^{l)}h^{(1)}dV  \\
 &\qquad \qquad -\frac{1}{2}\int_\Sigma v^mw^{(jk}v^{l)}h^{(2)}dV,   \\
 B&=\int_{\p \Sigma} v^mv^{(j}v^k g(\nu,k^{(2)}\nabla v^{l)})dS+\int_{\p \Sigma} v^mw^{(jk} g(\nu,k^{(1)}\nabla v^{l)})dS\\
&\qquad-\int_{\p \Sigma} v^mg(\nabla v^{(j},\nabla v^k)\p_\nu v^{l)}dS+\int_{\p \Sigma} v^mv^{(j} k^{(1)}(\nu,\nabla w^{kl)})dS.
\end{split}
\end{equation}

The integral identity for the third linearization then is:
\begin{lemma}[Integral identity for the third linearization]\label{Lem:Integral identity_3rd} Let $f_\eps$ be as in \eqref{f_epsilon}, and for $j,k,l,m\in \{1,\ldots,4\}$ let $v^{j}$, $w^{jk}$ and $w^{wjk}$ be as in \eqref{eq:not_for_lins}.

	The integral identity for the third linearization is
	\begin{multline}\label{eq:third_integral_id}
	 \int_{\p \Sigma} f_m \s \p^3_{\eps_j \eps_k\eps_l}\big|_{\epsilon=0} \Lambda (f_\epsilon) \, dS_g = \int_{\Sigma}g(\nabla v^{j},\nabla v^k)g(\nabla v^{l},\nabla v^m) dV  \\
	 +\int_{\Sigma}g(\nabla v^{j},\nabla v^l)g(\nabla v^{k},\nabla v^m) dV+\int_{\Sigma}g(\nabla v^{l},\nabla v^k)g(\nabla v^{j},\nabla v^m) dV \\
	 +H+R+B,
	\end{multline}
	which holds for any $j,k,m\in \{1,\ldots,4\}$. Here $H$, $R$ and $B$ are as in \eqref{eq:HRB}.
\end{lemma}
As we see, the integral identity in full generality is be pretty complicated. 
When we apply the integral identity in the proof of Theorem \ref{thm:main} two things will however happen: Firstly, the terms of $H$ will be of lower order when we use complex geometrical optics (CGOs) as the functions $v^j$. Secondly, the terms of $R$ will be recovered from the second linearization. Thus we will be able to neglect both the $H$ and $R$ terms in the proof. We will also be able to neglect the terms of $B$ as we assume boundary determination. We also note that the first three terms on the right hand side of \eqref{eq:third_integral_id} are not conformally invariant in conformal scalings. This fact will be used to recover a conformal factor in the proof.

The derivation of the integral identity is a straightforward but long calculation. A reader uninterested of this may jump directly to Section \ref{sec:proof_of_main_thm}. 
This time we denote
\begin{multline}\label{eq:I3}
 I_3=P^{(jl}v^{k)}+P^{(j}w^{kl)} +k^{(2)}(\nabla v^{(j},\nabla v^{k})v^{l)}+k^{(1)}(\nabla v^{(j},\nabla w^{kl)}) \\
  \qquad\quad+\frac{1}{2}g(\nabla v^{(j},\nabla v^{k})v^{l)}h^{(1)} +\frac{1}{2}w^{(jk}v^{l)}h^{(2)}+\frac{1}{2}v^{j}v^{k}v^{l}h^{(3)}
\end{multline}
so that $w^{jkl}$ solves
\[
 (\Delta_g+h^{(1)}/2)w^{jkl}=-I_3.
\]
By the same calculation as in \eqref{eq:2nd_lin_calc}, we obtain
\begin{equation}\label{eq:3rd_lin_calc}
 \int_{\p \Sigma}f_m\p_\nu w^{jkl}dS=-\int_\Sigma v^m I_3dV.
\end{equation}
We already calculated the formula for $P^j$ in \eqref{eq:formula_for_Pj}. The formula for $P^{jk}$ is 
\begin{multline}\label{eq:formula_for_Pjk}
 P^{jk}=
-k\nabla (d^{-1}d^{(2)}v^k v^l)\cdot \nabla-d^{-1}\nabla\cdot (k^{(2)}v^kv^ld \nabla)-d^{-1}\nabla\cdot (k^{(1)}w^{kl}d \nabla)
 \\
 +d^{-1}\nabla\cdot (kd g(\nabla v^k,\nabla v^l)\nabla)
\end{multline}
We have placed the calculation how the above formula is derived in Appendix \ref{appx:calculations}. 
Using again that $d^{-1}\nabla\cdot k d$ is the Riemannian divergence, we may write the operator $P^{jk}$ as
\begin{multline*}
  P^{jk}=-g(\nabla (d^{-1}d^{(2)}v^k v^l), \nabla \ccdot)
  -\nabla^g\cdot (v^jv^k gk^{(2)}\nabla) 
 -\nabla^g\cdot (w^{jk}gk^{(1)} \nabla) \\
  +\nabla^g\cdot (g(\nabla v^j,\nabla v^k)\nabla). 
\end{multline*}
 By integration by parts, we then have
\begin{multline}\label{eq:integral_P_jk}
 -\int_\Sigma v^mP^{(jk}v^{l)}=\int_{\Sigma}g(\nabla v^{(j},\nabla v^k)g(\nabla v^{l)},\nabla v^m) dV-\int_{\Sigma}v^{(j}v^{k} k^{(2)}(\nabla v^{l)},\nabla v^m) dV \\
  +\int_{\Sigma}v^m g(\nabla (d^{-1}d^{(2)}v^{(j}v^k),\nabla v^{l)})dV-\int_{\Sigma}w^{(jk}k^{(1)}(\nabla v^{l)},\nabla v^m)dV + \int_{\p \Sigma}B_1dS,
\end{multline}
where
\begin{multline}\label{eq:defs_for_B2}
B_1=\int_{\p \Sigma} v^mv^{(j}v^k g(\nu,k^{(2)}\nabla v^{l)})dS+\int_{\p \Sigma} v^mw^{(jk} g(\nu,k^{(1)}\nabla v^{l)})dS\\
-\int_{\p \Sigma} v^mg(\nabla v^{(j},\nabla v^k)\p_\nu v^{l)}dS.
\end{multline}
Here $\nu$ is the normal vector field on $\p \Sigma$ with respect to the metric $g$. Here also for example $g(\nu,k^{(1)}\nabla v^{l})=g_{ab}\nu^a(k^{(1)})^{bc}\p_cv^{l})$.

By \eqref{eq:formula_for_Pj} and integration by parts we have
\begin{equation}\label{eq:formula_for_Pj_2}
 -\int_\Sigma v^m P^{(j}w^{kl)}dV
 =-\int_\Sigma k^{(1)}(\nabla v^m,\nabla v^{(j})w^{kl)}dV \\
 + \int_{\p \Sigma} v^mv^{(j} k^{(1)}(\nu,\nabla w^{kl)})dS.
\end{equation}
Recall that
\begin{multline}\label{eq:I3_2}
   \int_\Sigma v^m I_3dV= \int_\Sigma v^mP^{(jk}v^{l)}+\int_\Sigma v^m P^{(j}w^{kl)} \\
 	+\int_\Sigma v^mk^{(2)}(\nabla v^{(j},\nabla v^{k})v^{l)}+\int_\Sigma v^mk^{(1)}(\nabla v^{(j},\nabla w^{kl)}) \\
   \qquad\quad+\frac{1}{2}\int_\Sigma v^mg(\nabla v^{(j},\nabla v^{k})v^{l)}h^{(1)} +\frac{1}{2}\int_\Sigma v^mw^{(jk}v^{l)}h^{(2)}+\frac{1}{2}\int_\Sigma v^mv^{j}v^{k}v^{l}h^{(3)}.
\end{multline}
With the notations for $H$, $R$ and $B$ in \eqref{eq:HRB}, and by \eqref{eq:integral_P_jk}, \eqref{eq:formula_for_Pj_2}  and  \eqref{eq:I3_2}, we then have 
\[
 -\int_\Sigma v^m I_3dV=\int_{\Sigma}g(\nabla v^{(j},\nabla v^k)g(\nabla v^{l)},\nabla v^m) dV+\int_\Sigma H dV+\int_\Sigma R dV+\int_{\p \Sigma} B dS.
\]
By \eqref{eq:3rd_lin_calc} we had 
\begin{equation*}
 \int_{\p \Sigma}f_m\p_\nu w^{jkl}dS=-\int_\Sigma v^m I_3dV.
\end{equation*}
We have obtained the integral identity for the third linearization \eqref{eq:third_integral_id}.

\section{Complex geometrics optics solutions with phases without critical points}\label{sec:CGOs}
In this section we construct complex geometrics optics solutions (CGOs). The construction is based on \cite{guillarmou2011identification}, but we also consider CGOs whose phase functions do not have critical points. Such CGOs have better decay estimates for the remainder terms in the small parameter $h>0$. Better decay will needed when we extract information from the integral identities corresponding to higher order linearizations of the minimal surface equation.

\subsection{Construction of the CGOs}

To begin, we assume that our Riemann surface $\Sigma$ is compactly contained in the open surface $M$ which in turn is compactly contained in the open surface $\widetilde M$ whose closure is a surface with boundary. 

For $q\in C^\infty_c(M)$, we recall the construction of \cite{guillarmou2011identification} of complex geometrics solutions to
\[
 (\Delta +q)u=0
\]
 on the Riemann surface $ M$. We keep the presentation brief and refer to \cite[Section 2]{guillarmou2011identification} for details. For a summary about the Riemannian differential calculus on Riemannian surfaces we refer to \cite{guillarmou2011calderon}. The complexified cotangent bundle $\C T^*\widetilde M$ has the splitting
\[
 \C T^*\widetilde M= T^*_{1,0}\widetilde M \oplus T^*_{0,1}\widetilde M
\]
determined by the eigenspaces of the Hodge star operator $\star$. 
In local  complex coordinate $z$ the space $T^*_{1,0}\widetilde M$ is spanned by $dz$ and $T^*_{0,1}\widetilde M$ is spanned by $d\ol z$.  The invariant definitions of $\p$ and $\op$ operators are given as
\[
 \op:=\pi_{0,1}d \text{ and } \p:=\pi_{1,0}d.
\]
By \cite[Proposition 2.1]{guillarmou2011identification} there is a right inverse $\op^{-1}$ for $\op$ in the sense that
\[
 \op\s \s  \op^{-1}\omega=\omega \text{ for all } \omega\in C_0^\infty(\widetilde M,T_{1,0}^*\widetilde M)
\]
such that $\op^{-1}$ is bounded from $L^p(T_{1,0}^*\widetilde M)$ to $W^{1,p}(\widetilde M)$ for any $p\in (1,\infty)$. We have analogous properties for 
\[
 \op^*=-i\star \p: W^{1,p}(T_{0,1}^*\widetilde M)\to L^p(\widetilde M),
\]
which is the Hermitean adjoint of $\op$. In complex coordinate $z$, the operator $\op^*$ is just $\p$.

We define 
\[
 \op_\psi^{-1}:=\mathcal{R}\op^{-1}e^{-2i\psi/h}\mathcal{E} \text{ and } \op_\psi^{*-1}:=\mathcal{R}\op^{*-1}e^{2i\psi/h}\mathcal{E},
\]
where $\mathcal{E} : W^{l,p}(M) \to W_c^{l,p}(\widetilde M)$ an extension operator for some $\widetilde M$ compactly containing $M$ and $\mathcal R$ is the restriction operator to $M$.
By \cite[Lemma 2.2 and Lemma 2.3]{guillarmou2011identification} we have for $p>2$ and $2\leq q\leq p$ the following estimates 
\begin{align}\label{eq:sobo_decay}
\begin{split}
 \norm{\overline \p_\psi^{-1}\omega}_{L^q(M)}&\leq C h^{1/q}\norm{\omega}_{W^{1,p}(M, T_{0,1}^*M)} \\
 \norm{\overline \p_\psi^{*-1}\omega}_{L^q(M)}&\leq C h^{1/q}\norm{\omega}_{W^{1,p}(M, T_{1,0}^*M)}.
 \end{split}
\end{align}
Moreover, there is $\eps>0$ such that 
\begin{align}\label{eq:sobo_decayL2}
\begin{split}
 \norm{\overline \p_\psi^{-1}\omega}_{L^2(M)}&\leq C h^{1/2+\eps}\norm{\omega}_{W^{1,p}(M, T_{0,1}^*M)} \\
 \norm{\overline \p_\psi^{*-1}\omega}_{L^2(M)}&\leq C h^{1/2+\eps}\norm{\omega}_{W^{1,p}(M, T_{1,0}^*M)}.
 \end{split}
\end{align}

Note that if $\psi$ has no critical points on $M$, we can obtain better estimates than \eqref{eq:sobo_decayL2} and \eqref{eq:sobo_decay}. Indeed, we have that for all $f\in C^\infty(M;T^*_{0,1}M)$,
\begin{eqnarray}\label{eq: no critical point expansion} \op_\psi^{-1} f = \mathcal{R}\op^{-1}e^{-2i\psi/h}\mathcal{E} f = \frac{ih}{2}\mathcal{R}\op^{-1}\left(\left( \bar\partial e^{-2i\psi/h}\right)\frac{\mathcal{E} f}{\bar\partial\psi}\right),
\end{eqnarray}
where for all $\omega\in \mathcal D'(M;T^*_{0,1}M)$, $\omega/\bar\partial\psi$ denotes the unique scalar function such that $\bar\partial \psi\s (\omega/\bar\partial\psi) = \omega$. In local conformal coordinates, $\bar\partial^{-1}$ has Schwartz kernel given by $(z-z')^{-1}$. Thus writing \eqref{eq: no critical point expansion} in local coordinates and integrating by parts yield
\begin{eqnarray}\label{eq: no critical point expansion II} \op_\psi^{-1} f =e^{-2i\psi/h} \frac{ih}{2} \frac{f}{\bar\partial\psi} + \frac{ih}{2}\mathcal{R}\op^{-1}\left( e^{-2i\psi/h}\bar\partial\left(\frac{\mathcal{E} f}{\bar\partial\psi}\right)\right).
\end{eqnarray}
Consequently, in the case when $\psi$ has no critical points Calder\'on-Zygmund gives the estimate
\begin{eqnarray}
\label{eq: no critical point expansion III}
\|\op_\psi^{-1} f\|_p \leq Ch \|f\|_{W^{1,p}}
\end{eqnarray}
for $p\in(1,\infty)$.

The complex geometrics optics solutions (CGOs) we use are of the same form as in \cite{guillarmou2011identification}, but our notation is slightly different. The CGOs with holomorphic phase are  of the form 
\begin{eqnarray}
\label{CGO v}
 v=e^{\Phi/h}(a+r_h),
\end{eqnarray}
where $\Phi=\phi+i\psi$ is holomorphic Morse function and $a$ is a holomorphic function defined on $\widetilde M$, cf. \cite[Proposition 3.1, Eq. 21]{guillarmou2011identification}. 
In particular, $v$ solves $(\Delta + q)v=0$ if and only if
\begin{align}\label{eq:equation_for_rh}
\begin{split}
r_h=-\overline{\p}_\psi^{-1}s_h,\  (1+\op_\psi^{*-1}q\overline\p_\psi^{-1})s_h=\op_\psi^{*-1}(qa).
 \end{split}
\end{align}
To obtain an explicit expression for $s_h$, we introduce
\begin{eqnarray}\label{def: Th}
 T_h:=-\op_\psi^{*-1}q\overline\p_\psi^{-1}.
\end{eqnarray}
and note that its formal transpose is given by 
\begin{eqnarray}
\label{adjoint of T}
T_h^t = e^{-2i\psi/h} \p^{*-1} qe^{2i\psi/h} \p^{-1}
\end{eqnarray}
(We note that $T_h$ is not exactly the operator $S_h$ in \cite{guillarmou2011identification}. The reason for this is that we use holomorphic amplitude $a$, whereas the construction in \cite{guillarmou2011identification} uses antiholomorphic amplitude.)
By the proof of \cite[Lemma 3.1]{guillarmou2011identification}, we have
\begin{eqnarray}
\label{Th norm estimate}
 \norm{T_h}_{L^r\to L^r}=O(h^{1/r}) \text{ and } \norm{T_h}_{L^2\to L^2}=O(h^{1/2-\eps}), 
\end{eqnarray}
for any $0<\eps<1/2$.

Thus we may derive an explicit expression for $s_h$ from \eqref{eq:equation_for_rh} via Neumann series as
\begin{eqnarray}
\label{eq: def of sh}
 s_h=-\sum_{j=0}^\infty T_h^j\op_\psi^{*-1}(qa). 
 \end{eqnarray}
Consequently
\begin{eqnarray}
\label{eq: def of rh}
 r_h=-\overline{\p}_\psi^{-1}s_h=-  \overline{\p}_\psi^{-1}\sum_{j=0}^\infty T_h^j\op_\psi^{*-1}(qa).
\end{eqnarray}
By the estimates in the proof of \cite[Lemma 3.2]{guillarmou2011identification} there is $\eps>0$ such that 
\begin{equation}\label{eq:rh_L2_estim}
 \norm{s_h}_{L^2}+\norm{r_h}_{L^2}=O(h^{1/2+\eps}),
\end{equation}
and also for $r>2$
\begin{equation}\label{eq:rh_Lr_estim}
 \norm{s_h}_{L^r}+\norm{r_h}_{L^r}=O(h^{1/r}).
\end{equation}
{\color{black} Note that we can improve on \eqref{eq:rh_Lr_estim} by {\color{black}H\"older's inequality}: for any $2<r<\infty$, choose $r'>r$ and $0<\theta<1$ such that $1/r=\theta/2+(1-\theta)/r'$. Then
\begin{equation}
||s_h||_{L^r}\leq ||s_h||_{L^2}^\theta ||s_h||_{L^{r'}}^{1-\theta}=O(h^{1/r+\theta\epsilon}),
\end{equation}
with an identical estimate holding for $r_h$. It follows that 
\begin{equation}\label{eq:rh_Lr_interpolation_estimate}
 \norm{s_h}_{L^r}+\norm{r_h}_{L^r}=O(h^{1/r+\eps_r}).
\end{equation}
for any $r\geq 2$.
}{\color{black}
Making use of the fact that $\partial\bar\partial^{-1}$ is a Calder\'on-Zygmund operator and apply the standard Calder\'on-Zygmund estimates, we also have that
\begin{equation}\label{eq:CZ_rh_norms}
||r_h||_{L^r},  ||\partial r_h||_{L^r}, ||\bar\partial r_h||_{L^r}=O(h^{1/r+\epsilon_r}),
\end{equation}
for any $r\in[2,\infty)$ and $\epsilon_r >0$ depending on $r$.
}
This completes the construction of the CGO $v$ in \eqref{CGO v}.

Next we construct another CGOs with antiholomorphic phase. By changing the sign of $\Phi$ and taking complex conjugate we also have solutions  to $(\Delta + q) \tilde v =0$ of the form
\[
\tilde v = e^{-\overline \Phi/h}(\overline a+\tilde r_h),
\]
where 
\begin{equation}\label{eq: form of tilderh}
 \tilde r_h=-\p_\psi^{-1}\sum_{j=0}^\infty \widetilde T_h^j(\p_{\psi}^{*-1} (q \bar a))
\end{equation}
with 

\[
 \p_\psi^{-1}:=\mathcal{R}\p^{-1}e^{-2i\psi/h}\mathcal{E} \text{ and } \p_\psi^{*-1}:=\mathcal{R}\p^{*-1}e^{2i\psi/h}\mathcal{E},
\]
and $\widetilde T_h =  -\p_{\psi}^{*-1}q\p_\psi^{-1}$. As expected, the remainder $\tilde r_h$ satisfies the same estimates as $r_h$.
Note that $\widetilde T_h$ is not the same as $\overline T_h$ since we also changed the sign of $\Phi$.

Observe that in the above construction if $\Phi$ has no critical points, we may apply the better estimate \eqref{eq: no critical point expansion III}  throughout the construction to get
\begin{eqnarray}
\label{eq: no critical point estimate}
\|r_h\|_p +\|s_h\|_p+ \|d r_h\|_p \leq Ch
\end{eqnarray}
for all $p\in (1,\infty)$.

\subsection{Choices for the solutions}\label{subsection:choices_of_solutions}
We choose the solution we are going to use in our inverse problem. Let $\Psi$ be a holomorphic function on $\Sigma$ without critical points, as constructed in the main Theorem of \cite{gunning1967immersion}. Let also $\Phi$ be a holomorphic Morse function on $\Sigma$ as constructed in \cite{guillarmou2010calderon}. Let $z_0\in \Sigma$ be a fixed critical point of $\Phi$. We may assume without loss of generality that $\Phi(z_0) = \Psi(z_0) = 0$. By multiplying $\Phi$ by a small enough positive constant, we may also assume without loss of generality that both $\Psi + \Phi$ and $-\Psi + \Phi$ have nonvanishing derivatives everywhere on $\Sigma$.

For later reference, we choose three CGO solutions to be used in the context of the second linearization as 
\begin{align}\label{eq:sols_for_second}
 \begin{split}
 v^1&=e^{\Theta_1/h}(a+r_h), \\
 v^2&=e^{\Theta_2/h}(a+r_h'), \\
 v^3&=e^{\Theta_3/h}(\overline a+\tilde r_h),
 \end{split}
\end{align}
where $\Theta_1 := \Psi + \Phi$, $\Theta_2 = -\Psi + \Phi$, and $\Theta_3 = -2\bar\Phi$ so that in local holomorphic coordinates centered at $z_0$ 
\begin{align*}
\begin{split}
\Theta_1&=z + O(z^2) \ \text{ and } \ \Theta_2=-z +O(z^2) \quad \text{are holomorphic}, \\
\Theta_3&=-\z^2+O(\z^3) \text{ is antiholomorphic}.
\end{split}
\end{align*}
Observe that $\tilde r_h$ satisfies the estimates given by \eqref{eq:CZ_rh_norms}, while $r_h$ and $r'_h$ satisfy the better estimates given by \eqref{eq: no critical point estimate}.

In both of the above cases, we take that the holomorphic function $a$ such that it has the expansion 
    \[
    a(z)=1+O(z^N)
    \]
    in local holomorphic coordinates centered at $z_0$, $N$ large, and that $a$ vanishes to high order at any other critical point.
This is possible by \cite[Lemma 2.2.4]{guillarmou2011calderon}.

We also choose four CGO solutions to be used in the context of the third linearization  as
\begin{align}\label{eq:sols_for_3rd}
    \begin{split}
    v^1&=e^{\Phi_1/h}(a+r_h),\quad \ \, v^3 = e^{\Phi_3/h}(a+ r'_h) \\
    v^2 &= e^{-\overline \Phi_2/h}(\overline a+\tilde r_h),\quad v^4= e^{-\overline \Phi_4/h}(\overline a+ \tilde r_h'),
    \end{split}
    \end{align}
where now $\Phi_1, \Phi_2= -\Psi + \Phi$ and $\Phi_3,\Phi_4 = \Psi + \Phi$.
Observe that as none of the holomorphic phase functions $\Phi_1,\ldots,\Phi_4$  contain critical points, the remainder terms in these solutions all satisfy the better estimates \eqref{eq: no critical point estimate}. 

\section{Asymptotic analysis of the integral identities}\label{Section_4}
In this section we analyse the integral identities corresponding to the second and third linearizations.
Let $z_0\in \Sigma$ and 
 \begin{align}\label{eq:3sols}
 \begin{split}
 v^1&=e^{\Theta_1/h}(a+r_h), \\
 v^2&=e^{\Theta_2/h}(a+r_h'), \\
v^3& = e^{\Theta_3/h}(\bar a+\tilde r_h)
 \end{split}
\end{align}
be the CGO solutions to the first linearized equation as described in Section \ref{subsection:choices_of_solutions}. Thus we have $\Theta_1 := \Psi + \Phi$, $\Theta_2 = -\Psi + \Phi$ and $\Theta_3 = -2\bar\Phi$, where $\Phi$ and $\Psi$ are holomorphic, $\Phi$ is Morse with a critical point at $z_0$ and $\Theta_1$ and $\Theta_2$ do not have critical points anywhere. In  local holomorphic coordinates centered at $z_0$
\begin{align*}
\begin{split}
\Theta_1=z + O(z^2) , \quad \Theta_2&=-z +O(z^2) \ \text{ and } \ \Theta_3=-\z^2+O(\z^3),
\end{split}
\end{align*}
and $a(z)=1+O(z^N)$ and $a$ vanishes to high order at other possible critical points of $\Phi$.
The correction term 
$\tilde r_h$ satisfies the estimates given by \eqref{eq:CZ_rh_norms} while $r_h$ and $r'_h$ satisfy the better estimates given by \eqref{eq: no critical point estimate}. We summarize the estimates for the correction terms by
\begin{eqnarray}
\label{eq: rh, rh', estimates}
\|r_h\|_r + \|dr_h\|_r + \|r_h'\|_r+ \|dr_h'\|_r \leq Ch
\end{eqnarray}
and
\begin{eqnarray}
\label{eq: tilderh estimates}
\|\tilde r_h\|_r+ \|d\tilde r_h\|_r \leq Ch^{1/r + \epsilon_r}
\end{eqnarray}
for all $r\in (1,\infty)$ and for some $\epsilon_r>0$.

Below, a $(2,0)$-tensor field $K=(K^{ab})$ is trace free (with respect to $g$) if $\tr(gK)=0$.  We will also consider the real linear extension of the tensor $K(\ccdot,\ccdot) : T^*M\times T^*M\to \mathbb R$ to a tensor on the complexified cotangent bundle $K(\ccdot, \ccdot):\mathbb CT^*M \times\mathbb CT^*M\to \C$ via 
 \[
K(V_1+iW_1, V_2+iW_2) = K(V_1, V_2) - K(W_1,W_2) + i K(V_1, W_2) + iK(W_1, V_2)  
 \]
for $V_1, V_2, W_1, W_2\in TM$. 
Direct computation in local holomorphic coordinates shows that
\begin{multline*}
 K(\nabla v,\nabla u\big)=(K^{11}+K^{22})(\p u \p v+\p u \p v)\\
 +(K^{11}-K^{22}+i(K^{12}+K^{21}))\p u \p v + (K^{11}-K^{22}-i(K^{12}+K^{21}))\op u \op v.
\end{multline*}
Note that $K^{11}+K^{22}=0$ if $\tr(gK)=0$, since in local holomorphic coordinates $g=c\s I_{2\times 2}$ for some function $c$.  
We use the notation
\begin{align}\label{eq:Kcomplex}
 \begin{split}
 K(\p u, \p v)&:=(K^{11}-K^{22}+i(K^{12}+K^{21}))\p u \p v \\
 K(\op u, \op v)&:=  (K^{11}-K^{22}-i(K^{12}+K^{21}))\op u \op v.
  \end{split}
\end{align}
We also denote the exterior derivative $dv$ of a function $v$  by $\nabla v$.

\begin{proposition}\label{prop:asymptotics_for_the_2nd_lin}
 Let $(\Sigma,g)$ be a Riemannian surface with boundary. 
 Assume that $K$ is a trace free $(2,0)$-tensor that vanishes to infinite order on the boundary. Let $v^1,v^2$ and $ v^3$ be as in \eqref{eq:3sols}. We have the expansion 
 \begin{multline}
 \label{eq: remainder estimates}
  \int_{\Sigma} v^1 K(\nabla v^2,\nabla v^3\big)dV+\int_{\Sigma}v^2 K(\nabla v^1,\nabla v^3\big)dV+\int_{\Sigma} v^3 K(\nabla v^1,\nabla v^2\big)dV= \\
    \frac{1}{h^2}\int_\Sigma e^{4i\psi/h}\s \overline a K\left( \partial \Theta_1a,  \partial \Theta_2 a\right) dV + o(1/h).
 \end{multline}

\end{proposition}

Let us then present the other result of this section. Let $\Psi$, $\Phi$ and $a$ be as before with $\Phi$ having a critical point $z_0\in \Sigma$. 
Let $\Phi_1, \Phi_2= -\Psi + \Phi$ and $\Phi_3,\Phi_4 = \Psi + \Phi$ and construct CGO solutions to the linearized equation as in \eqref{eq:sols_for_3rd}:

\begin{align}\label{eq:sols_for_4th linearization}
    \begin{split}
    v^1&=e^{\Phi_1/h}(a+r_h),\quad \ \, v^3 = e^{\Phi_3/h}(a+ r'_h) \\
    v^2 &= e^{-\overline \Phi_2/h}(\overline a+\tilde r_h),\quad v^4= e^{-\overline \Phi_4/h}(\overline a+ \tilde r_h'),
    \end{split}
    \end{align}
    Here in local holomorphic coordinates centered at $z_0$
    \begin{align*}
\begin{split}
\Phi_1,\Phi_2=z + O(z^2) \  \text{ and } \ \Phi_3,\Phi_4&=-z +O(z^2).
\end{split}
\end{align*}
Observe that as none of the holomorphic phase functions contain critical points, the remainder term in these solutions all satisfy estimate \eqref{eq: rh, rh', estimates}.

%
    
\begin{proposition}\label{prop:asymptotics_for_the_3rd_lin}
 Let $(\Sigma,g)$ be a Riemannian surface with boundary.  Assume that $Q\in C^\infty(\Sigma)$ vanishes to infinite order on the boundary. Let $v^1,\ldots, v^4$ be as in \eqref{eq:sols_for_4th linearization}. We have the expansion 
  \begin{multline*}
\int_{\Sigma}Q\Big[g(\nabla v^1, \nabla v^2)g(\nabla v^3, \nabla  v^4)+g(\nabla v^1, \nabla v^3)g(\nabla v^2, \nabla  v^4) \\
    +g(\nabla v^2, \nabla v^3)g(\nabla v^1, \nabla  v^4)\Big]dV=Ch^{-3}Q(P)+o(h^{-3}),
 \end{multline*}
 where $C\neq 0$ is a constant.
\end{proposition}

Next we prove Proposition \ref{prop:asymptotics_for_the_2nd_lin}. We will omit the proof of Proposition \ref{prop:asymptotics_for_the_3rd_lin} as the method used is only a minor modification of the one used in the proof of Proposition \ref{prop:asymptotics_for_the_2nd_lin}. 
We however note that since all the CGOs involved in Proposition \ref{prop:asymptotics_for_the_3rd_lin} have remainder terms decaying according to the better estimates \eqref{eq: rh, rh', estimates}, the proof is easier than that of Proposition \ref{prop:asymptotics_for_the_2nd_lin}.


We start the proof of Proposition \ref{prop:asymptotics_for_the_2nd_lin} by noting that by the fact that $K$ is a trace-free $(2,0)$-tensor field and our choice of extension to $\mathbb CT^*M\times \mathbb CT^*M$, we have that 
\begin{eqnarray}
\label{eq: no cross terms}
K(\nabla u, \nabla v) = K(\partial u ,\partial v) + K(\overline\partial u, \overline\partial v).
\end{eqnarray}

\begin{lemma}
\label{lem: first and second term} With $v^1$, $v^2$ and $v^3$ chosen as in \eqref{eq:3sols} and $K(\ccdot, \ccdot)$  trace-free, 
we have that
\[
\int_\Sigma v^1K(\nabla v^2, \nabla v^3)dV = o(1/h), \quad \int_\Sigma v^2K(\nabla v^1, \nabla v^3)dV = o(1/h).
\]
\end{lemma}
We prove Lemma \ref{lem: first and second term} later. Assuming the lemma, we only need to consider the term 
\[
\int_\Sigma v^3K(\nabla v^1, \nabla v^2)dV 
\]
 in \eqref{eq: remainder estimates}. To this end, using \eqref{eq: no cross terms} we get
\begin{eqnarray}
\label{eq: main term of prop 4.1}
\int_\Sigma v^3K(\nabla v^1, \nabla v^2)dV &= &\int_\Sigma v^3K(\partial v^1, \partial v^2)dV + \int_\Sigma v^3K(\overline\partial v^1, \overline\partial v^2)dV\\\nonumber
&=& \int_\Sigma v^3K(\partial v^1, \partial v^2)dV  + O(1).
\end{eqnarray}
The last estimate comes from direct calculation and the estimates on $\tilde r_h$,  $r_h$, and $r'_h$ given by \eqref{eq:CZ_rh_norms} and \eqref{eq: no critical point estimate}. Substituting the expression 
\[
 \partial v^1 =e^{\Theta_1/h}\left( \frac{1}{h} \partial \Theta_1(a+r_h) + \partial(a+r_h)\right)
\]
 and the similar expression for $\partial v^2$ into \eqref{eq: main term of prop 4.1}, we see that 
\begin{multline}
\label{eq: main term of prop 4.1 II}
\int_\Sigma v^3K(\nabla v^1, \nabla v^2)dV \\
= \int_\Sigma e^{4i\psi/h}(\overline a+ \tilde r_h) K\left( \frac{ \partial \Theta_1}{h}(a+r_h) + \partial(a+r_h),  \frac{\partial \Theta_2}{h} (a+r_h') + \partial(a+r_h')\right) dV + O(1) \\
= \frac{1}{h^2}\int_\Sigma e^{4i\psi/h}(\overline a+ \tilde r_h) K\left( \partial \Theta_1(a+r_h),  \partial \Theta_2 (a+r_h')\right) dV\\
+ \frac{1}{h} \int_\Sigma e^{4i\psi/h}\overline a\left(K\left(\partial\Theta_1(a+r_h), (a+r'_h)\right)+ K\left((a+r_h), \partial\Theta_2(a+r'_h)\right)\right)dV + O(1).
\end{multline}
Using either stationary phase or the estimates we have for $r'_h$, $r_h$, and $\tilde r_h$ given in \eqref{eq: rh, rh', estimates} and \eqref{eq: tilderh estimates}, we  conclude that the second to last term of \eqref{eq: main term of prop 4.1 II} is $o(1/h)$. So we have that
\begin{multline}\label{eq: main term of prop 4.1 III}
\int_\Sigma v^3K(\nabla v^1, \nabla v^2)dV \\
= \frac{1}{h^2}\int_\Sigma e^{4i\psi/h}(\overline a+ \tilde r_h) K\left( \partial \Theta_1(a+r_h),  \partial \Theta_2 (a+r_h')\right) dV+ o(1/h).
\end{multline}
Expanding the right side of \eqref{eq: main term of prop 4.1 III} we get that
\begin{equation}\label{eq: main term of prop 4.1 IVa}
\begin{split}
\int_\Sigma v^3K(\nabla v^1, \nabla v^2)dV 
&= \frac{1}{h^2}\int_\Sigma e^{4i\psi/h}\overline a K\left( \partial \Theta_1(a+r_h),  \partial \Theta_2 (a+r_h')\right) dV\\
&+  \frac{1}{h^2}\int_\Sigma e^{4i\psi/h}\tilde r_h K\left( \partial \Theta_1a,  \partial \Theta_2 a\right) dV\\
&+  \frac{1}{h^2}\int_\Sigma e^{4i\psi/h}\tilde r_h K\left( \partial \Theta_1(a+r_h),  \partial \Theta_2 r_h'\right) dV\\
&+ \frac{1}{h^2}\int_\Sigma e^{4i\psi/h}\tilde r_h K\left( \partial \Theta_1r_h,  \partial \Theta_2 a\right) dV+ o(1/h).\\
\end{split}
\end{equation}
The third and fourth terms of \eqref{eq: main term of prop 4.1 IVa} are $o(1/h)$ by the estimates \eqref{eq: rh, rh', estimates} on $r_h$, $r_h'$  and \eqref{eq: tilderh estimates} on $\tilde r_h$ . So we get 
\begin{eqnarray}
\label{eq: main term of prop 4.1 IV}
\int_\Sigma v^3K(\nabla v^1, \nabla v^2)dV 
&=& \frac{1}{h^2}\int_\Sigma e^{4i\psi/h}\overline a K\left( \partial \Theta_1(a+r_h),  \partial \Theta_2 (a+r_h')\right) dV\\\nonumber
&+&  \frac{1}{h^2}\int_\Sigma e^{4i\psi/h}\tilde r_h K\left( \partial \Theta_1a,  \partial \Theta_2 a\right) dV+ o(1/h).
\end{eqnarray}
Observe that 
\[
\tilde r_h = -\p_{2\psi}^{-1}\sum_{j=0}^\infty \widetilde T_h^j(\p_{2\psi}^{*-1} (q \overline a)) 
\]
is given by \eqref{eq: form of tilderh} with $2\psi$ in place of $\psi$. Thus we may apply the adjoint of $\partial^{-1}$ in the definition of $\tilde r_h$ in the second to last term of \eqref{eq: main term of prop 4.1 IV}   to get
\begin{multline}\label{eq: main term of prop 4.1 V}
\int_\Sigma v^3K(\nabla v^1, \nabla v^2)dV 
= \frac{1}{h^2}\int_\Sigma e^{4i\psi/h}\overline a K\left( \partial \Theta_1(a+r_h),  \partial \Theta_2 (a+r_h')\right) dV\\
+  \frac{1}{h^2}\int_\Sigma e^{4i\psi/h}\left(\sum_{j=0}^\infty \widetilde T_h^j(\p_{2\psi}^{*-1} (q \overline a))\right) \partial^{-1}\left(e^{4i\psi/h} K\left( \partial \Theta_1a,  \partial \Theta_2 a\right) \right)dV+ o(1/h).
\end{multline}
Observe that since both $\sum_{j=0}^\infty \widetilde T_h^j(\p_{2\psi}^{*-1} (q \overline a))$ and $\partial^{-1}\left(e^{4i\psi/h} K\left( \partial \Theta_1a,  \partial \Theta_2 a\right) \right)$ are both $O_{L^2}(h^{1/2+\epsilon})$ by \eqref{eq:sobo_decayL2}, we have that the second to last term of \eqref{eq: main term of prop 4.1 V} is $o(1/h)$. Therefore,
\begin{equation}\label{eq: main term of prop 4.1 VI}
\int_\Sigma v^3K(\nabla v^1, \nabla v^2)dV = \frac{1}{h^2}\int_\Sigma e^{4i\psi/h}\overline a K\left( \partial \Theta_1(a+r_h),  \partial \Theta_2 (a+r_h')\right)dV+o(1/h).
\end{equation}
Expanding  \eqref{eq: main term of prop 4.1 VI} yields
\begin{equation}\label{eq: main term of prop 4.1 VII}
\begin{split}
\int_\Sigma v^3K(\nabla v^1, \nabla v^2)dV 
&= \frac{1}{h^2}\int_\Sigma e^{4i\psi/h}\overline a K\left( \partial \Theta_1a,  \partial \Theta_2 a\right) dV\\
&+ \frac{1}{h^2}\int_\Sigma e^{4i\psi/h}\overline a K\left( \partial \Theta_1a,  \partial \Theta_2 r_h'\right) dV\\
& +  \frac{1}{h^2}\int_\Sigma e^{4i\psi/h}\overline a K\left( \partial \Theta_1r_h,  \partial \Theta_2a\right) dV\\
&+  \frac{1}{h^2}\int_\Sigma e^{4i\psi/h}\overline a K\left( \partial \Theta_1r'_h,  \partial \Theta_2 r_h'\right)dV+ o(1/h).
\end{split}
\end{equation}
The last term is $o(1/h)$ by simply using the $L^2$ estimates of $r_h$ and $r_h'$ given by \eqref{eq: no critical point estimate}. The third term can be dealt with by observing that $r_h = -\partial_{\psi_1}^{-1} s_h$ where $\psi_1 = {\rm Im}(\Theta_1)$. Note that since $\Theta_1$ is constructed to have no critical points, both $r_h$ and $s_h$ satisfy \eqref{eq: no critical point estimate}. So writing the third term on the right hand side of \eqref{eq: main term of prop 4.1 VII} as 
\begin{multline*}
\frac{1}{h^2}\int_\Sigma e^{4i\psi/h}\overline a K\left( \partial \Theta_1r_h,  \partial \Theta_2a\right) dV \\
=  \frac{1}{h^2}\int_\Sigma e^{2i\psi_1/h} s_h \partial^{-1}\left(e^{4i\psi/h}\overline a K\left( \partial \Theta_1,  \partial \Theta_2a\right) \right)dV = o(1/h).
\end{multline*}
The second term on the right hand side  of \eqref{eq: main term of prop 4.1 VII} can be estimated exactly the same way to be $o(1/h)$. Therefore we have that 
$$\int_\Sigma v^3K(\nabla v^1, \nabla v^2)dV 
= \frac{1}{h^2}\int_\Sigma e^{4i\psi/h}\overline a K\left( \partial \Theta_1a,  \partial \Theta_2 a\right) dV + o(1/h).$$
Combining this with Lemma \ref{lem: first and second term}, we have proven Proposition \ref{prop:asymptotics_for_the_2nd_lin}.\qed\\
It remains to prove Lemma \ref{lem: first and second term}, which we do now:

\begin{proof}[Proof of Lemma \ref{lem: first and second term}]  We only prove the first asymptotic of the lemma since the proof for the second one is exactly identical. We begin by using the fact that $K(\ccdot, \ccdot)$ is traceless, in the sense that $\tr(gK)=0$, 
to conclude that
\begin{eqnarray}\label{eq: first term I}
 \int_\Sigma v^1K(\nabla v^2, \nabla v^3)dV &=& \int_{\Sigma} v^1 K(\partial v^2, \partial v^3)dV + \int_{\Sigma} v^1 K(\overline\partial v^2, \overline\partial v^3)dV\\\nonumber
&=& h^{-1}\int_{\Sigma} e^{4i\psi/h} (a+r_h) K\left(\partial\Theta_2 (a+r'_h), \partial \tilde r_h\right) dV\\\nonumber
&+& \int_{\Sigma} e^{4i\psi/h} (a+r_h) K\left( \partial(a+r'_h), \partial \tilde r_h\right) dV\\\nonumber
&+&h^{-1}\int_{\Sigma} e^{4i\psi/h} (a+r_h) K\left( \overline\partial r'_h, \overline\partial\Theta_2(\overline a+ \tilde r_h) \right) dV\\\nonumber
&+& \int_{\Sigma}e^{4i\psi/h} (a+r_h) K\left( \overline\partial r_h', \overline\partial (\overline a + \tilde r_h)\right) dV.
 \end{eqnarray}
 Observe that either by stationary phase or by estimates \eqref{eq: rh, rh', estimates} and \eqref{eq: tilderh estimates}, every term on the right side of \eqref{eq: first term I} is $o(1/h)$.  This concludes the proof.
\end{proof}

\section{Proof of Theorem \ref{thm:main}}\label{sec:proof_of_main_thm} 
We prove Theorem \ref{thm:main}.  For this let us assume that $(\Sigma_1,g_1)$ and $(\Sigma_2,g_2)$ are $2$-dimensional Riemannian manifolds with mutual boundary $\p\Sigma$, that there are $1$-parameter families of Riemannian metrics $g_a(\ccdot,s)$ on $\Sigma_\beta$, that $0$ is a solution to minimal surface equation \eqref{eq:minimal_surface_general} on both $\Sigma_\beta$ and and that $\Lambda_{g_1}=\Lambda_{g_2}$. Here and below the index $\beta=1,2$ refers to quantities on the manifold $\Sigma_\beta$. 
The proof will be achieved by considering the first three linearizations of the minimal surface equation. 
\subsection{Step 1: First linearization}

By Lemma \ref{lem:high_ord_lin}, the first linearization of the minimal surface equation is  
\begin{equation}\label{eq:first_lin_proof}
	\begin{aligned}
		\begin{cases}
			(\Delta_{g_\beta}+h_{\beta}^{(1)}/2)v_\beta^{j}=0 
			& \text{ in } \Sigma_\beta,
			\\
			v_\beta^{j}=f_j
			&\text{ on }\p \Sigma.
		\end{cases}
	\end{aligned}
\end{equation}
 The index $j=1,\ldots,4$, refers to the boundary value $f_j\in C^\infty(\p \Sigma)$, which we assume to be the same for both $\beta=1,2$.  As remarked in Section \ref{Section 2}, the DN map of the minimal surface equation 
has   Frech\'et derivatives of all orders. 
 It thus follows from $\Lambda_{g_1}=\Lambda_{g_2}$, that the DN maps of the first linearizations \eqref{eq:first_lin_proof} also agree. 
%
%
 Consequently,  by Theorem \ref{thm:unique2d} there is a conformal mapping
\[
 F:\Sigma_1\to\Sigma_2, \quad F^*g_2=cg_1,
\]
which also satisfies $F|_{\p \Sigma}=\text{Id}$, and 
\[
 h_{1}^{(1)}=c F^*h_{2}^{(1)}
\]
The conformal factor $c$ satisfies $c|_{\p \Sigma}=1$. Here $F^*$ denotes the pullback by $F$. 
Especially $F$ preserves Cauchy data of functions.  
We mention that this is the only part of the proof where we use that $\Sigma_1$ and $\Sigma_2$ are diffeomorphic to $\Omega$ by boundary fixing maps.

We use $F$ to transform our analysis to $(\Sigma_1,g_1)$ as follows. We simplify our notation by setting
\[
 v^{j}:=v^{j}_1, \quad \widetilde v^{j}:=v_2^{j}\circ F
\]
and 
\[
  w^{jk}:=w^{jk}_1, \quad \widetilde w^{jk}_{2}=w_2^{jk}\circ F \ \text{ and } \ w^{jkl}:=w^{jkl}_1, \quad \widetilde w^{jkl}_{2}=w_2^{jkl}\circ F.
\]
and 
\[
 h^{(1)}:=h_{1}^{(1)} \quad \widetilde h^{(1)}:=h_{2}^{(1)}\circ F.
\]
We use similar notation for other quantities throughout out the proof. We also note that
 \[
  d\widetilde V_2:= F^*dV_2=c^{\dim(\Sigma_1)/2}dV_1=cdV_1.
 \]
We will also denote $dV:=dV_1$ and $\Sigma=\Sigma_1$.

Since $F$ is the identity on the boundary and
\begin{multline}\label{eq:morphism_of_sols}
 (\Delta_{g_1}+h_1^{(1)}/2)\widetilde v^{j}=\Delta_{c^{-1}\s F^*g_2}\widetilde v^{j}+(h_1^{(1)}/2)\widetilde v^{j}=c\s \Delta_{F^*g_2}v_2^{j}\circ F+(c F^*h_{2}^{(1)}/2)v_2^{j}\circ F \\
 =c\s F^*((\Delta_{g_2}+h_2^{(1)}/2)v_2)=0,
\end{multline}
we see that both $v^{j}_1$ and $v_2^{j}\circ F$ satisfy the same equation $(\Delta_{g_1}+h_1^{(1)}/2)v=0$ on $\Sigma_1$ and have the same boundary value $f_j$. Here we used that the Laplace-Beltrami operator in dimension $2$ is conformally invariant. By uniqueness of solutions, we thus have
\[
 v^{j}=\widetilde v^{j}.
\]

We record the following lemma:
\begin{lemma}\label{lem:F_id_infty}
 Let $(\Sigma_1,g_1)$ and $(\Sigma_2,g_2)$ be compact Riemannian surfaces with a mutual boundary $\p \Sigma$. Let also $L_\beta=\Delta_{g_\beta}+q_\beta$, $\beta=1,2$, be such that the corresponding DN maps of $L_1$ and $L_2$ are defined and satisfy $\Lambda_{L_1}=\Lambda_{L_2}$. Assume that $F:\Sigma_1\to \Sigma_2$ is a morphism of solutions in the sense that
\[
 F^*U_1=U_2, 
\]
for all $U_\beta$, $\beta=1,2$, that solve
\[
 L_\beta U_\beta=0, \quad U_\beta|_{\p \Sigma} =f
\]
for some $f\in C^\infty(\p \Sigma)$. Assume also that the formal Taylor series of $g_\beta$ and $q_\beta$ agree in $g_\beta$-boundary normal coordinates (that is, values and the normal derivatives in the boundary normal coordinates agree in all orders). 

Then the coordinate representation of $F$ in boundary normal coordinates is the identity mapping to infinite order on the boundary: $F(x^1,\ldots, x^n)=\text{Id} + O(x_n^\infty)$.
\end{lemma}
We have placed the proof of the lemma in Appendix \ref{appx:proof_of_lemma}. We remark for possible future references that the lemma generalizes to higher dimensions and for more general second order elliptic operators.

Recall that $h^{(1)}=\p_u|_{u=0} \text{Tr}(g_u^{-1}\p_sg_u)$. Thus, by the assumptions of Theorem \ref{thm:main} that we are proving, the formal Taylor series of $h_1^{(1)}$ and $h_2^{(1)}$ agree in associated boundary normal coordinates. The same holds for $g_1$ and $g_2$. By \eqref{eq:morphism_of_sols}, the mapping $F$ is morphism of solutions. Consequently, by the above lemma 
\begin{equation}\label{eq:F_id_to_infinite}
F=\text{Id} + O(x_n^\infty) \text{ in boundary normal coordinates on } \{x_n=0\}.
\end{equation}
Especially we have $F_*\p_{\nu_1}=\p_{\nu_2}$ and the functions 
$v_1$ and $v_2$ agree to infinite order in associated boundary normal coordinates.


%
%
%

\subsection{Step 2: Second linearization} Next we will first show by using the second linearization that
\[
 k_1^{(1)}=c\s \tilde k_2^{(1)},
\]
where
\[
		 \tilde k_2^{(1)}=F^*k_2^{(1)}.
		\]
Here $F$ is the conformal mapping of the last section. We note that $k_\beta^{(1)}$, $\beta =1,2$, are symmetric $2$-tensor fields on $\Sigma_\beta$. Then we will also show that $h_1^{(2)}=F^*h_2^{(2)}$ and $w_1^{(jk)}=F^*w_2^{(2)}$.

By Lemma \ref{Lem:Integral identity_2nd}, the associated integral identity   of the second linearization is
\begin{multline}\label{second_integral_id_proof}
	 \int_{\p \Sigma} f_m \s \p^2_{\eps_j \eps_k}\big|_{\epsilon=0} \Lambda_{g_\beta} (f_\epsilon) \, dS_\beta =\int_{\Sigma_\beta} v^m k_\beta^{(1)}(\nabla v_\beta^k,\nabla v_\beta^j\big)dV_\beta+\int_{\Sigma_\beta}v_\beta^k k_\beta^{(1)}(\nabla v_\beta^j,\nabla v_\beta^m\big)dV_\beta\\
 +\int_{\Sigma_\beta} v_\beta^j k_\beta^{(1)}(\nabla v_\beta^k,\nabla v_\beta^m\big)dV_\beta -\frac{1}{2}\int_{\Sigma_\beta} h_\beta^{(2)}v_\beta^jv_\beta^kv_\beta^mdV_\beta  \\
 - \int_{\p \Sigma} v_\beta^m k_\beta^{(1)}(\nu_\beta,\nabla v_\beta^{(j})v_\beta^{k)}dS_\beta,
 \end{multline}
 where $\beta=1,2$ refers to quantities on $\Sigma_\beta$ and $j,k,m\in \{1,\ldots,4\}$.    
 We change variables in the terms on the right-hand side of \eqref{second_integral_id_proof}  with $\beta = 2$ by using the conformal mapping $F$. 
		Using also $ v^{j}=v_2^{j}\circ F$,
		we have
		\begin{multline*}
		 \int_{\Sigma_2} v_2^m k_2^{(1)}(\nabla v_2^k,\nabla v_2^j\big)dV_2=\int_{\Sigma_1} F^*\left(v_2^m k_2^{(1)}(\nabla v_2^k,\nabla v_2^j)\right)F^*dV_2 \\
		 =\int_{\Sigma_1} c v^m \tilde k^{(1)}(\nabla v^k,\nabla v^j)dV.
		\end{multline*}
		Here the factor $c$ is due to the volume form. We also used the fact that $k_2^{(1)}$ is a tensor on $\Sigma_2$ so that $F^*$ acts on it by the usual coordinate transformation rules.  We have similarly for other terms in the \eqref{second_integral_id_proof}. Consequently, using the assumptions $k_1^{(1)}|_{\p \Sigma}=k_2^{(1)}|_{\p \Sigma}$ and $\Lambda_{g_1} (f_\epsilon) =\Lambda_{g_2} (f_\epsilon)$, the fact that $c|_{\p \Sigma}=1$,  and subtracting the identities \eqref{second_integral_id_proof} on $\Sigma=\Sigma_1$ and $\Sigma_2$ give 
\begin{multline}\label{eq:2nd_ord_recovery_proof}
 0=\int_{\Sigma} v^1 (k^{(1)}_1-c\tilde k^{(1)}_2)(\nabla v^2,\nabla v^3\big)dV+\int_{\Sigma}v^k (k^{(1)}_1-c\tilde k^{(1)}_2)(\nabla v^1,\nabla v^3\big)dV_1\\
 +\int_{\Sigma} v^3 (k^{(1)}_1-c\tilde k^{(1)}_2)(\nabla v^1,\nabla v^2\big)dV -\frac{1}{2}\int_{\Sigma} (h^{(2)}_1-c\tilde h^{(2)}_2)v^1v^2v^3dV.
\end{multline}
Here we also chose $(j,k,l)=(1,2,3)$. We note that since $F^*g_2=c g_1$ and $g_1=g_2$ and $F$ is the identity to infinite order on $\p \Sigma$ in associated boundary normal coordinates by \eqref{eq:F_id_to_infinite}, we have that 
\begin{align}\label{eq:c_is_1_to_infinite}
\begin{split}
 c&=1 \text{ to infinite order on } \p\Sigma \\
 k^{(1)}_1&=c\tilde k^{(1)}_2 \text{ to infinite order on } \p\Sigma.
 \end{split}
\end{align}

We define a $(2,0)$-tensor
\[
 K:=k^{(1)}_1-c\tilde k^{(1)}_2,
\]
on $\Sigma$. The tensor $K$ vanishes to infinite order at $\partial \Sigma$ by \eqref{eq:c_is_1_to_infinite}.
We note also that $K$ is trace free with respect to $g=g_1$. Indeed, since we are dealing with minimal surfaces $(\Sigma_\beta,g_\beta)$, $\beta=1,2$, which have the property that the mean curvature vanishes we first obtain
\[
 0=h_\beta=\tr(g_\beta^{-1}\p_sg_\beta)=-\tr(g_\beta k^{(1)}_\beta)
\]
on $\Sigma_\beta$,  see \eqref{eq_h_equiv_0}. Consequently, since $k_2^{(2)}$ is a tensor, $F^*g_2=cg_1$ and the trace of the $(1,1)\s $-tensor $g_2k_2^{(1)}$ is invariant, we have
\begin{multline}
\tr(g_1K)= \tr(g_1(k^{(1)}_1-c\tilde k^{(1)}_2)) =\tr(g^{-1}\p_sg_1)-\tr(c^{-1}(F^*g_2)cF^* k_2^{(1)}) \\
= -\tr(g_2k_2^{(1)})|_F=0.
\end{multline}
%
%
%

Let $z_0\in \Sigma$ be fixed and let us choose the solutions as in Section \ref{subsection:choices_of_solutions}
 \begin{equation}\label{eq:sols_for_second_proof}
 \begin{split}
v^1&=e^{\Theta_1/h}(a+r_h), \\
 v^2&=e^{\Theta_2/h}(a+r_h'), \\
 v^3&=e^{\Theta_3/h}(\overline a+\tilde r_h),
 \end{split}
\end{equation}
where $\Theta_1 := \Psi + \Phi$, $\Theta_2 = -\Psi + \Phi$, and $\Theta_3 = -2\bar\Phi$ so that in local holomorphic coordinates centered at $z_0$ 
\begin{align*}
\begin{split}
\Theta_1&=z + O(z^2) , \quad \Theta_2=-z +O(z^2), \quad \text{is holomorphic} \\
\Theta_3&=-\z^2+O(\s\z^3) \text{ is antiholomorphic}.
\end{split}
\end{align*}
Observe that $\tilde r_h$ satisfies the estimates given by \eqref{eq:CZ_rh_norms} while $r_h$ and $r'_h$ satisfy the better estimates given by \eqref{eq: no critical point estimate}.

The holomorphic amplitude $a$ is $1$ to high order at $z_0$ and $0$ to high order at other possible critical points of the phase functions $\Theta_j$. 
With these solutions, we note that on the right hand side of \eqref{eq:2nd_ord_recovery_proof} 
the last term is of size $O(h)$ by stationary phase and the $L^p$ estimates for $r_h$, $r_h'$ and $\tilde r_h$, see \eqref{eq:rh_L2_estim}--\eqref{eq:rh_Lr_estim}. Therefore \eqref{eq:2nd_ord_recovery_proof} becomes
\begin{multline}\label{eq:2nd_ord_recovery_proof II}
 0=\int_{\Sigma} v^1 K(\nabla v^2,\nabla v^3\big)dV+\int_{\Sigma}v^2 K(\nabla v^1,\nabla v^3\big)dV_1
 +\int_{\Sigma} v^3K(\nabla v^1,\nabla v^2\big)dV  + O(h).
\end{multline}

Since $K$ is symmetric and trace free, we may apply Proposition \ref{prop:asymptotics_for_the_2nd_lin} to \eqref{eq:2nd_ord_recovery_proof II} we see that
\[
0 =   \frac{1}{h^2}\int_\Sigma e^{4i\psi/h}\bar a K\left( \partial \Theta_1a,  \partial \Theta_2a\right) dV + o(1/h).
\]
Recall from \eqref{eq:Kcomplex} that for general functions $u$ and $v$ we have
\[
K(\p u, \p v):=(K^{11}-K^{22}+i(K^{12}+K^{21}))\p u \p v. 
\]
We apply stationary phase to this expression using the facts that $a(z_0) = 1$ and $a$ vanishes to high order at all other critical points. We get that, in local holomorphic coordinates,
\[
0=K(\p z, \p z)(z_0)=K^{11}(z_0) - K^{22}(z_0) + i(K^{12}(z_0) + K^{21}(z_0)).
\]
Since $K(\ccdot, \ccdot)$ is a symmetric $(2,0)$-tensor with real valued entries, and whose trace vanishes when taken with respect to $g$, we have that $K(z_0)= 0$. It was shown in \cite[Proposition 2.3.1]{guillarmou2011calderon} that critical points of holomorphic Morse functions form a dense subset of $\Sigma$.
Repeating the argument for a dense set of points $z_0$ of $\Sigma=\Sigma_1$ shows that
\begin{equation}\label{eq_k_determined}
 k_1^{(1)}= c\tilde k_2^{(1)} \text{ on } \Sigma
\end{equation}
as claimed.

Next we show that $h_1^{(2)}=c\tilde h_2^{(2)}$  on  $\Sigma_1$. 
By \eqref{eq_k_determined}, the identity \eqref{eq:2nd_ord_recovery_proof} now reads
\[
  0=\int_{\Sigma} (h^{(2)}_1-c\tilde h^{(2)}_2)v^1v^2v^3dV.
\]
We choose $v_1$, $v_2$ and $v_3$ as before in \eqref{eq:sols_for_second_proof}. 
Arguing similarly as before, we obtain by stationary phase and the $L^p$ estimates for $r_h$, $r_h'$ and $\tilde r_h$, see \eqref{eq:rh_L2_estim}--\eqref{eq:rh_Lr_estim}, that
 $h^{(2)}_1(z_0)=c(z_0)\tilde h^{(2)}_2(z_0)$. 
Repeating the argument for all points of $\Sigma$ yields
\[
 h^{(2)}_1= c\tilde h^{(2)}_2 \text{ on } \Sigma.
\]

Recall that the second linearizations $w_\beta^{jk}$, $\beta =1,2$,  
 satisfy 
 \begin{equation*}
 \begin{aligned}
		\begin{cases}
  (\Delta_{g_\beta}+h_\beta^{(1)}/2)w_\beta^{jk}+P_\beta^{(j}v_\beta^{k)}+k_\beta^{(1)}(\nabla v_\beta^{j},\nabla v_\beta^{k})+\frac{1}{2}h_\beta^{(2)}v_\beta^{j}v_\beta^{k}= 0& \text{ in } \Sigma \\
  w_\beta^{jk}=0
			&\text{ on }\p \Sigma,
  		\end{cases}
	\end{aligned}
 \end{equation*}
 where $P_\beta^j=-\nabla^{g_\beta}\cdot_\beta (v_\beta^jg_\beta k_\beta^{(1)}\nabla)$. Since we have recover all the coefficients of the above equation up to a conformal scaling factor, we see by uniqueness of solutions to elliptic equations that $w_1^{jk}=F^*w_2^{jk}=:\tilde w_2^{jk}$. All in all, we have shown by using the second linearization that
 \begin{align*}
  k_1^{(1)}=c\tilde k_2^{(1)} , \quad   h^{(2)}_1&=c\tilde h^{(2)}_2, \quad w_1^{jk}=\tilde w_2^{jk},
 \end{align*}
 for all $j,k=1,\ldots,4$.

\subsection{Step 3: Third linearization}
To complete the proof, we are left to show that the conformal factor $c=1$ on $\Sigma_1$. For this, we use the third linearization. By Lemma \ref{Lem:Integral identity_3rd}, the corresponding integral identity is 
	\begin{multline}\label{eq:third_integral_id_proof}
	 \int_{\p \Sigma} f_m \s \p^3_{\eps_j \eps_k\eps_l}\big|_{\epsilon=0} \Lambda (f_\epsilon) \, dS_g = \int_{\Sigma}g(\nabla v^{j},\nabla v^k)g(\nabla v^{l},\nabla v^m) dV  \\
	 +\int_{\Sigma}g(\nabla v^{j},\nabla v^l)g(\nabla v^{k},\nabla v^m) dV+\int_{\Sigma}g(\nabla v^{l},\nabla v^k)g(\nabla v^{j},\nabla v^m) dV \\
	 +H+R+B,
	\end{multline}
	where
	\begin{equation*}
 \begin{split}
 H&=-\int_{\Sigma}v^{(j}v^{k} k^{(2)}(\nabla v^{l)},\nabla v^m) dV +\int_{\Sigma}v^m g(\nabla (d^{-1}d^{(2)}v^{(j}v^k),\nabla v^{l)})dV \\
 &\qquad-\int_\Sigma v^m k^{(2)}(\nabla v^{(j},\nabla v^{k})v^{l)}dV-\frac{1}{2}\int_\Sigma v^mv^{j}v^{k}v^{l}h^{(3)}dV,\\
 R&=-\int_{\Sigma}w^{(jk}k^{(1)}(\nabla v^{l)},\nabla v^m)dV-\int_\Sigma k^{(1)}(\nabla v^m,\nabla v^{(j})w^{kl)}dV \\
 &\qquad-\int_{\Sigma}v^mk^{(1)}(\nabla v^{(j},\nabla w^{kl)})dV-\frac{1}{2}\int_\Sigma v^mg(\nabla v^{(j},\nabla v^{k})v^{l)}h^{(1)}dV  \\
 &\qquad \qquad -\frac{1}{2}\int_\Sigma v^mw^{(jk}v^{l)}h^{(2)}dV,   \\
 B&=\int_{\p \Sigma} v^mv^{(j}v^k g(\nu,k^{(2)}\nabla v^{l)})dS+\int_{\p \Sigma} v^mw^{(jk} g(\nu,k^{(1)}\nabla v^{l)})dS\\
&\qquad-\int_{\p \Sigma} v^mg(\nabla v^{(j},\nabla v^k)\p_\nu v^{l)}dS+\int_{\p \Sigma} v^mv^{(j} k^{(1)}(\nu,\nabla w^{kl)})dS.
\end{split}
\end{equation*}
The identity holds for any $j,k,m\in \{1,\ldots,4\}$ and for quantities on both $\Sigma=\Sigma_1$ and $\Sigma=\Sigma_2$. 

As already remarked after Lemma \ref{Lem:Integral identity_3rd} we will be able to disregard all the terms $H$, $R$ and $B$. Let us explain how. First of all, the terms of $B$ are all boundary terms, which we assume to be known. We have also shown that $v_1^{j}=\tilde v_2^{j}$ and $w_1^{jk}=\tilde w_2^{jk}$ on $\Sigma_1$. Especially, $v_1^{j}= v_2^{j}$ and $w_1^{jk}=w_2^{jk}$ to infinite order on the boundary by \eqref{eq:F_id_to_infinite}. Thus, when we subtract the integral identity \eqref{eq:third_integral_id_proof} on $\Sigma_1$ from that on $\Sigma_2$, the boundary terms $B$ will cancel. 

Next, we note that we have recovered all the quantities in $R$, up to a possible conformal factor. Let us see how the conformal factor behaves when we make change of variables in integration. Let us consider the simplest term
\[
 \int_{\Sigma_2} v_2^mw_2^{(jk}v_2^{l)}h_2^{(2)}dV_2
\]
in $R$ on $\Sigma_2$. This term equals
\[
 \int_{\Sigma_1} F^*(v_2^mw_2^{(jk}v_2^{l)}h_2^{(2)})F^*dV_2=\int_{\Sigma_1} v^mw^{(jk}v^{l)}c^{-1}h_1^{(2)}cdV_1=\int_{\Sigma_1} v^mw^{(jk}v^{l)}h_1^{(2)}dV_1.
\]
Thus the conformal factor $c$ cancels out when we make a change of variables. The other terms of $R$ behave the same way. Consequently, when we subtract the terms of $R$ on $\Sigma_1$ from those on $\Sigma_2$, they will cancel out.

Finally we note that while we have not recovered the terms of $H$, those terms have the property that they contain at most $2$ derivatives on $v^j$ and do not  contain any $w^{jk}$. 
We will later argue based on this property that the terms of $H$ are negligible as $h\to 0$. 

Let us then consider the first three terms on the right hand side of \eqref{eq:third_integral_id_proof}. These will be the leading order terms when we fix the solutions $v^j$ as CGOs. By making a change of variable in integration, we have 
\begin{multline*}
 \int_{\Sigma_2}g_2(\nabla v_2^{j},\nabla v_2^k)g_2(\nabla v_2^{l},\nabla v_2^m) dV_2=\int_{\Sigma_1}F^*\left(g_2(\nabla v_2^{j},\nabla v_2^k)g_2(\nabla v_2^{l},\nabla v_2^m)\right) F^*dV_2   \\
 =\int_{\Sigma_1}F^*g_2(\nabla v^{j},\nabla v^k)F^*g_2(\nabla v^{l},\nabla v^m)cdV_1=\int_{\Sigma_1}c^{-1}g_1(\nabla v^{j},\nabla v^k)g_1(\nabla v^{l},\nabla v^m)dV_1.
\end{multline*}
Here we used that $F^*g_2=cg_1$ when both sides are acting to vector fields (and so, $F^*g_2(\omega_1,\omega_2)=c^{-1}g_1(F^*\omega_1,F^*\omega_2)$ for 1-forms $\omega_1$ and $\omega_2$ as 
above)
and that $v^j=F^*v_2^j$. Thus, by what we have argued, subtracting the identity \eqref{eq:third_integral_id_proof} on the $\Sigma_1$ from that on $\Sigma_2$ and using $\Lambda_{g_1} (f_\epsilon) =\Lambda_{g_2} (f_\epsilon)$ we have
\begin{multline}\label{eq:withH1H2}
H_1-H_2=\int_{\Sigma_1}(1-c^{-1}) \Big[g_1(\nabla v^1\cdot \nabla v^2)g_1(\nabla v^3\cdot \nabla  v^4)+g_1(\nabla v^1\cdot \nabla v^3)g_1(\nabla v^2\cdot \nabla  v^4) \\
 +g_1(\nabla v^1\cdot \nabla v^4)g_1(\nabla v^2\cdot \nabla  v^3)\Big]dV_1.
 \end{multline}
 Here we also chose $(j,k,l,m)=(1,2,3,4)$. 
 
 Let $z_0\in \Sigma$. We choose the four solutions of the equation \eqref{linear_eq} of the form \eqref{eq:sols_for_4th linearization},
 \[
    v^1 =e^{\Phi_1/h}(a+r_h),\quad \ \, v^3 = e^{\Phi_3/h}(a+ r'_h), \quad v^2  = e^{-\overline \Phi_2/h}(\overline a+\tilde r_h),\quad v^4= e^{-\overline \Phi_4/h}(\overline a+ \tilde r_h'),
   \]
 where $z_0\in \Sigma$ is a critical point of $\Phi$. Recall that the correction terms of these solutions satisfy the better estimates \eqref{eq: no critical point estimate}.
%
The holomorphic function $a$ is chosen such that it has the expansion $a(z)=1+O(z^N)$ near the critical point $z_0$ and vanishes to high order at all other possible critical points of $\Phi$. As $g_1=g_2$ to infinite order on $\p \Sigma$ and $F^*g_2=g_1$, we have that $c^{-1}$ is $1$ to infinite order on the boundary. Consequently, by Proposition \ref{prop:asymptotics_for_the_3rd_lin}, we have
\begin{multline}\label{eq:expansion_for_3rd_lin_proof}
\int_{\Sigma}(1-c^{-1})\Big[g(\nabla v_1\cdot \nabla v_2)g(\nabla v_3\cdot \nabla  v_4)+g(\nabla v_1\cdot \nabla v_3)g(\nabla v_2\cdot \nabla  v_4) \\
 +g(\nabla v_2\cdot \nabla v_3)g(\nabla v_1\cdot \nabla  v_4)\Big]dV=h^{-1}(1-c^{-1}(z_0))+o(h^{-1}).
 \end{multline}

Let us complete the proof by arguing that the terms $H_1$ and $H_2$ are negligible. Indeed, as they contain only $v^j$ and up to $2$ derivatives of them, and everything else in these terms is independent of $h$, these terms are of size $O(h^{-2})$ by stationary phase and the estimates for  the correction terms.  
Combining this fact with \eqref{eq:withH1H2} and \eqref{eq:expansion_for_3rd_lin_proof} shows that
\[
 h^{-3}(1-c^{-1}(z_0))+o(h^{-3})=0
\]
as $h\to 0$. Thus $c(z_0)=1$. Repeating this argument for a dense set of points $z_0\in \Sigma_1$ shows that
\[
 F^*g_2=g_1 \text{ on } \Sigma_1.
\]
Finally noting that the scalar second fundamental form reads $\eta_{ab}(x)=\p_s|_{s=0}g_{ab}(x,s)$ in Fermi-coordinates (see e.g. \cite{lassas2016calder} for the formula), knowing $k^{(1)}$ and $g$ determines $\eta$. Thus also $F^*\eta_2=\eta_1$. This concludes the proof of Theorem \ref{thm:main}.

\appendix

\addtocontents{toc}{\protect\setcounter{tocdepth}{1}}
\section{Calculations and boundary determination for isometries}

\subsection{Calculations}\label{appx:calculations}
We record computations used in Section \ref{Section 2}. For the below calculations, we note that
\[
 \frac{d^2}{d\eps_j\eps_k}\Big|_{\eps=0}(1+\abs{\nabla u}^2_u)^{-1/2}=-g(\nabla v_j,\nabla v_k).
\]

The calculation behind the equation \eqref{eq:third_deriv_of_f} is the following
\begin{multline*}
 \frac{d^3}{d\eps_{j}d\eps_kd\eps_l}\Big|_{\eps=0}f(u,\nabla u) \\
 =\frac{1}{2}\frac{1}{(1+\abs{\nabla u}^2_{u})^{1/2}}|_{\eps=0}\left(\p_{\eps_{jkl}}|_{\eps=0}k_u^{(1)}(\nabla u,\nabla u)\right)+\frac{1}{2}\p_{\eps_{jkl}}|_{\eps=0}\left((1+\abs{\nabla u}^2_{u})^{1/2}h_u\right) \\
 =\frac{1}{2}\p_{\eps_{jk}}|_{\eps=0}\left[\left(k_u^{(2)}(\nabla u,\nabla u)u_l+2k_u^{(1)}(\nabla u_l,\nabla u)\right)\right] \\
 +\frac{1}{2}\p_{\eps_{jk}}|_{\eps=0}\left[(1+\abs{\nabla u}^2_{u})^{-1/2}g_u(\nabla u_l,\nabla u)h_u+(1+\abs{\nabla u}^2_{u})^{1/2}h_u^{(1)}u_l \right] \\
 =\frac{1}{2}\p_{\eps_{j}}|_{\eps=0}\Big[2k_u^{(2)}(\nabla u_k,\nabla u)u_l+2k_u^{(2)}(\nabla u_l,\nabla u)u_k+2k_u^{(1)}(\nabla u_{kl},\nabla u)+2k_u^{(1)}(\nabla u_{l},\nabla u_k)\Big] \\
 +\frac{1}{2}\p_{\eps_{j}}|_{\eps=0}\Big[g_u(\nabla u_l,\nabla u_k)h_u+g_u(\nabla u_l,\nabla u)h_u^{(1)}u_k \\ 
 \qquad\qquad\qquad\qquad\qquad\qquad\qquad\qquad\qquad\qquad +g_u(\nabla u_k,\nabla u)h_u^{(1)}u_l+h_u^{(2)}u_lu_k+h_u^{(1)}u_{kl} \Big] \\
 =k^{(2)}(\nabla v^{(j},\nabla v^k)v^{l)}+k^{(1)}(\nabla w^{(jk},\nabla v^{l)})+\frac{1}{2}g(\nabla v^{(j},\nabla v^k)v^{l)}h^{(1)} \\
 +\frac{1}{2}w^{(jk}v^{l)}h^{(2)}+\frac{1}{2}v^{j}v^{k}v^{l}h^{(3)}+\frac{1}{2}w^{jkl}h^{(1)}.
\end{multline*}
In the above calculation $u|_{\eps=0}\equiv 0$ and $h=0$ was used several times.

The calculation behind the formula for \eqref{eq:formula_for_Pjk} is the following
\begin{multline*}
 P^{jk}
 =-\frac{d^2}{d\eps_jd\eps_k}\Big|_{\eps=0}d_u^{-1}\nabla\cdot k_ud_u(1+\abs{\nabla u}^2_u)^{-1/2}\nabla  \\
 =\frac{d}{d\eps_l}\Big|_{\eps=0}\Big[d^{-2}_u d^{(1)}_u u_k\nabla\cdot k_ud_u (1+\abs{\nabla u}^2_u)^{-1/2} \nabla-d^{-1}_u\nabla\cdot k_u^{(1)}u_kd_u (1+\abs{\nabla u}^2_u)^{-1/2} \nabla\\
 -d^{-1}_u\nabla\cdot k_u d_u^{(1)}u_k (1+\abs{\nabla u}^2_u)^{-1/2} \nabla -d^{-1}_u\nabla\cdot k_u d_u \frac{d}{d\eps_k}(1+\abs{\nabla u}^2_u)^{-1/2} \nabla\Big]  \\
 =d^{-2} d^{(2)} v^k v^l\nabla\cdot kd \nabla-d^{-1}\nabla\cdot k^{(2)}v^kv^ld \nabla-d^{-1}\nabla\cdot k^{(1)}w^{kl}d \nabla-d^{-1}\nabla\cdot k d^{(2)}v^kv^l
 \nabla  \\
 +d^{-1}\nabla\cdot kd g(\nabla v^k,\nabla v^l)\nabla \\
 =-k\nabla (d^{-1}d^{(2)}v^k v^l)\cdot \nabla-d^{-1}\nabla\cdot (k^{(2)}v^kv^ld \nabla) \\
 -d^{-1}\nabla\cdot (k^{(1)}w^{kl}d \nabla)
 +d^{-1}\nabla\cdot (kd g(\nabla v^k,\nabla v^l)\nabla).
\end{multline*}

\subsection{Boundary determination for the isometry}\label{appx:proof_of_lemma} 
\begin{proof}[Proof of Lemma \ref{lem:F_id_infty}]
The proof is similar to that of \cite[Lemma 3.4]{lassas2016calder}. 
 Let $p\in \p M$, let us fix coordinates on a neighborhood $\Gamma$ of $p$ in $\p \Sigma$ and let us consider the associated boundary normal coordinates on $\Sigma_1$ and $\Sigma_2$.  
Let $U_\beta$, $\beta=1,2$, be solutions to $L_\beta U_\beta=0$ with $U_\beta|_{\p \Sigma}=f$.
 Below we denote by $F$, $U_1$ and $U_2$ their coordinate representation in the boundary normal coordinates. 
 
 Since $F$ the preserves boundary data of solutions $U_1$ and $U_2$ by assumption, we have
 \[
  F|_{\p M}=\text{Id}_{\p M}.
 \]
 Since $\Lambda_{L_1}=\Lambda_{L_2}$, we have
 \[
  \p_{x_n}U_1=\p_{x_n}U_2 \text{ on } \{x_n=0\}.
 \]
 It follows that $\nabla U_1=\nabla U_2$ on $\{x_n=0\}$. Here $\nabla$ is the gradient in $\R^n$. Combining this with the assumptions that $F$ is a morphism of solutions, we then have
 \[
  \nabla U_1=\nabla U_2=\nabla(U_1\circ F)=DF^T\nabla U_1|_{\{x_n = 0\}}
 \]
 for all solutions $U_1$ to $L_1U_1=0$ in $\Sigma_1$.
 By Runge approximation (see e.g. \cite[Proposition B.1]{lassas2016calder}) and by using linearity of $L_1$, it is possible to choose $\nabla U_1(p)$ over a dense subset of $\mathbb R^n$. Thus 
 \[
  DF(p)=I_{n\times n}.
 \]
 Repeating the argument for all $p\in \Gamma$ shows that $DF|_{x_n=0}(p)=I_{n\times n}$ on $\{x_n=0\}$.
 
 Following the proof of \cite[Lemma 3.4]{lassas2016calder}, we may write, for $\beta =1,2$
 \[
  \Delta_{g_\beta}+q_\beta=-\p_{x_n}^2+P_\beta,
 \]
 where $P_\beta$ is a partial differential operator containing only first order derivatives in $x_n$ and $x'$ derivatives up to second order and whose coefficients are polynomial expressions of $g_\beta$ and $q_\beta$. Since
 \[
  0=(\Delta_{g_\beta}+q_\beta)U_\beta=-\p_{x_n}^2U_\beta+P_\beta U_\beta,
 \]
 we have by the condition $\nabla U_1=\nabla U_2$ on $\{x_n=0\}$ and the assumption that the formal Taylor series of $g_\beta$ and $q_\beta$ agree to infinite order at $\{x_n=0\}$ in boundary normal coordinates that
 \begin{equation}\label{eq:op_decomp}
  \p_{x_n}^2U_1=P_1 U_1=P_2 U_2=\p_{x_n}^2U_2.
 \end{equation}
 Consequently, we have $\nabla^2 U_1=\nabla^2 U_2$ on $\{x_n=0\}$. 
Then, by using also $DF|_{x_n=0}=I_{n\times n}$ we have
\begin{multline*}
 \p_{x_n}^2U_1=\p_{x_n}^2(U\circ F)= \p_{ab}^2U_2|_F\p_{x_n}F^a\p_nF^b+\p_aU_2|_F\p_{x_n}^2F^a \\
 = \p_{ab}^2U_1\p_{x_n}F^a\p_nF^b+\p_aU_2\p_{x_n}^2F^a=\p_{x_n}^2U_1+\p_aU_1\p_{x_n}^2F^a.
\end{multline*}
We have employed the Einstein summation convention in the above equation.
Thus
\[
 \p_aU_1\p_{x_n}^2F^a=0 \text{ on } \{x_n=0\}.
\]
As we may choose $\nabla U_1(p)$ freely in a dense subset of $\mathbb R^n$, we have
 \[
  \p_{x_n}^2F^a(p)=0 \text{ on } \{x_n=0\},
 \]
 for all $a=1,\ldots,n$. Repeating the argument for all $p\in \Gamma$ shows that $\p_{x_n}^2F^a=0$ on $\{x_n=0\}$.
 
The proof is completed by differentiating \eqref{eq:op_decomp} in $x_n$ and by using an induction argument similar to that in \cite[Lemma 3.4]{lassas2016calder}.
 
%
%
%
%

\end{proof}

%
%
%

\bibliography{ref, ref_minimal_surface, ref_intro}
\bibliographystyle{abbrv}

\end{document}